\documentclass{amsart}
\usepackage{graphics,verbatim, amsmath, amssymb, amsthm, amsfonts,mathtools,longtable}	
\usepackage{ifthen}
\newtheorem{proposition}{Proposition}[section]
\newtheorem{theorem}[proposition]{Theorem}

\newtheorem{lemma}[proposition]{Lemma}
\newtheorem{prop}[proposition]{Proposition}
\newtheorem{cor}[proposition]{Corollary}
\newtheorem{conj}[proposition]{Conjecture}

\theoremstyle{definition}
\newtheorem{example}[proposition]{Example}

\theoremstyle{remark}
\newtheorem{remark}[proposition]{Remark}

\numberwithin{equation}{section}

\usepackage{caption}

\usepackage[bookmarks=true, bookmarksopen=false,%
    colorlinks=true,%
    linkcolor=darkblue,%
    citecolor=darkblue,%
    filecolor=darkblue,%
    menucolor=darkblue,%
    linktoc=all,%
    urlcolor=darkblue
]{hyperref}
\usepackage[usenames]{xcolor}
\definecolor{darkblue}{cmyk}{1,0.3,0,0.1}  %blue

\newcounter{margincounter}
\newcommand{\newword}[1]{\textbf{\emph{#1}}}

\newcommand{\integers}{\mathbb Z}
\newcommand{\reals}{\mathbb R}

\newcommand{\from}{\leftarrow}

\newcommand{\ep}{\varepsilon}
\newcommand{\cl}{\operatorname{cl}}

\newcommand{\Seed}{\operatorname{Seed}}
\newcommand{\Camb}{\operatorname{Camb}}

\newcommand{\sgn}{\operatorname{sgn}}
\renewcommand{\int}{\operatorname{int}}

\newcommand{\cov}{\mathrm{cov}}

\newcommand{\covered}{{\,\,<\!\!\!\!\cdot\,\,\,}}
\newcommand{\set}[1]{{\lbrace #1 \rbrace}}

\newcommand{\pidown}{\pi_\downarrow}
\newcommand{\piup}{\pi^\uparrow}
\newcommand{\br}[1]{{\langle #1 \rangle}}
\newcommand{\e}{{\mathbf e}}
\newcommand{\g}{{\mathbf g}}

\renewcommand{\d}{{\mathbf d}}

\renewcommand{\c}{{\mathbf c}}
\renewcommand{\b}{{\mathbf b}}

\newcommand{\A}{{\mathcal A}}

\newcommand{\C}{{\mathcal C}}

\newcommand{\tB}{{\widetilde{B}}}
\newcommand{\hG}{{\hat{G}}}
\newcommand{\hC}{{\hat{C}}}

\newcommand{\F}{{\mathcal F}}
\newcommand{\X}{{\mathcal X}}

\newcommand{\join}{\vee}

\renewcommand{\Join}{\bigvee}
\newcommand{\Meet}{\bigwedge}

\newcommand{\ck}{^{\vee\!}}

\renewcommand{\th}{^\mathrm{th}}

\newcommand{\FF}{\mathbb{F}}
\newcommand{\ZZ}{\mathbb{Z}}
\newcommand{\QQ}{\mathbb{Q}}
\newcommand{\RR}{\mathbb{R}}
\newcommand{\PP}{\mathbb{P}}

\newcommand{\Cone}{\mathrm{Cone}}

\newcommand{\Ex}{\mathrm{Ex}}
\newcommand{\Trop}{\mathrm{Trop}}

\DeclareMathOperator{\Span}{Span}

\DeclareMathOperator{\inv}{inv}

\DeclareMathOperator{\Cart}{Cart}
\newcommand{\twomatrix}[2]{\begin{psmallmatrix} #1 \\ #2 \end{psmallmatrix}}

\begin{document}

\title{Combinatorial frameworks for cluster algebras}
\author{Nathan Reading and David E Speyer}
\address[Nathan Reading]{\,Department of Mathematics, North Carolina State University, Raleigh, NC, 27695 USA}
\email{reading@math.ncsu.edu}
\address[David E Speyer]{Department of Mathematics, University of Michigan, Ann Arbor, MI, 48109 USA}
\email{speyer@umich.edu}
\thanks{Nathan Reading was partially supported by NSA grant H98230-09-1-0056, by Simons Foundation grant \#209288, and by NSF grant DMS-1101568.   David E Speyer was supported in part by a Clay Research Fellowship}

\begin{abstract}
We develop a general approach to finding combinatorial models for cluster algebras.
The approach is to construct a labeled graph called a \newword{framework}.
When a framework is constructed with certain properties, the result is a model incorporating information about exchange matrices, principal coefficients, $\g$-vectors, and $\g$-vector fans.
The idea behind frameworks arises from Cambrian combinatorics and sortable elements, and in this paper, we use sortable elements to construct a framework for any cluster algebra with an acyclic initial exchange matrix.
This Cambrian framework yields a model of the entire exchange graph when the cluster algebra is of finite type.
Outside of finite type, the Cambrian framework models only part of the exchange graph.
In a forthcoming paper, we extend the Cambrian construction to produce a complete framework for a cluster algebra whose associated Cartan matrix is of affine type.
\end{abstract}

\maketitle

%\vspace{-12pt}

\setcounter{tocdepth}{2}
\tableofcontents

\section{Introduction}\label{intro sec} 
Cluster algebras were introduced in~\cite{ca1} as a tool for studying total positivity and canonical bases in semisimple algebraic groups.
They have since appeared in various fields, including Teichm\"{u}ller theory, Poisson geometry, quiver representations, Lie theory, algebraic geometry and algebraic combinatorics. 
Among the key open problems surrounding cluster algebras are various structural conjectures found in \cite{FZ-CDM,ca4}. 
Many of these conjectures are proved in important special cases (e.g. finite type or skew-symmetric \cite{CK2,Demonet,QP2,ca2,ca4,Fu-Keller,IIKNS,Nagao,Pla1,Pla2,YZ}), but the general (infinite type skew-symmetrizable) cases remain mostly open.
See Section~\ref{conj sec} for statements of many of these conjectures and remarks on their status.

This paper continues a series of papers aimed at creating combinatorial models for cluster algebras within the context of Coxeter groups and root systems, using in particular the machinery of sortable elements and Cambrian (semi)lattices.
Cambrian lattices were introduced in~\cite{cambrian} as certain lattice quotients (or alternately sublattices) of the weak order on a finite Coxeter group.
Conjectures in that paper, later proved in \cite{HLT,sortable,sort_camb,camb_fan}, established the relevance of Cambrian lattices to Coxeter-Catalan combinatorics.
In particular, Cambrian lattices are closely related to generalized associahedra, which serve as combinatorial models for cluster algebras of finite type \cite{ga,ca2}. 
(See e.g.~\cite{Armstrong,rsga} for an introduction to Coxeter-Catalan combinatorics.)
In the process of proving the conjectures, sortable elements were defined in \cite{sortable} and shown in \cite{sort_camb} to provide a direct combinatorial realization of Cambrian lattices.
Finally, in~\cite{typefree}, the combinatorial theory surrounding sortable elements was further developed and extended to arbitrary Coxeter groups (from the special case of finite Coxeter groups).

In this paper, we construct Cambrian combinatorial models for cluster algebras.
The insights gained from our first direct constructions of Cambrian models led to a general blueprint for building combinatorial/algebraic models.
Accordingly, we begin the paper by defining the notion of a \newword{framework} for an exchange matrix $B$.
The matrix $B$ defines a Cartan companion $\Cart(B)$ and thus a root system.
The matrix $B$ also defines a bilinear form $\omega$ on the root space.
In essence, a framework is a graph $G$ with two labelings:  each pair $(v,e)$ consisting of a vertex contained in an edge is given a \newword{label} $C(v,e)$ and a \newword{co-label} $C\ck(v,e)$.  
These labels are vectors in the root space that are related by a positive scaling.
They should be thought of as a root and its corresponding co-root, and indeed, in important special cases, this is the case.
The labeling/co-labeling must satisfy certain conditions;  when it does, a detailed combinatorial model of the cluster algebra can be extracted from the framework.

The following theorem summarizes how to recover combinatorial properties of a cluster algebra from a framework for it. The precise details are Theorems~\ref{framework exchange} and~\ref{framework principal}.
\begin{theorem} \label{framework summary} 
Let $B$ be a skew-symmetrizable integer matrix. 
Let $\mathcal{A}$ be the cluster algebra whose initial seed $t_0$ has exchange matrix $B$ and principal coefficients.
Let $(G, C, C\ck)$ be a framework for $B$.
There is a map from vertices $v$ of $G$ to seeds $t$ such that
\begin{itemize}
\item There is a base vertex $v_b$ mapping to $t_0$.
\item Edges in $G$ correspond to mutations of seeds.
\item The exchange matrix of $t$ has entries $[\omega(C\ck(v,e),C(v,f))]_{e,f\ni v}$.
\item The $\c$-vectors of cluster variables in $t$ are the simple root coordinates of the vectors $C(v,e)$.
\item The $\g$-vectors of cluster variables in $t$, with respect to the seed $t_0$, are the fundamental-weight coordinates of the basis of the weight space that is dual to the basis $\set{C\ck(v,e)}_{e \ni v}$
\end{itemize}
\end{theorem}

The framework may also contain information about denominator vectors.
For acyclic $B$, we define an explicit linear map with an easily computed inverse and conjecture (as Conjecture~\ref{nu conj}) that the map relates denominator vectors to $\g$-vectors for cluster variables outside the initial seed.
The conjecture is easily verified when $B$ is $2\times 2$, and we prove the conjecture when $\Cart(B)$ is of finite type. 

A framework is \newword{complete} \label{complete} if $G$ is a regular graph of the correct degree, and a complete framework models the entire exchange graph of the cluster algebra.
We define some other conditions on frameworks in Section~\ref{global sec}, including the notions of an \newword{exact} framework and a \newword{well-connected polyhedral} framework.
When the framework is complete and exact, the graph $G$ is isomorphic to the exchange graph.
When the framework is also polyhedral and well-connected, it defines a fan that coincides with the fan of $\g$-vectors. 
The existence of a complete, exact and/or well-connected polyhedral framework for $B$ implies several of the conjectures from \cite{FZ-CDM,ca4}.
See Theorem~\ref{complete conj} and Corollaries~\ref{exact conj} and~\ref{polyhedral conj} for details, and see Section~\ref{sec:defns} for a table of our definitions.

Frameworks are not only a convenient way to create models of cluster algebras, but in fact they are, conjecturally, the \emph{only} way to create models.
This idea is made precise in Theorem~\ref{T model} and Corollary~\ref{identity model}, which state that, assuming some conjectures from \cite{ca4}, every cluster algebra defines a framework.  
The theorem and corollary, together with the recipes in Theorem~\ref{framework summary} for reading off combinatorial information from a framework imply that every model for principal coefficients and/or $\g$-vectors and/or (conjecturally) denominator vectors \emph{is} a framework.

Comparing the method of determining $\g$-vectors from a framework to the method of determining principal coefficients, we arrive (Corollary~\ref{NZ cor}) at an insight that also recently appeared as the first assertion of \cite[Theorem~1.2]{NZ}:  
Assuming a certain conjecture from~\cite{ca4}, the matrix whose rows are the $\g$-vectors associated to $B$ is the inverse of the matrix whose columns are the $\c$-vectors associated to $-B^T$.

The main thrust of the project, however, is to give direct proofs of the structural conjectures by constructing explicit frameworks based on sortable elements and Cambrian lattices, without relying on results from other approaches to cluster algebras.
In a sense we can do this quite generally:
For every acyclic $B$, we construct a \newword{Cambrian framework} and prove that the framework is exact, polyhedral, and well-connected.
The underlying graph is the Cambrian (semi)lattice, so that the vertices of the graph are the sortable elements.
The labels are certain roots that can be read off combinatorially from the sortable elements, and the co-labels are the associated co-roots.
In the Cambrian framework, the $\g$-vectors can also be read off combinatorially without having to compute a dual basis.
(See Theorem~\ref{camb nu g}.)
Conjecturally, the denominator vectors can also be read off combinatorially, and this conjecture is already proven in the case where $\Cart(B)$ is of finite type.

When $\Cart(B)$ is of finite type, the Cambrian framework is complete, and thus defines a combinatorial and polyhedral model for the entire exchange graph.
Furthermore, the existence and properties of the Cambrian framework, when $\Cart(B)$ is finite type, imply \cite[Conjecture~4.7]{ca4} as well as Conjecture~\ref{nu conj}, described above. 
The Cambrian framework also provides new proofs of many of the conjectures from \cite{FZ-CDM,ca4} when $\Cart(B)$ is of finite type.
See Theorem~\ref{finite type conj} for details.

When $\Cart(B)$ is of infinite type, the Cambrian framework is not complete.
The incompleteness of the Cambrian framework is the result of a fundamental obstacle:  every cone in the Cambrian fan intersects the Tits cone of the Coxeter group associated to $B$.
But cones in the (conjectural) fan defined by $\g$-vectors do not all intersect the Tits cone.
In a forthcoming paper, the authors construct a complete framework for exchange matrices $B$ such that $\Cart(B)$ is of affine type and use this framework to prove many of the conjectures from \cite{FZ-CDM,ca4} for such $B$. 
In another paper \cite{ST}, the second author and Hugh Thomas construct a complete reflection framework for an arbitrary acyclic exchange matrix $B$.
(See Theorem~\ref{ST frame}.)
Their framework is constructed from the representation theory of quivers.

We conclude this introduction with two remarks about the key features of frameworks.
As indicated above, the labels on a framework are essentially the 
\mbox{$\c$-vectors}.
That is, the labels are the columns of the bottom halves of principal-coefficients extended exchange matrices.
Thus, passing from vertex to vertex, the labels need to change in a way that amounts to matrix mutation.
The mutation relation is given in \cite[Equation~(5.9)]{ca4}, assuming one of the conjectures from \cite{ca4}.
The relation depends, of course, on the top halves of the extended exchange matrices, i.e.\ the exchange matrices.
What is missing from \cite[Equation~(5.9)]{ca4} is the insight that the exchange matrices themselves are determined from the labels, as described  in Theorem~\ref{framework summary}.
Thus \cite[Equation~(5.9)]{ca4} is replaced by a self-contained mutation rule in terms only of the labels.
This rule is called the \newword{Transition condition}, and is the most important feature of the notion of a framework.
In many important cases, all of the labels are real roots, the co-labels are the corresponding co-roots, and the Transition condition can be replaced by the \newword{Reflection condition}.
This condition says when $u$ and $v$ are connected by an edge and $t$ is the reflection associated to $C(v,e)$, each label on $u$ is either identical to a label on $v$, or is obtained from a label on $v$ by the action of $t$.
The Reflection condition was discovered in connection with combinatorial/polyhedral investigations of Cambrian fans.

Another technical feature of frameworks that should not be overlooked is the use of root and co-root lattices (and thus weight and co-weight lattices) to handle the possible absence of skew-symmetry in $B$.
Simply by placing vectors in the correct lattice (i.e.\ deciding whether to write the prefix ``co''), the difficulties caused by asymmetric $B$ completely disappear from the general theory of frameworks.
This is exactly analogous to the purpose of roots and co-roots, etc.\ in handling asymmetric Cartan matrices.
Given the fact that a Cambrian model exists for every acyclic $B$, based on the Cartan companion $\Cart(B)$, this analogy is hardly accidental.

\bigskip 

\noindent
\textbf{Acknowledgments.}
We thank Bernhard Keller for helpful comments on an earlier version.
We thank Salvatore Stella for sharing important insights.
We thank the referees for their careful reading and for helpful expository and mathematical suggestions.  

\section{Frameworks and reflection frameworks}\label{frame sec}
In this section, we define the general notion of a framework and a special kind of framework called a reflection framework.
This section contains a great number of definitions. We collect our definitions (from this section and others) in a table in Section~\ref{sec:defns}.

\subsection{Frameworks}\label{frame subsec}
The exchange matrix $B=[b_{ij}]$ associated to a cluster algebra is a skew-symmetrizable integer matrix, with rows and columns indexed by a set~$I$, with $|I|=n$.
That means that there exists a positive real-valued function~$\delta$ on $I$ such that $\delta(i) b_{ij}=-\delta(j) b_{ji}$ for all $i,j\in I$.
Let $\Cart(B)$ be the matrix with diagonal entries $2$ and off-diagonal entries $a_{ij}=-|b_{ij}|$. \label{defn:Cart}
Then $\Cart(B)$ is a symmetrizable generalized Cartan matrix in the sense of \cite{Kac} (see also \cite[Section~2.2]{typefree}), called the \newword{Cartan companion} of $B$.
In particular, $\delta(i) a_{ij}=\delta(j) a_{ji}$ for all $i$, $j\in I$.

Let $V$ be a real vector space of dimension $n$ with a basis $\Pi=\set{\alpha_i:i\in I}$ and let $V^*$ be the dual vector space.
The set $\Pi$ is called the set of \newword{simple roots}. 
The canonical pairing between $x\in V^*$ and $y\in V$ is written $\br{x,y}$. 
We set $\alpha_i\ck= \delta(i)^{-1} \alpha_i$.
The vectors $\alpha_i\ck$ are called the \newword{simple co-roots} and the set of simple co-roots is written $\Pi\ck$.
We write $D$ for the \newword{fundamental domain}, $\bigcap_{i \in I} \set{x\in V^*: \br{x,\alpha_i}\ge 0}$. 

The exchange matrix $B$ defines a bilinear form $\omega$ by setting $\omega(\alpha_{i}\ck,\alpha_{j})=b_{ij}$. 
The form $\omega$ is skew-symmetric, because 
\[\omega(\alpha_{i}\ck,\alpha_{j})=b_{ij}=-\frac{\delta(j)}{\delta(i)}b_{ji}=-\frac{\delta(j)}{\delta(i)}\omega(\alpha_{j}\ck,\alpha_i)=-\omega(\alpha_{j},\alpha_i\ck).\]

A \newword{quasi-graph} \label{quasi-graph} is a collection of vertices, full edges, and half-edges.
The full edges are edges in the usual graph-theoretical sense, so that each edge is incident to two vertices.
Each half-edge should be thought of as dangling from a vertex, and thus not connecting that vertex to any other.
(In the usual graph-theoretic terminology, a quasi-graph is a hypergraph with edges of size $1$ or $2$.)
We will usually drop the adjective ``full'' in referring to full edges.
Half-edges should not be confused with ``self-edges'' or ``loops,'' i.e.\ edges that connect a vertex to itself.
In this paper, all quasi-graphs will be \newword{simple}, meaning that no two full edges connect the same pair of vertices, and that every full edge connects two distinct vertices.
The \newword{degree} of a vertex in a simple quasi-graph is the total number of edges (including half-edges) incident to that vertex, and the quasi-graph is \newword{regular of degree $n$} if every vertex has degree~$n$.
A quasi-graph $G$ is \newword{connected} if the graph obtained from $G$ by ignoring half-edges is connected in the usual sense.

An \newword{incident pair} \label{incident pair} in a quasi-graph $G$ is a pair $(v,e)$ where $v$ is a vertex contained in an edge $e$.
For each vertex $v$, let $I(v)$ denote the set of edges $e$ such that $(v,e)$ is an incident pair.
A \newword{framework for $B$} will be a triple $(G,C,C\ck)$, where $G$ is a connected quasi-graph that is regular of degree $n$ (where $B$ is $n\times n$) and each of $C$ and $C\ck$ is a labeling of each incident pair in $G$ by a vector in $V$, satisfying certain conditions given below.
The \newword{label} \label{label} on $(v,e)$ will be written $C(v,e)$, and the notation $C(v)$ will stand for the set $\set{C(v,e):e\in I(v)}$.
The \newword{co-label} on $(v,e)$ will be written $C\ck(v,e)$, and the notation $C\ck(v)$ will stand for the set $\set{C\ck(v,e):e\in I(v)}$.\\

\noindent 
\textbf{Co-label condition:}  \label{Co-label condition} For each incident pair $(v,e)$, the co-label $C\ck(v,e)$ is a positive scalar multiple of the label $C(v,e)$.\\

We will see below that in an important special case, each label $C(v,e)$ will be a real root, and each co-label $C\ck(v,e)$ will be the associated co-root $(C(v,e))\ck$.
In general the label $C(v,e)$ need not be a real root, so there is no available notion of a co-root associated to $C(v,e)$.
Despite the fact that $C\ck(v,e)$ may not be a co-root in any meaningful sense, when $\beta=C(v,e)$, we will write $\beta\ck$ for $C\ck(v,e)$.\\

\noindent 
\textbf{Sign condition:} \label{Sign condition} For each incident pair $(v,e)$:
\begin{enumerate}
\item The label $C(v,e)$ is not the zero vector; and
\item Either $C(v,e)$ or $-C(v,e)$ is in the nonnegative span of the simple roots.\\
\end{enumerate}

Assuming the Sign condition, each label has a well-defined \newword{sign} $\sgn(C(v,e))\in\set{\pm 1}$, namely $\sgn(C(v,e))=1$ if $C(v,e)$ is in the nonnegative span of the simple roots, or $\sgn(C(v,e))=-1$ if $-C(v,e)$ is in the nonnegative span of the simple roots. 
Assuming the Co-label condition, the analogous definition for co-labels yields $\sgn(C\ck(v,e))=\sgn(C(v,e))$.  
\\

\noindent
\textbf{Base condition:}  \label{Base condition}
There exists a vertex $v_b$ such that $C(v_b)$ is the set of simple roots and $C\ck(v_b)$ is the set of simple co-roots.\\

The Base condition lets us identify the indexing set $I$ with $I(v_b)$ by identifying $e\in I(v_b)$ with the index $i\in I$ such that $C(v_b,e)=\alpha_i$.

We use the notations $[x]_{+} =\max(x,0)$ and $[x]_{-} = \min(x,0)$.   
\\  

\noindent
\textbf{Transition condition:} \label{Transition condition}
Suppose $v$ and $v'$ are distinct vertices incident to the same edge $e$.
Then $C(v,e)=-C(v',e)$.  
Furthermore, if $\beta=C(v,e)$,  $\beta\ck=C\ck(v,e)$,  and $\gamma\in C(v)\setminus\set{\beta}$, then $\gamma+[\sgn(\beta)\omega(\beta\ck,\gamma)]_+\,\beta$ is in $C(v')$.\\

\noindent
\textbf{Co-transition condition:} \label{Co-transition condition}
Suppose $v$ and $v'$ are distinct vertices incident to the same edge $e$.
Then $C\ck(v,e)=-C\ck(v',e)$.  
Furthermore, if $\beta\ck=C\ck(v,e)$,  $\beta=C(v,e)$,  and $\gamma\ck\in C\ck(v)\setminus\set{\beta\ck}$, then $\gamma\ck+[-\sgn(\beta\ck)\omega(\gamma\ck,\beta)]_+\,\beta\ck$ is in $C\ck(v')$.\\

The triple $(G,C,C\ck)$ will be called a \newword{framework for $B$} \label{framework} if $G$ is connected and if the triple satisfies the 
Co-label,
Sign,
Base, 
Transition,
and
Co-transition
conditions.

\begin{example}\label{B2 example}
We now construct a framework for the non-skew-symmetric exchange matrix $B=\begin{psmallmatrix*}[r]0&2\\-1&0\end{psmallmatrix*}$. 
This is skew-symmetrizable with $\delta(2)=2\delta(1)$.
The Cartan companion of $B$ is $\Cart(B)=\begin{psmallmatrix*}[r]2&-2\\-1&2\end{psmallmatrix*}$.
The skew-symmetric form $\omega$ is given by $\omega(\alpha_1\ck,\alpha_2)=2$ and $\omega(\alpha_2\ck,\alpha_1)=-1$.
A framework for $B$ is shown in Figure~\ref{B2 frame fig}. 
(In Section~\ref{camb frame sec}, we will see that this is a Cambrian framework.)  
\newsavebox{\smlmat}
\savebox{\smlmat}{$\begin{psmallmatrix*}[r]0&2\\-1&0\end{psmallmatrix*}$}
\begin{figure}
\begin{picture}(0,0)(-72,-75)
\put(12,-68){\small$\alpha_2$}
\put(-15,-71){\small$\alpha_1$}
\put(-50,-60){\small$-\alpha_1$}
\put(-47,-44){\small$2\alpha_1+\alpha_2$}
\put(-60,-10){\small$-2\alpha_1-\alpha_2$}
\put(-60,7){\small$\alpha_1+\alpha_2$}
\put(-47,41){\small$-\alpha_1-\alpha_2$}
\put(-45,58){\small$\alpha_2$}
\put(-24,69){\small$-\alpha_2$}
\put(47,-8){\small$-\alpha_2$}
\put(55,5){\small$\alpha_1$}
\put(12,63){\small$-\alpha_1$}
\end{picture}
\includegraphics{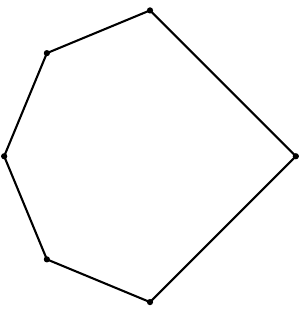}
\qquad
\begin{picture}(0,0)(-72,-75)
\put(12,-68){\small$\alpha\ck_2$}
\put(-19,-73){\small$\alpha\ck_1$}
\put(-55,-62){\small$-\alpha\ck_1$}
\put(-47,-44){\small$\alpha\ck_1+\alpha\ck_2$}
\put(-60,-10){\small$-\alpha\ck_1-\alpha\ck_2$}
\put(-60,7){\small$\alpha\ck_1+2\alpha\ck_2$}
\put(-47,41){\small$-\alpha\ck_1-2\alpha\ck_2$}
\put(-45,58){\small$\alpha\ck_2$}
\put(-24,69){\small$-\alpha\ck_2$}
\put(46,-8){\small$-\alpha\ck_2$}
\put(51,5){\small$\alpha\ck_1$}
\put(12,63){\small$-\alpha\ck_1$}
\end{picture}
\includegraphics{frameB2.pdf}
\caption[A framework]{A framework for the exchange matrix $B=$\usebox{\smlmat}}
\label{B2 frame fig}
\end{figure}
The figure shows two copies of the same graph, with labels on one and co-labels on the other.
The label on an incident pair $(v,e)$ is shown near $e$, closer to $v$ than to the other vertex of $e$.
The vertex at the bottom is $v_b$.
\end{example}

\begin{remark}\label{skew sym remark}
When $B$ is skew-symmetric, rather than merely skew-symmetrizable, each simple co-root equals the corresponding simple root.
By a simple inductive argument, a framework $(G,C,C\ck)$ for $B$ must have $C\ck(v,e)=C(v,e)$ for every incident pair $(v,e)$.
Thus, for skew-symmetric $B$, we may as well define a framework to be a pair $(G,C)$ satisfying the Sign condition, the Base condition (ignoring the requirement about $C\ck$), and the Transition condition (replacing $\beta\ck$ by $\beta$).
\end{remark}

\begin{remark}\label{co-label implications}
One might wonder whether the Transition and Co-transition conditions together imply the Co-label condition.  
Without some further condition, they do not, because of the issue mentioned below, in the paragraph between the statement and proof of Proposition~\ref{cotrans restate}.
\end{remark}

We now establish some first properties of frameworks.

\begin{prop}\label{basis}
Suppose $(G,C,C\ck)$ is a framework for $B$ and let $v$ be a vertex of $G$.
Then the label set $C(v)$ is a basis for the root lattice and $C\ck(v)$ is a basis for the co-root lattice.
\end{prop}
\begin{proof}
By the Base condition, the proposition holds for $v=v_b$.
If $v\neq v_b$, then since $G$ is connected, there is a finite path from $v$ to $v_b$.
The Transition condition says that at each step in the path, the set $C(\cdot)$ changes by a sequence of Gauss-Jordan operations that alter a label by adding an integer multiple of another label. 
Thus each step preserves the property of being a basis for the root lattice.
The assertion for $C(v)$ follows by an easy induction, and the assertion for $C\ck(v)$ follows by the analogous proof.
\end{proof}

Proposition~\ref{basis} also allows us to define \newword{mutations} of edges.
The use of the term ``mutation'' in this context is inspired by, and will be compatible with, the use of the term in connection with cluster algebras.
Let $e$ be a full edge connecting $v$ to $v'$.
We will define a function $\mu_e$ from $I(v)$ to $I(v')$.
Define $\mu_e(e)$ to be~$e$.
If $f\in I(v)\setminus\set{e}$, then for $\beta=C(v,e)$ and $\gamma=C(v,f)$, define $\mu_e(f)$ to be the edge $f'\in I(v')$ such that $C(v',f')=\gamma+[\sgn(\beta)\omega(\beta\ck,\gamma)]_+\,\beta$.
The notion of mutations of edge sets allows us to make the Transition condition slightly more specific, by identifying roots in $C(v)$ with roots in $C(v')$ in terms of edge mutations.
\\

\noindent
\textbf{Transition condition, strengthened:} \label{Transition condition, strengthened}
Suppose $v$ and $v'$ are distinct vertices incident to the same edge $e$.
Then $C(v,e)=-C(v',e)$.
Furthermore, if $f\in I(v)\setminus\set{e}$, then 
\[C(v',\mu_e(f))=C(v,f)+[\sgn(C(v,e))\omega(C\ck(v,e),C(v,f))]_+\,C(v,e).\]
The following proposition is immediate by the definition of $\mu_e$.
\begin{prop}\label{trans restate}
A framework for $B$ satisfies the strengthened Transition condition.
\end{prop}

We strengthen the Co-transition condition similarly, as follows:\\

\noindent
\textbf{Co-transition condition, strengthened:} \label{Co-transition condition, strengthened}
Suppose $v$ and $v'$ are distinct vertices incident to the same edge $e$.
Then $C\ck(v,e)=-C\ck(v',e)$.
Furthermore, if $f\in I(v)\setminus\set{e}$, then 
\[C\ck(v',\mu_e(f))=C\ck(v,f)+[-\sgn(C\ck(v,e))\omega(C\ck(v,f),C(v,e))]_+\,C\ck(v,e).\]

\begin{prop}\label{cotrans restate}
A framework for $B$ satisfies the strengthened Co-transition condition.
\end{prop}

Proposition~\ref{cotrans restate} is not immediate like Proposition~\ref{trans restate}, because, \emph{a priori}, we do not know that an edge mutation operation defined in terms of the Co-transition condition would agree with $\mu_e$.
The Co-label condition allows us to avoid this difficulty.

\begin{proof}
Suppose $v$ and $v'$ are distinct vertices incident to the same edge $e$.
The assertion that $C\ck(v,e)=-C\ck(v',e)$ is part of the (unstrengthened) Co-transition condition.
Furthermore, if $f\in I(v)\setminus\set{e}$, then the Co-label condition and the strengthened Transition condition imply that 
$C\ck(v',\mu_e(f))$ is a positive scalar multiple of 
\begin{equation}\label{is}
C(v,f)+[\sgn(C(v,e))\omega(C\ck(v,e),C(v,f))]_+\,C(v,e).
\end{equation}
The (unstrengthened) co-Transition condition says that $C\ck(v')$ contains the vector
\begin{equation}\label{contains}
C\ck(v,f)+[-\sgn(C\ck(v,e))\omega(C\ck(v,f),C(v,e))]_+\,C\ck(v,e).
\end{equation}
Writing $C\ck(v,e)=aC(v,e)$ and $C\ck(v,f)=bC(v,f)$, we rewrite \eqref{is} as
\begin{equation}\label{is rewrite}
C(v,f)+[\sgn(C(v,e))\omega(aC(v,e),C(v,f))]_+\,C(v,e).
\end{equation}
and rewrite \eqref{contains} as
\begin{equation}\label{contains rewrite}
bC(v,f)+[-\sgn(C(v,e))\omega(bC(v,f),C(v,e))]_+\,aC(v,e).
\end{equation}
Using the antisymmetry and linearity of $\omega$, we see that \eqref{contains rewrite} is $b$ times \eqref{is rewrite}.
Proposition~\ref{basis} implies that only one multiple of \eqref{is} is in $C\ck(v')$, so we conclude that \eqref{contains} is $C\ck(v',\mu_e(f))$.
\end{proof}

In light of Propositions~\ref{trans restate} and~\ref{cotrans restate}, we will tacitly use the strengthened forms of the Transition and Co-transition conditions when needed.
The point is that the original statements of the conditions are easier to state and potentially easier to check, but that the strengthened statements provide more precise control for use in arguments.

\begin{remark}\label{ratios}
The proof of Proposition~\ref{cotrans restate} also establishes the following fact:
The ratio between $C\ck(v',\mu_e(f))$ and $C(v',\mu_e(f))$ equals the ratio between $C\ck(v,f)$ and $C(v,f)$.
Thus the co-label-to-label ratios appearing throughout a framework are exactly the ratios between simple co-roots and simple roots.
\end{remark}

The dual bases to the co-label sets $C\ck(v)$ will be of great importance.
Given a framework $(G,C,C\ck)$, and a vertex $v$ of $G$, let $R(v)$ be the basis of $V^*$ that is dual to the basis $C\ck(v)$ of $V$.  
We emphasize that $R(v)$ is dual to $C\ck(v)$, not to $C(v)$.
More specifically, for each $e\in I(v)$, let $R(v,e)$ be the basis vector in $R(v)$ that is dual to $C\ck(v,e)$.

\begin{prop}\label{dual adjacent}
Let $(G,C,C\ck)$ be a framework for $B$ and let $v$ and $v'$ be adjacent vertices of $G$.
Then $R(v)\cap R(v')$ contains exactly $n-1$ vectors.
Specifically, if $e$ is the edge connecting $v$ to $v'$ and $f\in I(v)\setminus\set{e}$, then $R(v,f)=R(v',\mu_e(f))$.
Also, $R(v,e)$ and $R(v', e)$ lie on opposite sides of the hyperplane spanned by $R(v) \cap R(v')$.
\end{prop}
\begin{proof}
Let $f\in I(v)\setminus\set{e}$.
We know $\br{R(v,f),C\ck(v,p)}=\delta_{f,p}$ (Kronecker delta) for $p\in I(v)$.
The Co-transition condition says, in particular, that $C\ck(v,p)$ and $C\ck(v',\mu_e(p))$ differ by a multiple of $C\ck(v,e)$.
Thus $\br{R(v,f),C\ck(v',\mu_e(p))}$ equals $\br{R(v,f),C\ck(v,p)}=\delta_{f,p}$ for $p\in I(v)$.
Now $R(v,f)=R(v',\mu_e(f))$, because $R(v,f)$ satisfies the conditions that define $R(v',\mu_e(f))$.

Finally, by the definition of a dual basis,  the $(n-1)$-plane spanned by $R(v) \cap R(v')$ is $C\ck(v,e)^{\perp}$. 
Since, by the definition of a dual basis, $\langle C\ck(v,e), R(v,e) \rangle =1$ and $\langle C\ck(v,e), R(v',e) \rangle = \langle -C\ck(v',e), R(v',e) \rangle =-1$, we see that $R(v,e)$ and $R(v', e)$ are on opposite sides of this plane.
\end{proof}

We define $\Cone(v)$ to be the simplicial cone in $V^*$ spanned by the $R(v,e)$, as $e$ ranges over the neighbors of $v$, so $\Cone(v)=\bigcap_{e\in I(v)}\set{x\in V^*:\br{x,C\ck(v,e)}\ge 0}$. 
Proposition~\ref{dual adjacent} has the following corollary. 
\begin{cor} \label{cor:BdyFacet}
Let $(G,C,C\ck)$ be a framework for $B$ and let $v$ and $v'$ be adjacent vertices of $G$.
Then the cones $\Cone(v)$ and $\Cone(v')$ intersect in a common facet.
\end{cor}

We conclude with the following observation:
\begin{prop}\label{-B^T}
If $(G,C,C\ck)$ is a framework for $B$ then $(G,C\ck,C)$ is a framework for $-B^T$.
\end{prop}
In the notation $-B^T$, the superscript $T$ denotes transpose.
Thus moving from $B$ to $-B^T$ should be thought of as ``transposing the absolute values but not the signs.''
The Cartan companion $\Cart(B)$ of $B$ defines the simple roots $\Pi$ and simple co-roots $\Pi\ck$.
The exchange matrix $-B^T$ has Cartan companion $\Cart(-B^T)=\Cart(B)^T$, which defines simple roots $\Pi\ck$ and simple co-roots $\Pi$.
Proposition~\ref{-B^T} is immediate from the definition, once we switch the roles of simple roots and simple co-roots, and switch the roles of the Transition condition and Co-transition condition.

\subsection{Reflection frameworks}\label{ref frame subsec}
We now define a special kind of framework called a reflection framework, in which all labels are roots and co-roots, and the Transition condition can be rephrased in terms of the action of reflections in the Coxeter group.

The \newword{Euler form} \label{Euler form} $E$ associated to $B$ is defined on the bases of simple roots and co-roots as follows: 
\[E(\alpha\ck_{i},\alpha_{j})=\left\lbrace\begin{array}{ll}
[b_{ij}]_{-} &\mbox{if } i\neq j,\mbox{ or}\\
1&\mbox{if }i=j.
\end{array}\right.\]

Recall that the Cartan companion $\Cart(B)$ of $B$ is the matrix $[a_{ij}]$ with diagonal entries $a_{ii}=2$ and off-diagonal entries $a_{ij}=-|b_{ij}|$, where $B=[b_{ij}]_{i,j\in I}$.
Define a bilinear form $K$ on $V$ by $K(\alpha\ck_i, \alpha_j)=a_{ij}$.
The form $K$ is symmetric.
(The proof is essentially identical to the proof that $\omega$ is anti-symmetric.)

\begin{prop}\label{sym antisym}
The form $\omega$ is given by $\omega(\beta,\gamma)=E(\beta,\gamma)-E(\gamma,\beta)$, and the form $K$ is given by $K(\beta,\gamma)=E(\beta,\gamma)+E(\gamma,\beta)$.
\end{prop}
\begin{proof}
We will check these identities for $\beta=\alpha_i\ck$ and $\gamma=\alpha_j$.
If $i=j$, then the identities are easy, so suppose $i\neq j$.
Then $E(\alpha_i\ck,\alpha_j)-E(\alpha_j,\alpha_i\ck)$ equals
\[E(\alpha_i\ck,\alpha_j)-\frac{\delta(j)}{\delta(i)}E(\alpha_j\ck,\alpha_i)
=\min(b_{ij},0)-\frac{\delta(j)}{\delta(i)}\min(b_{ji},0),\]
which equals $\min(b_{ij},0)+\max(b_{ij},0)=b_{ij}$ because $\delta(i) b_{ij}=-\delta(j) b_{ji}$.
This is the first identity.

Similarly, $E(\alpha_i\ck,\alpha_j)+E(\alpha_j,\alpha_i\ck)=\min(b_{ij},0)-\max(b_{ij},0)=-|b_{ij}|=a_{ij}$.
\end{proof}

The following is an immediate corollary of Proposition~\ref{sym antisym}.
\begin{cor}\label{E 0}
If $\beta,\gamma\in V$ then $|\omega(\beta,\gamma)|=|K(\beta,\gamma)|$ if and only if either $E(\beta,\gamma)=0$ or $E(\gamma,\beta)=0$ or both.
\end{cor}

The Cartan matrix $\Cart(B)=[a_{ij}]$ defines a Coxeter group $W$ generated by $S=\set{s_i:i\in I}$ whose defining relations $(s_is_j)^{m(i,j)}$ are given by 
\[m(i,j)=\left\lbrace\begin{array}{ll}
2&\mbox{if }a_{ij}\cdot a_{ji}=0,\\
3&\mbox{if }a_{ij}\cdot a_{ji}=1,\\
4&\mbox{if }a_{ij}\cdot a_{ji}=2,\\
6&\mbox{if }a_{ij}\cdot a_{ji}=3,\mbox{ or}\\
\infty&\mbox{if }a_{ij}\cdot a_{ji}\ge 4.
\end{array}\right.\]
The Cartan matrix also defines an action of $W$ on $V$.
The action of a generator $s_i\in S$ on a simple root $\alpha_j$ is $s_i(\alpha_j)=\alpha_j-a_{ij} \alpha_i$ and the action on a simple co-root $\alpha\ck_j$ is $s_i(\alpha\ck_j)=\alpha\ck_j-a_{ji} \alpha\ck_i$.
The action of $W$ preserves the form $K$.

A \newword{real root} $\beta$ is a vector in the orbit, under the action of the Coxeter group $W$, of some simple root.
\newword{Real co-roots} are defined similarly. 
The \newword{real root system} $\Phi$ associated to $\Cart(B)$ is the set of all real roots.
A root is \newword{positive} if it is in the nonnegative linear span of the simple roots.
Otherwise, it is in the nonpositive linear span of the simple roots and is called \newword{negative}.
The \newword{reflections} in $W$ are the elements conjugate to elements of $S$.
There is a bijection $t\to\beta_t$ from reflections to positive roots and a bijection $t\to\beta\ck_t$ from reflections to positive co-roots such that $t$ acts on $V$ by sending $x\in V$ to $tx=x-K(\beta\ck_t, x) \beta_t$.
The action of $W$ on $V$ commutes with passing from roots to co-roots:
That is, $w(\beta\ck)=(w\beta)\ck$ for any $w\in W$ and any root $\beta$.

The set of  \newword{almost positive roots} is the union of the set of positive roots and the set of negative simple roots. 
Similarly, the almost positive co-roots are those co-roots that are either positive or negative simple. 
If $V$ is two dimensional, then the almost positive roots are vectors in a two dimensional vector space, no two of which are positive multiples of each other, so we may consider them to be cyclically ordered as they wind around the origin.  
This concept and ordering will return later in the paper.

Let $G$ be a connected quasi-graph and let $C$ be a labeling of each incident pair in $G$ by a vector in $V$.
We define some additional conditions on $(G,C)$.\\

\noindent
\textbf{Root condition:}   \label{Root condition}
Each label $C(v,e)$ is a real root with respect to the Cartan matrix $\Cart(B)$.\\

A pair $(G,C)$ satisfying the Root condition automatically satisfies the Sign condition.
Suppose $v$ is a vertex of $G$.
Define $C_+(v)$ to be the set of positive roots in $C(v)$ and define $C_-(v)$ to be the set of negative roots in $C(v)$.
Let $\Gamma(v)$ be the directed graph whose vertex set is $C(v)$, with an edge $\beta \to \beta'$ if $E(\beta, \beta') \neq 0$. 
\\

\noindent
\textbf{Euler conditions:} \label{Euler conditions}
Suppose $v$ is a vertex of $G$ and let $e$ and $f$ be distinct edges incident to $v$. 
Write $\beta=C(v,e)$ and $\gamma=C(v,f)$.
Then
\begin{enumerate}
\item[(E0) ]  At least one of $E(\beta,\gamma)$ and $E(\gamma,\beta)$ is zero.
\item[(E1) ] If $\beta\in C_+(v)$ and $\gamma\in C_-(v)$ then $E(\beta,\gamma)=0$. 
\item[(E2) ] If $\sgn(\beta)=\sgn(\gamma)$ then $E(\beta,\gamma)\le0$. 
\item[(E3) ] The graph $\Gamma(v)$ is acyclic.
\end{enumerate}
Condition (E0) follows immediately from condition (E3), but we list it separately for convenience.\\

\noindent
\textbf{Reflection condition:}  \label{Reflection condition}
Suppose $v$ and $v'$ are distinct vertices incident to the same edge $e$.
If $\beta=C(v,e)=\pm\beta_t$ for some reflection $t$ and $\gamma\in C(v)$, then $C(v')$ contains the root 
\[\gamma'=\left\lbrace\begin{array}{ll}
t\gamma&\mbox{if }\omega(\beta_t\ck,\gamma)\ge 0,\mbox{ or}\\
\gamma&\mbox{if }\omega(\beta_t\ck,\gamma)<0.
\end{array}\right.\]

Applying the Reflection condition with $\gamma=\beta$, we see that $-\beta\in C(v')$.
Condition (E0) implies, in light of Proposition~\ref{sym antisym}, that $\omega(\beta_t\ck,\gamma)$ and $K(\beta_t\ck,\gamma)$ agree in absolute value.
Thus except in the case $\gamma=\beta$, we can replace the ``$<$'' sign by ``$\le$'' in the Reflection condition.
Since $\gamma$ and $t\gamma$ differ by a multiple of $\beta$, and since $\omega$ is antisymmetric, the Reflection condition is symmetric in $v$ and $v'$.
In particular, for any edge, it is enough to check the condition in one direction.
Just as we strengthened the Transition condition, we can strengthen the Reflection condition, by defining $\mu_e$ so that the strengthened condition holds.

The pair $(G,C)$ is a \newword{reflection framework} if it satisfies the Base condition (ignoring the assertion about the labeling $C\ck$), the Root condition, the Reflection condition and the Euler conditions.

If a reflection framework exists for $B$, then in particular $B$ is acyclic in the usual sense for exchange matrices:  
Namely, the directed graph on $I$, with directed edges $i\to j$ if and only if $b_{ij}<0$,  is acyclic.
The acyclicity follows from (E3) because the directed graph on $I$ is isomorphic to $\Gamma(v_b)$ by the Base condition.

The following proposition relates reflection frameworks to frameworks.
\begin{prop}\label{ref implies frame}
Suppose $(G,C)$ is a reflection framework and let $C\ck$ be the labeling of incident pairs of $G$ such that $C\ck(v,e)$ is the co-root associated to $C(v,e)$.
Then $(G,C,C\ck)$ is a framework.
\end{prop}

\begin{proof}  
The Co-label condition holds by the definition of $C\ck$, the Sign condition holds by the Root condition, and the Base condition holds by hypothesis and by the definition of $C\ck$.
We need to verify the Transition and Co-Transition conditions.

Suppose $v$ is a vertex connected, by an edge $e$, to another vertex $v'$.
By the Root condition, $C(v,e)$ is a root $\beta=\pm\beta_t$.
By the anti-symmetry of $\omega$ we see that $\omega(\beta_t\ck,\beta)=0$, so the Reflection condition says, in particular, that $C(v',e)=t\beta=-\beta$, so the first condition of the Transition condition holds.
Furthermore, take $f$ to be an edge, distinct from $e$, in $I(v)$.
Write $\gamma$ for $C(v,f)$ and $\gamma'$ for $C(v',\mu_e(f))$.
By the (strengthened) Reflection condition, $\gamma'=t\gamma=\gamma-K(\beta\ck,\gamma) \beta$ if $\omega(\beta_t\ck,\gamma)\ge 0$ or $\gamma'=\gamma$ if $\omega(\beta_t\ck,\gamma)<0$.
The (strengthened) Transition condition is the assertion that $\gamma'=\gamma+[\omega(\beta_t\ck,\gamma)]_+\,\beta$,  
so we have established this condition when $\omega(\beta_t\ck,\gamma)<0$, and it remains to establish the following:
If $\omega(\beta_t\ck,\gamma)\ge 0$ then $\omega(\beta_t\ck,\gamma)=-K(\beta\ck,\gamma)$.
Suppose $\omega(\beta_t\ck,\gamma)\ge 0$ and consider the four cases given by the signs of $\beta$ and $\gamma$.

If $\beta$ is positive, then $\beta=\beta_t$, so $E(\beta\ck,\gamma)-E(\gamma,\beta\ck)=\omega(\beta_t\ck,\gamma)\ge 0$.
If $\gamma$ is also positive, then Condition (E2) says that $E(\beta\ck,\gamma)\le 0$ and $E(\gamma,\beta\ck)\le 0$, and by condition (E0) we conclude that $E(\beta\ck,\gamma)=0$.
If $\gamma$ is negative, then condition (E1) says that $E(\beta\ck,\gamma)=0$.
In either case, $\omega(\beta_t\ck,\gamma)=\omega(\beta\ck,\gamma)=-E(\gamma,\beta\ck)$ and $K(\beta\ck,\gamma)=E(\gamma,\beta\ck)$.

If $\beta$ is negative, then $\beta=-\beta_t$, so $E(\beta\ck,\gamma)-E(\gamma,\beta\ck)=-\omega(\beta_t\ck,\gamma)\le 0$.
If $\gamma$ is also negative, then again, $E(\beta\ck,\gamma)\le 0$ and $E(\gamma,\beta\ck)\le 0$, so this time $E(\gamma,\beta\ck)=0$.
If $\gamma$ is positive, then $E(\gamma,\beta\ck)=0$ by condition (E1), so $\omega(\beta_t\ck,\gamma)=-\omega(\beta\ck,\gamma)=-E(\beta\ck,\gamma)$ and $K(\beta\ck,\gamma)=E(\beta\ck,\gamma)$.

The reflection condition implies that $(\gamma')\ck=t(\gamma\ck)=\gamma\ck-K(\beta\ck,\gamma\ck)\beta=\gamma\ck-K(\gamma\ck,\beta)\beta\ck$ if $\omega(\beta_t\ck,\gamma)\ge 0$ or $(\gamma')\ck=\gamma\ck$ if $\omega(\beta_t\ck,\gamma)<0$.
The sign of $\omega(\beta_t\ck,\gamma)$ agrees with the sign of $\omega(\gamma\ck,\beta_t)$, so the Co-transition condition follows by a similar argument.
\end{proof}

In fact, the proof above of Proposition~\ref{ref implies frame} establishes that, when the Root condition and the Euler conditions hold, the Transition condition, the Reflection condition, and the Co-transition conditions are all equivalent.
The proof still goes through under a weakening of the Euler conditions:  We only need conditions (E0)--(E2), and only in the case where at least one of the edges $e$ and $f$ is a full edge.

\begin{example}\label{B2 ref example}
The framework described in Example~\ref{B2 example} is a reflection framework.
\end{example}

We pointed out above that a reflection framework does not exist for $B$ if $B$ is not acyclic.
Acyclicity of $B$ turns out to be the only factor determining the existence of a complete reflection framework.
The following is \cite[Theorem~6.5]{ST}.
\begin{theorem}\label{ST frame}
Given any acyclic exchange matrix $B$, a complete reflection framework exists for $B$. 
\end{theorem}

\begin{remark}\label{ST remark}
Since this paper and \cite{ST} cite each other, we want to explicitly dispel any worries about circular reasoning.
This paper quotes only \cite[Theorem~6.5]{ST} (our Theorem~\ref{ST frame}) and mentions some corollaries to that theorem.
The paper \cite{ST} cites this paper only for definitions and to draw further conclusions from \cite[Theorem~6.5]{ST}, but not to prove \cite[Theorem~6.5]{ST}.
\end{remark}

\section{Cluster algebras and frameworks}\label{frame clus sec}
In this section, we review background material on cluster algebras, show how frameworks are combinatorial models for cluster algebras, and establish the properties of these models.

\subsection{Cluster algebras}\label{ca background sec}
As before, let $I$ be a finite indexing set with $|I|=n$ and let $B$ be a skew-symmetrizable integer matrix, with rows and columns indexed by $I$.
Let $\PP$ be a semifield (an abelian multiplicative group with a second commutative, associative operation $\oplus$ such that the group multiplication distributes over $\oplus$). 
Let $Y=(y_i:i\in I)$ be a tuple of elements of $\PP$.
Let $\FF$ be the field of rational functions in $n$ indeterminates with coefficients in $\QQ\PP$.
Let $X=(x_i:i\in I)$ be algebraically independent elements of $\FF$ with $\FF = \QQ \PP(x_i)_{i \in I}$.

Let $T$ be the $n$-regular tree.
We will continue the graph notation from Section~\ref{frame sec}.
For each pair $v,v'$ of vertices of $T$, connected by the edge $e$, let $\mu_e$ be a bijection from the set $I(v)$ of edges incident to $v$ to the set $I(v')$.
This defines two maps called $\mu_e$ for each edge $e$.
We will let the context distinguish the two, and we require that $\mu_e:I(v')\to I(v)$ is the inverse of $\mu_e:I(v)\to I(v')$.

Distinguish a vertex $v_b$ of $T$ and identify $I$ with the set $I(v_b)$ of edges incident to $v_b$.
The data of $(B,Y,X)$ constitute the \newword{initial seed} of the cluster algebra that we define below.
The matrix $B$ is the \newword{exchange matrix} in the seed, the elements $y_i$ are the \newword{coefficients} in the seed and the tuple $X=(x_i:i\in I)$ is the \newword{cluster}, with the individual elements $x_i$ called \newword{cluster variables}.
More generally, a \newword{seed} is any triple consisting of an exchange matrix with rows and columns indexed by $I'$ (for $|I'|=n$), a tuple $(y'_i:i\in I')$ of coefficients and a cluster $(x'_i:i\in I')$ of algebraically independent elements of $\FF$. 

We define $(B^{v_b},Y^{v_b},X^{v_b})=(B,Y,X)$ and further overload the notation $\mu_e$ to define \newword{seed mutations}.
These seed mutations inductively associate a seed to each vertex of $T$.
The indexing set for the seed at $v$ is the set $I(v)$ of edges incident to $v$.
Let $B^v=[b^v_{pq}]_{p,q\in I(v)}$ be the exchange matrix associated to $v$, let $Y^v=(y^v_{p}:p\in I(v))$ be the coefficients, and let $X^v=(x^v_p:p\in I(v))$ be the cluster.\\

\noindent
\textbf{Matrix mutation.} 
Let $e$ be an edge connecting $v$ to $v'$ and define $B^{v'}=\mu_e(B^v)$ by setting
\begin{equation}
b^{v'}_{\mu_e(p)\mu_e(q)} = 
\begin{cases} 
- b^v_{pq} & \mbox{if $p=e$ or $q=e$} \\
b^v_{pq} + \sgn(b^v_{pe})[b^v_{pe}b^v_{eq}]_{+}& \mbox{otherwise.} 
\end{cases}
\label{BMatrixRecurrence}
\end{equation}

\noindent
\textbf{Coefficient mutation.}
Let $e$ be an edge connecting $v$ to $v'$ and define $Y^{v'}=\mu_e(Y^v)$ by setting
\begin{equation}
y^{v'}_{\mu_e(p)}=
\begin{cases}
(y^v_e)^{-1} &\mbox{if }p=e\\
y^v_p(y^v_e)^{[b^v_{ep}]_+}(y^v_e\oplus1)^{-b^v_{ep}} &\mbox{if }p\in I(v)\setminus\set{e}
\end{cases}
\label{YRecurrence}
\end{equation}

\noindent
\textbf{Cluster mutation.}
Let $e$ be an edge connecting $v$ to $v'$ and define $X^{v'}=\mu_e(X^v)$ by setting  
\begin{equation}
x^{v'}_{\mu_e(q)} = 
\begin{cases}
\frac{1}{(y^v_e\oplus 1)x^v_e} \left( y^v_e \prod_p (x^v_p)^{[b^v_{pe}]_+} + \prod_p (x^v_p)^{[-b^v_{pe}]_+} \right) &\mbox{if }q=e\\
x^v_q & \mbox{if }q\neq e.
\label{ClusterRecurrence}
\end{cases}
\end{equation}
In both products above, $p$ ranges over the set $I(v)$.\\

\noindent
\textbf{Seed mutation.}
Let $e$ be an edge $v$ to $v'$ and define $(B^{v'},Y^{v'},X^{v'})$ to be $\mu_e(B^v,Y^v,X^v)=(\mu_e(B^v),\mu_e(Y^v),\mu_e(X^v))$.
Seed mutation is involutive.  
That is, $\mu_e(\mu_e(B^v,Y^v,X^v))=(B^v,Y^v,X^v)$.\\

The \newword{cluster algebra} $\A(B,Y,X)$ is the subalgebra of $\FF$ generated by all of the cluster variables $x_p^v$ where $v$ ranges over all vertices of $T$ and $p$ ranges over $I(v)$.
The cluster algebra is determined by any of its seeds $(B^v,Y^v,X^v)$.
Up to isomorphism, it is determined by any of the pairs $(B^v,Y^v)$.

Two vertices $v,v'\in T$ define \newword{equivalent seeds} if there is a bijection $\lambda:I(v)\to I(v')$ such that, first $b^{v'}_{\lambda(e)\lambda(f)}=b^v_{ef}$, second $y^{v'}_{\lambda(e)}=y^v_e$, and third $x^{v'}_{\lambda(e)}=x^v_e$ for all $e,f\in I(v)$.
When such a $\lambda$ exists, it is unique, and furthermore, by seed mutation it induces a bijection from the neighbors of $v$ to the neighbors of $v'$ such that each neighbor of~$v$ defines a seed that is equivalent to the seed at the corresponding neighbor of $v'$.
Thus we can define a quotient of $T$ by identifying two vertices $v$ and $v'$ if they define equivalent seeds, and identifying edges of $v$ with edges of $v'$ according to the map $\lambda$.
This quotient is the \newword{exchange graph} $\Ex(B,Y,X)$ associated to $(B,Y,X)$.
The mutation maps $\mu_e$ are compatible with this quotient.
Thus we can think of $\mu_e$ as a map, for each pair $v,v'$ of vertices of the exchange graph connected by the edge $e$, from the set $I(v)$ of edges incident to $v$ in the exchange graph to the set $I(v')$ of edges incident to $v'$ in the exchange graph.
In all of the notation defined above, we can safely replace the $n$-regular tree $T$ by the exchange graph $\Ex(B,Y,X)$.

The first surprising theorem \cite[Theorem~3.1]{ca1} about cluster algebras is the following result, known as the Laurent phenomenon: 
\begin{theorem}\label{Laurent thm}
For any vertex $v$ in $\Ex(B,Y,X)$, the cluster algebra $\A(B,Y,X)$ is contained in the ring $\ZZ\PP[(x_i^v)^{\pm1}:i\in I(v)]$.
In other words, if $x$ is any cluster variable in any cluster, then $x$ can be uniquely written as
\begin{equation}
x = \frac{N(x_i^v:i\in I(v))}{\prod_{i\in I(v)}(x_i^v)^{d_i}}, \label{Laurent}
\end{equation}
where the $d_i$ are integers and $N$ is a polynomial in the variables $(x_i^v:i\in I(v))$ with coefficients in $\integers\PP$ which is not divisible by any of the variables $x_i^v$.
\end{theorem}
The vector $\d^v(x)=\sum_{i\in I}d_i\alpha_i$ in the root lattice is the \newword{denominator vector} of $x$ with respect to the vertex $v$.
Usually, the denominator vector is defined to be the integer vector $(d_i:i\in I(v))$, which can be recovered from $\d^v(x)$ by taking simple root coordinates.
We will only consider denominator vectors with respect to the vertex $v_b$, so we will use the abbreviation $\d(x)$ for $\d^{v_b}(x)$.

The most important instances of cluster algebras are the cluster algebras of geometric type.
Let $J$ be an indexing set, disjoint from $I$, with $|J|=m$ and let $(x_j:j\in J)$ be independent variables.
Let $\PP=\Trop(x_j:j\in J)$ be the \newword{tropical semifield} generated by $(x_j:j\in J)$.
This is the free abelian group generated by $(x_j:j\in J)$, written multiplicatively, with an addition operation $\oplus$ defined by 
\[\prod_{j\in J}x_j^{a_j}\oplus\prod_{j\in J}x_j^{b_j}=\prod_{j\in J}x_j^{\min(a_j,b_j)}.\]

Let $\tB=[b_{ij}]$ be an integer matrix with rows indexed by the disjoint union $I\uplus J$ and columns indexed by $I$ such that $B$ is the $n\times n$ matrix consisting of the rows of $\tB$ indexed by $I$.  
Such a matrix is called an \newword{extended exchange matrix}.
The rows of $\tB$ indexed by $J$ specify a tuple $(y_i:i\in I)$ of elements of $\PP$ by setting $y_i=\prod_{j\in J}x_j^{b_{ji}}$, for each $i\in I$.
Thus a pair $(\tB,X)$ encodes a seed.
Following the construction from above, we associate a seed $(\tB^v,X^v)$ to each vertex.
The cluster algebra generated in this way is called a \newword{cluster algebra of geometric type}.
The matrix $\tB^v$ has rows indexed by $I(v)\cup J$ and columns indexed by $I(v)$.
If $v$ and $v'$ are connected by an edge $e$, the relationship between the \emph{extended} exchange matrices $\tB^v$ and $\tB^{v'}$ is given by the matrix mutation relation (\ref{BMatrixRecurrence}), where $p$ is in $I(v)\cup J$ rather than $I(v)$.
In particular, coefficient mutation does not need to be treated separately, but coefficients associated to a vertex $v$ can still be recovered as $y^v_i=\prod_{j\in J}x_j^{b^v_{ji}}$.
Mutation of clusters can also be written more simply.\\

\noindent
Let $e$ be an edge $v$ to $v'$ and define $X^{v'}=\mu_e(X^v)$ by setting 
\begin{equation}
x^{v'}_{\mu_e(q)} = 
\begin{cases}
\frac{1}{x^v_q} \left(\prod_p (x^v_p)^{[b^v_{pq}]_+} + \prod_p (x^v_p)^{[-b^v_{pq}]_+} \right) &\mbox{if }q=e\\
x^v_q & \mbox{if }q\neq e.
\label{ClusterRecurrenceGeom}
\end{cases}
\end{equation}  
In both products, $p$ now ranges over the set $I(v)\cup J$ rather than the set $I(v)$.

Of primary importance among cluster algebras of geometric type are the cluster algebras with \newword{principal coefficients}. 
In this case we take $J$ to be (a disjoint copy of) $I$, so that the initial extended exchange matrix $\tB$ is a $2n\times n$ matrix with $B$ in the top $n$ rows.
The bottom $n$ rows are taken to be the $n\times n$ identity matrix.  
In general, we write $\tB^v$ as $\twomatrix{B^v}{H^v}$ where $B^v$ is the exchange matrix associated to $v$ as before and $H^v$ is a matrix with rows indexed by $I$ and columns indexed by $I(v)$. 
Let $\A_\bullet(B)$ and $\Ex_\bullet(B)$ be the cluster algebra $\A(B,Y,X)$ and exchange graph $Ex(B,Y,X)$ where $Y$ are principal coefficients.
These depend on $X$ only up to isomorphism.

Recall from Section~\ref{frame subsec} that we associate to $B$ a Cartan matrix $\Cart(B)$, simple roots $\Pi$ and simple co-roots $\Pi\ck$.
The \newword{fundamental weights} are the vectors in the basis of $V^*$ that is dual to the basis $\Pi\ck$ for $V$.
Since the indexing set $I$ indexes rows and columns of $B$, it also indexes rows and columns of $\Cart(B)$, and thus indexes $\Pi$ and $\Pi\ck$.
We write $\rho_i$ for the fundamental weight that is dual to $\alpha\ck_i$.
The \newword{weight lattice} is the lattice generated by the fundamental weights.

In a cluster algebra with principal coefficients, the \newword{$\g$-vector} $\g(x)$ of a cluster variable $x$ is a recursively-defined vector in the weight lattice.
We will write $\g_e^v$ for $\g(x^v_e)$.
For each $x^{v_b}_i=x_i$ in the initial cluster $X$, the $\g$-vector of $x_i$ is $\rho_i$.
For other cluster variables, the $\g$-vector is defined by the following recursion.\\

\noindent
\textbf{$\g$-vector mutation.}  
Let $e$ be an edge connecting $v$ to $v'$. 
The $\g$-vectors of the cluster $X^{v'}$ are given by  
\begin{equation}
\g^{v'}_{\mu_e(q)} = 
\begin{cases}
-\g^v_q + \sum_{p\in I(v)} [-b^v_{pq}]_+\,\g^v_p - \sum_{i\in I}  [-b^v_{iq}]_+\,\b_i &\mbox{if }q=e\\
\g^v_q & \mbox{if }q\neq e.
\label{gRecurrence}
\end{cases}
\end{equation}
Here, $\b_i$ is the vector in $V^*$ whose fundamental-weight coordinates are given by the $i\th$ column of the initial exchange matrix $B$.

In a cluster algebra with principal coefficients, the \newword{$\c$-vectors} at the seed $v$ are the vectors whose simple-root coordinates are given by the columns of $H^v$.
Specifically, write $\c_e^v$ for the column of $H^v$ indexed by $e\in I(v)$.
It is important to remember that $\c_e^v$ is not ``the $\c$-vector associated to the cluster variable $x_e^v$'' because this notion is not well-defined:
When $x_e^v=x_{e'}^{v'}$ for vertices $v\neq v'$ and edges $e\in I(v)$ and $e'\in I(v')$, typically $\c_e^v\neq \c_{e'}^{v'}$.

\begin{remark}\label{gvec weight or int vec}
Usually (including in the introduction to this paper) the $\g$-vector is defined as an integer vector rather than a vector in the weight lattice.
The integer vector can be recovered by taking fundamental-weight coordinates.
Similarly, the $\c$-vectors are usually defined as integer vectors (the columns of $H_v$) rather than as vectors in the root lattice, but the integer vector can be recovered by taking simple-root coordinates.
\end{remark}
\begin{remark}
Taking the recursive formula above as a definition, it is not immediately clear that the $\g$ vector is well-defined, but the definition in \cite[Sections~6--7]{ca4} does not have this problem.
The recursive formulation above is the alternate form \cite[Equation~(6.13)]{ca4} of \cite[Proposition~6.6]{ca4}, rewritten to define a vector in the weight lattice.
\end{remark}

The motivating questions of this paper are how to compute the exchange graph, the denominator vectors, the $\g$-vectors, the matrices $B^v$ and the $\c$-vectors.
As we will see in this paper, \emph{the $\c$-vectors should be considered the most fundamental objects}.

\subsection{Polyhedral geometry}\label{poly sec} 
In discussing cluster algebras and frameworks, it will be useful to use the language of polyhedral cones and fans.
We briefly review some background material.
A \newword{closed polyhedral cone} is a subset $F$ of $V^*$ that is of the form $\bigcap \{ x \in V^* : \br{x, \beta_i} \geq 0 \}$ for a finite list of vectors $\beta_1$, \ldots, $\beta_k$ in $V^*$.
Equivalently, a closed polyhedral cone is the nonnegative linear span of a finite set of vectors in $V^*$.
In this paper, the term \newword{cone} will be used as a shorthand for ``closed polyhedral cone.''
The cone is called \newword{simplicial} if $\beta_1$, \ldots, $\beta_k$ can be chosen so as to form a basis for $V$, or equivalently, if it is the nonnegative linear span of a basis for $V^*$.
One can similarly define cones and fans in $V$, but we will only consider them in $V^*$.

If $F$ is a cone, then a subset $G$ of $F$ is called a \newword{face} of $F$ if there is some $\lambda\in V$ (thought of as a linear functional on $V^*$) that is nonnegative on $F$ with $F\cap\ker\lambda=G$.
Note that $F$ is a face of itself.  
(Take the zero linear functional.)
A \newword{facet} of $F$ is a face $G$ of $F$ with $\dim(G)=\dim(F)-1$.
The \newword{relative interior} of a cone $F$ is the set of points of $F$ not in any proper face of $F$.
Topologically, the relative interior is the interior of $F$ as a subset of $\Span_{\RR}(F)$.

We'll say that cones $F_1$ and $F_2$ \newword{meet nicely} if $F_1 \cap F_2$ is a face of both $F_1$ and $F_2$.
A collection $\F$ of cones in $V^*$ is called a \newword{fan} if
\begin{enumerate}
\item For any cone $F$ in $\F$, and any face $G$ of $F$, the cone $G$ is also in $\F$.
\item Any two cones $F_1$ and $F_2$ in $\F$ meet nicely.
\end{enumerate}
See~\cite[Chapter V]{Ewald} for more on fans.

We will need some well-known, easy facts from polyhedral geometry:

\begin{prop}[Chapter II of \cite{Ewald}]
A cone has finitely many faces, each of which is itself a cone.
\end{prop}
\begin{prop}[Proposition~2.3 of~\cite{Ziegler}]
If $F$ is a cone, $G$ is a face of $F$, and $H$ is a face of $G$, then $H$ is a face of $F$.
\end{prop}
\begin{prop}[Proposition~2.3 of~\cite{Ziegler}]
If $F$ is a cone, then any two faces of $F$ meet nicely.
\end{prop}
\begin{prop}[Lemma~14 of~\cite{ranktests}] \label{prop:FanLemma}
Suppose $F$ and $G$ are cones that meet nicely. Let $F'$ be a face of $F$ and let $G'$ be a face of $G$.
Then $F'$ and $G'$ meet nicely. 
\end{prop}
Proposition~\ref{prop:FanLemma} immediately implies the following lemma, which simplifies the process of checking that a set of cones is a fan.

\begin{lemma} \label{MaxCheckFan}
Let $\C$ be a collection of cones. 
Suppose that every pair of cones in $\C$ meet nicely. 
Let $\F$ be the collection of all faces of cones in $\C$. 
Then $\F$ is a fan.
\end{lemma}

\subsection{Some conjectures about cluster algebras}\label{conj sec}
In this section, we review some conjectures from \cite{FZ-CDM,ca4} and make a few new conjectures that are suggested by the results of this paper.

The following conjecture is~\cite[Conjecture~4.14(3)]{FZ-CDM}.
Informally, the conjecture says that, if two cluster variables are equal to each other, then they are equal for an obvious reason.
\begin{conj}\label{vertex conj}
For any cluster variable $x$, the seeds whose clusters contain $x$ induce a connected subgraph of the exchange graph.
\end{conj}

We will prove the following stronger conjecture for some matrices $B$.
\begin{conj}\label{face conj}
For any set $\X$ of cluster variables, the seeds whose clusters contain $\X$ as a subset induce a connected subgraph of the exchange graph.
\end{conj}

Two extended exchange matrices $\tB^u$ and $\tB^v$ are equivalent if there exists a bijection $\lambda:I(u)\to I(v)$ such that $b^u_{ef}=b^v_{\lambda(e)\lambda(f)}$ for all $e,f\in I(u)$ and $b^u_{ie}=b^v_{i\lambda(e)}$ for all $e\in I(u)$ and $i\in J$.
The following is \cite[Conjecture~4.7]{ca4}.

\begin{conj}\label{tildeB equiv}
Let $u$ and $v$ be vertices in the exchange graph of a cluster algebra with principal coefficients.
Then $\tB^u$ and $\tB^v$ are equivalent if and only if $u=v$.
\end{conj}

In other words, when we associate a seed to each vertex of the $n$-regular tree $T$, the seeds associated to two vertices are equivalent if and only if the two extended exchange matrices are equivalent.  
We offer the following strengthening of Conjecture~\ref{tildeB equiv}.

\begin{conj}\label{H equiv} 
A vertex $v$ in the exchange graph of a cluster algebra with principal coefficients is completely determined by its set of $\c$-vectors.
That is, if $u$ and $v$ are distinct vertices in the exchange graph, then $H^u$ and $H^v$ are not related by permuting columns.
\end{conj}

The following conjecture is \cite[Conjecture~7.10(2)]{ca4}.

\begin{conj}\label{g lattice conj}
For each $v\in\Ex_\bullet(B)$, the vectors $\{ \g(x^v_e):e\in I(v) \}$ are a $\integers$-basis for the weight lattice.
\end{conj}

A collection of vectors in $\reals^n$ is \newword{sign-coherent} if, for all $i$ from $1$ to $n$, the $i\th$ coordinates of the vectors are either all nonnegative, or all nonpositive.
The following conjecture is \cite[Conjecture~6.13]{ca4}.

\begin{conj}\label{g sign-coherent}
For each vertex $v$ of $\Ex_\bullet(B)$, the fundamental-weight coordinate vectors of $\{ \g(x^v_e):e\in I(v) \}$ are a sign-coherent collection.
\end{conj}
In other words, the conjecture is that for each vertex $v$ of $\Ex_\bullet(B)$ and each $i\in I$, all of the vectors $\g(x^v_e):e\in I(v)$ are weakly on the same side of the hyperplane~$\alpha_i^\perp$.

A \newword{cluster monomial} is a monomial in the cluster variables contained in a single cluster.
The \newword{support} of a cluster monomial is the set of cluster variables appearing in the monomial with nonzero exponent.
The $\g$-vector of a cluster monomial is the sum (with multiplicities) of the $\g$-vectors of the cluster variables in its support.

The following is \cite[Conjecture~4.16]{FZ-CDM}.
\begin{conj}\label{mon indep}
The cluster monomials form a linearly independent set.
\end{conj}

For each vertex $v$ in the exchange graph, let $\Cone(v)$ be the cone in $V^*$ spanned by the weights $\g(x^v_e):e\in I(v)$.
Note that we earlier defined $\Cone(v)$ for $v$ a vertex of a framework, and we have now defined $\Cone(v)$ for $v$ a vertex of the exchange graph.
Theorem~\ref{framework principal}(\ref{g vec}) will show that these notations are compatible when they both make sense.
The following is a restatement of \cite[Conjecture~7.10(1)]{ca4}.
(For another restatement, see \cite[Conjecture~1.5]{QP2}.)

\begin{conj}\label{g fan conj}
The collection $\set{\Cone(v):v\in\Ex_\bullet(B)}$ is a fan and the map $v\mapsto\Cone(v)$ is injective.   
\end{conj}

In particular (and as stated in \cite[Conjecture~7.10(1)]{ca4}), different cluster monomials have different $\g$-vectors.
Conjectures~\ref{g fan conj} and~\ref{F 1 conj} imply Conjecture~\ref{mon indep}, as explained in \cite[Remark~7.11]{ca4}.

The assertion that Conjecture~\ref{face conj} and Conjecture~\ref{g fan conj} both hold is equivalent to the following conjecture:
\begin{conj}\label{strong g fan conj}
Suppose two cluster monomials have the same $\g$-vector.
If one is supported on some set $\X$ of cluster variables in a seed, and the other is supported on some set $\X'$ of cluster variables in another seed, then $\X=\X'$, and furthermore, the two seeds are related by a sequence of seed mutations that do not exchange any variables in $\X$.
\end{conj}

Each cluster variable $x$ is a rational function in the initial cluster variables $x_i:i\in I$ and the initial coefficients $y_i:i\in I$.
The \newword{$F$-polynomial} of $x$ is obtained by specializing each $x_i$ to $1$ in this rational function.
Theorem~\ref{Laurent thm} implies that the $F$-polynomial is a polynomial in $y_i:i\in I$ with integer coefficients.
The following are \cite[Conjecture~5.4]{ca4} and \cite[Conjecture~5.5]{ca4}.

\begin{conj}\label{F 1 conj}
Each $F$-polynomial has constant term $1$.
\end{conj}

\begin{conj}\label{F max conj}
Each $F$-polynomial has a unique monomial of maximal degree.
This monomial has coefficient $1$ and is divisible by all other monomials in the $F$-polynomial.
\end{conj}

In light of \cite[Proposition~5.3]{ca4}, Conjecture~\ref{F 1 conj} holds for all $B$ if and only if and Conjecture~\ref{F max conj} holds for all $B$.
More specifically, given a particular $B$, Conjecture~\ref{F 1 conj} holds for $B$ and $-B$ if and only if and Conjecture~\ref{F max conj} holds for $B$ and $-B$.

There are several other formulations of Conjecture~\ref{F 1 conj} listed in \cite[Proposition~5.6]{ca4}.
We will state one of them, which, we will see, corresponds to the Sign condition.
The following condition appears as condition (ii') in the proof of \cite[Proposition~5.6]{ca4}, where it is shown to be equivalent to Conjecture~\ref{F 1 conj}.

\begin{conj}\label{H coherent}
For each vertex $v$ of $\Ex_\bullet(B)$, the rows of $H^v$ are sign-coherent.
In other words, each $\c$-vector has a well-defined sign.
\end{conj}

We add a new conjecture that is suggested by the Cambrian framework constructed in Section~\ref{camb sec}.
We define a map $\nu:V\to V^*$ by setting
\[\nu(\alpha_j)=-\sum_{i\in I}E(\alpha_i\ck,\alpha_j)\rho_i.\]   
When $B$ is acyclic, after suitably ordering $I$, the map $\nu$ is given by the negative of an upper uni-triangular matrix, and therefore it is invertible.
The inverse matrix is easily constructed by a standard combinatorial trick. Define a bilinear form $F$ on $V$ by setting 
\[F(\alpha_i\ck,\alpha_j)=\sum (-E(\alpha_{i_0}\ck,\alpha_{i_1}))(-E(\alpha_{i_1}\ck,\alpha_{i_2}))\cdots(-E(\alpha_{i_{k-1}}\ck,\alpha_{i_k})),\]
where the sum is over all paths $i=i_0 \to i_1 \to \cdots \to i_k=j$ in the complete graph with vertices $I$.
If $k=0$, then the summand is interpreted as $1$.
Since $B$ is acyclic, this is really a finite sum.
Define $\eta:V^*\to V$ by 
\[\eta(\rho_j)=-\sum_{i\in I}F(\alpha_i\ck,\alpha_j)\alpha_j.\]
It is easy to verify that the maps $\eta$ and $\nu$ are inverse to each other.

\begin{conj}\label{nu conj}
If $B$ is acyclic and $x$ is a cluster variable not contained in the initial seed, then $\g(x)=\nu(\d(x))$.
Equivalently, $\d(x)=\eta(\g(x))$.
\end{conj}
As written, the conjecture relates a vector in the weight lattice to a vector in the root lattice.
Equivalently, the conjecture says that the $\g$-vector and the denominator vector, realized as integer vectors, are related by the action of the matrices associated to $-E$ and $-F$.
This conjecture is particularly interesting in connection with \cite[Proposition~7.16]{ca4}, \cite[Conjecture~7.17]{ca4}, and~\cite[Conjecture~6.11]{ca4}.
We will see in Theorem~\ref{nu conj camb} that the conjecture holds when $\Cart(B)$ is of finite type.
It is also easily verified when $n=2$ 
using the recursions \cite[(7.6--7.7)]{ca4} and \cite[Proposition~6.6]{ca4} for $\d$-vectors and $\g$-vectors.

If Conjecture~\ref{nu conj} holds, then Conjectures~\ref{g lattice conj}, \ref{g fan conj}, and~\ref{strong g fan conj} imply the following three conjectures in the case where $B$ is acyclic.
The second of the three is \cite[Conjecture~4.17]{FZ-CDM}, and can be restated in the language of fans, like Conjecture~\ref{g fan conj}.
The first  is only stated in the case of acyclic $B$ in light of a counterexample to the general conjecture given in \cite[Remark~7.7]{ca4}.

\begin{conj}\label{denom lattice conj}
Suppose $B$ is acyclic.
For each vertex $v$ of $\Ex_\bullet(B)$, the vectors $\d(x^v_e):e\in I(v)$ are a $\integers$-basis for the weight lattice. 
\end{conj}

\begin{conj}\label{denom fan conj}
Different cluster monomials have different denominator vectors.
\end{conj}

\begin{conj}\label{strong denom fan conj}
Conjecture~\ref{strong g fan conj} holds with $\g$-vectors replaced by denominator vectors.
\end{conj}

The following table describes the mildest conditions on $B$ under which these conjectures are known and where they are proved, to the best of our knowledge.
We have omitted some results which were established in~\cite{ca4} under extra hypotheses which have since been removed.
See~\cite[Section~13]{ca4} for a similar table prepared at an earlier time.

\begin{center}
\begin{tabular}{|l|l|}\hline
\ref{vertex conj}, \ref{face conj}. \ Note that  these & Finite type:  \cite{ca2}.  \\
\parbox[t]{1.6 in}{  conjectures do not refer to a particular initial seed.}& \parbox[t]{3 in}{\hangindent 0.2 in Acyclic skew-symmetric: \cite{CK2}  proves \ref{vertex conj} explicitly, and the same argument gives \ref{face conj}.} \\\hline
\ref{tildeB equiv}, \ref{H equiv}& Skew-symmetric:  \cite{IIKNS}, using \cite{Pla1,Pla2}. \\\hline
\ref{g lattice conj}, \ref{g sign-coherent},  \ref{g fan conj}, \ref{F 1 conj}, \ref{F max conj}, & Skew-symmetric: \cite{QP2}, see also \cite{CKLP, Fu-Keller, Nagao, Pla1,Pla2}.  \\
 thus implying \ref{mon indep} and   \ref{H coherent} & A class of cases including acyclic: \cite{Demonet}. \\\hline
 \ref{nu conj} & $B$ is acyclic and skew symmetric: Implicit in \cite{CK2}. \\\hline
\ref{strong g fan conj}, \ref{denom lattice conj}, \ref{denom fan conj}, \ref{strong denom fan conj}& $\Cart(B)$ of finite type:  \cite{YZ}, see Remark~\ref{somewhat similar manner}. \\\hline
\end{tabular}
\end{center}
In this paper, we use frameworks to give an independent proof of Conjectures~\ref{vertex conj}--\ref{strong denom fan conj} for $\Cart(B)$ of finite type and principal coefficients.  
Conjectures~\ref{tildeB equiv}, \ref{H equiv}, and~\ref{nu conj} are new for $\Cart(B)$ of finite type, as far as the authors know.  

\subsection{From frameworks to cluster algebras}\label{frame to clus sec}
We now show that every framework is a combinatorial model for the associated cluster algebra.

Let $G$ be a quasi-graph.  
A \newword{covering} of $G$ is a quasi-graph $G'$ and a surjective map $p:G'\to G$ such that if $v$ and $v'$ are connected by an edge in $G'$, then $p(v)$ and $p(v')$ are connected by an edge in $G$, and such that $p$ induces a bijection between full edges incident to $v$ and full edges incident to $p(v)$.
We also require that half-edges of $v$ and half-edges of $p(v)$ are in bijection and require a covering map $p$ to include a specific choice of bijection from half-edges of each $v\in G'$ to half-edges of $p(v)$.

Given an $n$-regular, connected quasi-graph $G$ and a vertex $v_b$ of $G$, we define the \newword{universal cover} $\hG$ of $G$ at $v_b$ and a corresponding covering map $p$.
The vertices of $\hG$ are sequences $u_0,\ldots,u_k$ for $k\ge 0$ such that $u_0=v_b$ and each pair $u_{i-1},u_i$ is connected by an edge $e_i$ of $G$, subject to the restriction that $e_i\neq e_{i+1}$ for $i$ from $1$ to $k-1$.
If $k>0$, then $u_0,\ldots,u_k$  is connected to $u_0,\ldots,u_{k-1}$ by an edge in $\hG$, and these are all of the full edges of $\hG$.
Also, for each half-edge incident to $v_k$ in $G$, there is exactly one half-edge incident to $u_0,\ldots,u_k$ in $\hG$, and these are all of the half-edges of $\hG$.
The covering map $p:\hG\to G$ is the map sending $u_0,\ldots,u_k$ to $u_k$.
For each vertex $u_0,\ldots,u_k$ of $\hG$, we fix any bijection between the half-edges incident to $u_0,\ldots,u_k$ and the half-edges incident to $u_k$, and use these bijections to complete the definition of $p$.  
Thus $p:\hG\to G$ is a covering.
The universal cover $\hG$ has no cycles, and $p$ is the identity map on vertices if and only if $G$ has no cycles.

A framework $(G,C,C\ck)$ lifts to a framework $(\hG,\hat{C},\hat{C}\ck)$ by setting $\hat{C}(v,e)=C(p(v),p(e))$ and $C\ck(v,e)=C\ck(p(v),p(e))$.
We will overload notation by dropping the hats from $\hC$ and $\hC\ck$ and we will refer to $(\hG,C,C\ck)$ as the universal cover of $(G,C,C\ck)$.
The maps $\mu_e:I(v)\to I(v')$ also lift to $\hG$ in the obvious way.

\begin{theorem}\label{framework exchange}
Suppose $(G,C,C\ck)$ is a framework for $B$ and let $\A=\A(B,Y,X)$ be a cluster algebra whose initial exchange matrix is $B$.
Then there exists a covering $v\mapsto\Seed(v)=(B^v,Y^v,X^v)$ from the universal cover $\hG$ to an induced subgraph of the exchange graph $\Ex(B,Y,X)$ of $\A$, sending $v_b$ to the initial seed, such that the exchange matrix $B^v=[b^v_{ef}]_{e,f\in I(v)}$ has $b^v_{ef}=\omega(C\ck(v,e),C(v,f))$.
\end{theorem}
The assertion that $\Seed:\hG\to\Ex(B,Y,X)$ is a covering of a subgraph implies that this labeling of rows and columns of the exchange matrix by edges in $\hG$ makes sense.

\begin{theorem}\label{framework principal}
Suppose $(G,C,C\ck)$ is a framework for $B$ and let $\A_\bullet=\A_\bullet(B)$ be a cluster algebra with principal coefficients whose initial exchange matrix is $B$.
Let $v$ be a vertex of $\hG$.
Then 
\begin{enumerate}
\item \label{extended mat}
For each $e\in I(v$), the $\c$-vector $\c_e^v$ is $C(v,e)$.
\item \label{coef sign-coherent}
Each $\c$-vector $\c_e^v$ has a definite sign: it is either in the nonnegative span of the simple roots or in the nonpositive span of the simple roots.
\item \label{g vec}
If $X^v=(x^v_e:e\in I(v))$ is the cluster in $\Seed(v)$, then for each $e\in I(v$), the $\g$-vector $\g^v_e=\g(x^v_e)$ is $R(v,e)$.
\item \label{weight basis}
The $\g$-vectors $(\g^v_e:e\in I(v))$ are a basis for the weight lattice.
\item \label{F 1}
If $F^v_e$ is the $F$-polynomial of $x^v_e$, then the constant term of $F^v_e$ is $1$. 
\end{enumerate}
\end{theorem}

The first two assertions of Theorem~\ref{framework principal} can be rephrased in terms of the matrix $H^v=[h^v_{ie}]_{i\in I,\,e\in I(v)}$, the bottom half of the extended exchange matrix $\tB^v$ associated to $\Seed(v)$.
Assertion~\eqref{extended mat} is the statement that each $h^v_{ie}$ is $\left[\alpha_i:C(v,e)\right]$, the coefficient of $\alpha_i$ in the simple-root expansion of $C(v,e)$.
Assertion~\eqref{coef sign-coherent} says that the rows of $H^v$ are sign-coherent. 
The third assertion can also be interpreted in similar terms:
If we replace the initial exchange matrix $B$ by $-B^T$ to obtain a matrix $(H')^v$, assertion~\eqref{g vec} says that the rows of the inverse of $(H')^v$ are the $\g$-vectors $(\g^v_e:e\in I(v))$ in $\A_\bullet(B)$.
(See Proposition~\ref{-B^T}.)

We first prove Theorem~\ref{framework exchange}.
For each vertex of $G$, define $B^v=[b^v_{ef}]_{e,f\in I(v)}$ by setting $b^v_{ef}=\omega(C\ck(v,e),C(v,f))$.
The matrix $B^{v_b}$ coincides with the initial exchange matrix $B$, where $I(v_b)$ is identified with $I$ as explained in connection with the Base condition in Section~\ref{frame subsec}.  

\begin{lemma}\label{Bs work}
Suppose $(G,C,C\ck)$ is a framework for $B$.
If $v$ and $v'$ are vertices in $G$, connected by the edge $e$, then the matrices $B^v$ and $B^{v'}$ are related by matrix mutation at $e$, and by applying $\mu_e$ to the row and column indices.
\end{lemma}

\begin{proof}
Let $e$, $e'$ and $e''$ be distinct edges in $I(v)$.
Let $\beta=C(v,e)$, $\gamma=C(v,e')$ and $\delta=C(v,e'')$ and let $\beta'=C(v',\mu_e(e))$, $\gamma'=C(v',\mu_e(e'))$ and $\delta'=C(v',\mu_e(e''))$.
The corresponding co-labels will be denoted by adding $\ck$ to $\beta$, etc.
Since matrix mutation is an involution, and by the Transition condition and the Co-label condition, we may as well take $\sgn(\beta)=\sgn(\beta\ck)=1$.
The proof consists of verifying the following three identities.
\begin{eqnarray}
\label{beta row}
\omega((\beta')\ck,\gamma')&=&-\omega(\beta\ck,\gamma)\\
\label{beta col top}
\omega((\gamma')\ck,\beta')&=&-\omega(\gamma\ck,\beta)\\
\label{top mut}
\omega((\gamma')\ck,\delta')&=&\omega(\gamma\ck,\delta) +\sgn(\omega(\gamma\ck,\beta))\left[\omega(\gamma\ck,\beta),\omega(\beta\ck,\delta)\right]_+
\end{eqnarray}
The Transition condition, with $\sgn(\beta)=1$, says that $\gamma'=\gamma+[\omega(\beta\ck,\gamma)]_+\,\beta$ and similarly the Co-transition condition says that $(\gamma')\ck=\gamma\ck+[-\omega(\gamma\ck,\beta)]_+\,\beta\ck$.
Thus \eqref{beta row} and \eqref{beta col top} follow immediately from the linearity of $\omega$ and the fact that $\omega(\beta\ck,\beta)=0$.
Again using the Transition condition, $\delta'=\delta+[\sgn(\beta)\omega(\beta\ck,\delta)]_+\,\beta$, so 
\[\omega((\gamma')\ck,\delta')=\omega(\gamma\ck,\delta) +\omega(\beta\ck,\delta)[-\omega(\gamma\ck,\beta)]_++\omega(\gamma\ck,\beta)[\omega(\beta\ck,\delta)]_+.\]
It is now trivial to check that \eqref{top mut} holds in all four cases given by $\pm\omega(\beta\ck,\delta)\ge 0$ and $\pm\omega(\gamma\ck,\beta)\le 0$.
\end{proof}

Lemma~\ref{Bs work} has an immediate extension:  In the statement of the lemma, we may replace $(G,C,C\ck)$ by $(\hG,C,C\ck)$.

We next extend $v\mapsto B^v$ to a map from vertices of $\hG$ to seeds.
The seeds are defined recursively in terms of completely general coefficients.
The initial seed is given by the following data: the exchange matrix $B=B^{v_b}$, an $n$-tuple $Y=Y^{v_b}=(y_i:i\in I)$ of elements of the semifield $\PP$, and an $n$-tuple $X=X^{v_b}=(x_i:i\in I)$ of elements of the field $\FF$.
We continue to identify $I$ with $I(v_b)$.

We define, for each vertex $v$ of $\hG$, a tuple $Y^v=(y^v_e:e\in I(v))$ and a tuple $X^v=(x^v_e:e\in I(v))$ and set $\Seed(v)=(B^v,Y^v,X^v)$.
Recall a vertex of $\hG$ is a path $\chi$ in $G$, starting at $v_b$, such that no edge occurs twice consecutively in the path.
Given such a path $\chi$, representing a vertex $v$ of $\hG$, we define $\Seed(v)$ in the obvious way:
If $\chi$ is a single vertex (necessarily $v_b$), then $\Seed(v)$ is the initial seed.
Otherwise, let $e$ be the last edge in the path and let $\chi'$ be the path in $G$ obtained from $\chi$ by deleting the last vertex of the path $\chi$.
Let $v'$ be the vertex of $\hG$ corresponding to $\chi'$.
Then $\Seed(v')$ is defined by induction, and we define $\Seed(v)$ to be the seed obtained by mutating $\Seed(v')$ at the edge $e$.

We have constructed the map $\Seed$ so that it has the property that, for adjacent vertices $v$ and $v'$ in $\hG$, connected by the edge $e$, the seeds $\Seed(v)$ and $\Seed(v')$ are related by mutation at $e$, and by applying $\mu_e$ to the indexing sets.
Together with Lemma~\ref{Bs work}, this completes the proof of Theorem~\ref{framework exchange}.

Now suppose, $\A=\A_\bullet(B)$.
For each vertex of $G$, define $\tB^v=\twomatrix{B^v}{H^v}$, where $B^v$ is the matrix defined before Lemma~\ref{Bs work} and $H^v=[h^v_{ie}]_{i\in I,\,e\in I(v)}$ is defined by setting $h^v_{ie}=[\alpha_i:C(v,e)]$. 
Then $H^{v_b}$ is the identity matrix by the Base condition.

\begin{lemma}\label{tilde Bs work}
Suppose $(G,C,C\ck)$ is a framework for $B$.
Let $v$ and $v'$ be adjacent in $G$, via the edge $e$.
Then $\tB^v$ and $\tB^{v'}$ are related by matrix mutation at $e$, and by applying $\mu_e$ to the column indices and to the indices $I(v)$ of the top $n$ rows.
\end{lemma}
\begin{proof}
Lemma~\ref{Bs work} establishes that the restrictions to top square submatrices are indeed related by matrix mutation.
The remainder of the proof consists of verifying the following two identities, for $i\in I$ and for $\beta$, $\beta'$, $\gamma$, and $\gamma'$ as in the proof of Lemma~\ref{Bs work}.
As in that proof, we will also assume that $\sgn(\beta)=1$.
\begin{eqnarray}
\label{beta col bottom}
\left[\alpha_i:\beta'\right]&=&-\left[\alpha_i:\beta\right]\\
\label{bottom mut}
\left[\alpha_i:\gamma'\right]&=&\left[\alpha_i:\gamma\right] +\sgn(\left[\alpha_i:\beta\right])\left[\left[\alpha_i:\beta\right]\omega(\beta\ck,\gamma)\right]_+
\end{eqnarray}
The identity \eqref{beta col bottom} is immediate by the Transition condition.
Since $\sgn(\beta)=1$, the right side of \eqref{bottom mut} is $\left[\alpha_i:\gamma\right] +[\alpha_i:\beta]\left[\omega(\beta\ck,\gamma)\right]_+$, which, by the Transition condition, is the coefficient of $\alpha_i$ in $\gamma'=\gamma+[\omega(\beta\ck,\gamma)]_+\,\beta$.
\end{proof}

Lemma~\ref{tilde Bs work} proves Theorem~\ref{framework principal}(\ref{extended mat}), and Theorem~\ref{framework principal}(\ref{coef sign-coherent}) follows immediately by the Sign condition.  

We now prove Theorem~\ref{framework principal}(\ref{g vec}).
When $v=v_b$, the vector $R(v_b,i)$ is the fundamental weight dual to $\alpha\ck_i$, for each $i\in I=I(v_b)$, so the assertion holds in this case.
Thus by Theorem~\ref{framework exchange}, we need only check that the assertion at a vertex $v$ implies the assertion at any adjacent vertex $v'$.

Let $v$ and $v'$ be connected by an edge $e$.
Without loss of generality, take $\sgn(C(v,e))=1$.
Using Theorem~\ref{framework exchange} and assuming the assertion for $\g$-vectors at $v$, the $\g$-vector recursion \eqref{gRecurrence} says that $\g^{v'}_{\mu_e(f)}=R(v,\mu_e(f))$ if $f\in I(v')\setminus\set{e}$, and that $\g^{v'}_e$ equals
\begin{equation}
\label{g recur rewrite}
-R(v,e) - \sum_{p\in I(v)}\left[\omega\left(C\ck(v,p),C(v,e)\right)\right]_{-} \, R(v,p)
 + \sum_{i\in I}\left[[\alpha_i:C(v,e)]\right]_{-}\, \b_i.
\end{equation}

By Proposition~\ref{dual adjacent}, $R(v,\mu_e(f))=R(v',f)$.
It remains to show that \eqref{g recur rewrite} is $R(v',e)$ by showing that for every $f\in I(v')$, the pairing of $C\ck(v',f)$ with \eqref{g recur rewrite} equals $\delta_{fe}$.
Since the sign of $C(v,e)$ is $1$, the second sum in \eqref{g recur rewrite} vanishes.
In the first sum, $\omega\left(C\ck(v,p),C(v,e)\right)$ vanishes for $p=e$ by the antisymmetry of $\omega$.
Thus \eqref{g recur rewrite} is
\begin{equation}
\label{g recur easier}
-R(v,e) - \sum_{p\in I(v)\setminus\set{e}}\left[\omega\left(C\ck(v,p),C(v,e)\right)\right]_{-} \,R(v,p)
\end{equation}
The Co-transition condition says that $C\ck(v',e)=-C\ck(v,e)$, so the pairing of  \eqref{g recur easier} with $C\ck(v',e)$ is $\br{-R(v,e),-C\ck(v,e)}=1$.
Suppose $f\in I(v')\setminus\set{e}$.
Since the sign of $C(v,e)$ is $1$, the Co-transition condition says that $C\ck(v',f)$ equals $C\ck(v,\mu_e(f)) - [\omega(C\ck(v,\mu_e(f)), C(v,e))]_{-} \,C\ck(v,e)$.
Thus \eqref{g recur easier} paired with $C\ck(v',f)$ is 
\[[\omega(C\ck(v,\mu_e(f)),C(v,e))]_{-}  - [\omega(C\ck(v,\mu_e(f)),C(v,e))]_{-}=0.\]
We have proved Theorem~\ref{framework principal}(\ref{g vec}).

Theorem~\ref{framework principal}(\ref{weight basis}) follows from Theorem~\ref{framework principal}(\ref{g vec}) and Proposition~\ref{basis} because the weight lattice is the dual lattice to the co-root lattice.
Theorem~\ref{framework principal}(\ref{F 1}) follows from Theorem~\ref{framework principal}(\ref{coef sign-coherent}) by the argument given in the proof of \cite[Proposition~5.6]{ca4}, which shows that Conjectures~\ref{F 1 conj} and~\ref{H coherent} are equivalent.
We must be careful, because \cite[Proposition~5.6]{ca4} states that one conjecture, for all $B$, is equivalent to the other conjecture, for all $B$.
Theorems~\ref{framework principal}(\ref{F 1}) and Theorem~\ref{framework principal}(\ref{coef sign-coherent}) refer to a specific $B$, and perhaps only to part of the exchange graph.
However, the proof of \cite[Proposition~5.6]{ca4} does not require $B$ to vary, and argues by induction on distance, in the exchange graph, to the initial seed.
Since $G$ is connected by hypothesis, the argument goes through.
This completes the proof of Theorem~\ref{framework principal}.

Barot, Geiss and Zelevinsky~\cite{capsm} gave a description of seeds of finite type using the concept of quasi-Cartan companions: A \newword{quasi-Cartan companion} to an exchange matrix $B=[b_{ij}]$ is a symmetrizable matrix $A=[a_{ij}]$ with $a_{ii}=2$ for all $i$ and $|a_{ij}|=|b_{ij}|$ for all distinct $i$ and $j$.
In particular, Barot, Geiss and Zelevinsky showed that a seed which is mutation equivalent to an orientation of a Dynkin diagram has a positive definite quasi-Cartan companion. 
We will now explain how quasi-Cartan companions occur in our setting.

Suppose $B$ is an acyclic exchange matrix.
Theorem~\ref{ST frame} and the definition of a reflection framework imply that for each vertex $v$ in the principal-coefficients exchange graph $\Ex_\bullet(B)$, the $\c$-vectors $\c_e^v$ are roots.
Define $A^v$ to be the matrix indexed by $I(v)$ whose $ef$-entry is $K((\c_e^v)\ck,\c_f^v)$.

\begin{cor}\label{q-C cor}
Suppose $B$ is an acyclic exchange matrix.
For each vertex $v$ in the principal-coefficients exchange graph $\Ex_\bullet(B)$, the matrix $A^v$ is a quasi-Cartan companion of $B^v$.
\end{cor}
\begin{proof}  
Theorem~\ref{ST frame} says that a complete reflection framework $(G,C)$ exists for~$B$.
Recall that a vertex $v$ in $\Ex_\bullet(B)$ is an equivalence class of seeds.
If $v'$ is some vertex of $G$ such that $\Seed(v')$ is in the equivalence class $v$, then Theorem~\ref{framework exchange} says that $B^v=[\omega(C\ck(v,e),C(v,f))]_{e,f\in I(v)}$.
By Theorem~\ref{framework principal}\eqref{extended mat}, the matrix $A^v$ is $[K(C\ck(v,e),C(v,f))]_{e,f\in I(v)}$.
As was already remarked, condition (E0) on a reflection framework combines with Proposition~\ref{sym antisym} to imply that $\omega(C\ck(v,e),C(v,f))$ and $K(C\ck(v,e),C(v,f))$ agree in absolute value for $e\neq f$.
Also, $C\ck(v,e)$ is the co-root associated to $C(v,e)$, so $K(C\ck(v,e),C(v,e))=2$.
\end{proof}

\subsection{From cluster algebras to frameworks}\label{clus to frame sec}  
We now show that, assuming Conjecture~\ref{H coherent}, every exchange matrix $B$ has a framework.
The point is to validate the notion of a framework by showing that, if cluster algebras behave as we expect them to, frameworks are unavoidable.
In the process, we establish some interesting statements about cluster algebras.

Let $B$ be an exchange matrix.  
Let $T$ be the $n$-regular tree considered in Section~\ref{ca background sec} and recall the maps $\mu_e$ defined for each edge $e$.
Let $v\mapsto\tB^v$ be the map that associates to each vertex of $T$ an extended exchange matrix, with exchange matrix $B$ and principal coefficients at the base vertex. 
Define a labeling $H$ of incident pairs in $T$ by taking $H(v,e)$ to be the $\c$-vector $\c_e^v$ (the vector whose simple root coordinates are given by the bottom $n$ entries of the column of $\tB^v$ labeled by $e$).
Similarly, we consider a different map $v\mapsto(\tB')^v$ that associates to each vertex of $T$ an extended exchange matrix, with exchange matrix $-B^T$ and principal coefficients at the base vertex.
Define a co-labeling $H\ck$ by taking $H\ck(v,e)$ to be the vector whose simple co-root coordinates are given by the bottom $n$ entries of the column of $(\tB')^v$ labeled by $e$.
(Recall from the discussion immediately following Proposition~\ref{-B^T} that the simple co-roots associated to $\Cart(B)$ are the simple roots associated to $\Cart(B)^T=\Cart(-B^T)$.)  
\begin{theorem}\label{T model}
Suppose Conjecture~\ref{H coherent} holds for $B$.
Then the triple $(T,H,H\ck)$ is a framework for $B$.
\end{theorem}
\begin{proof}
The hypothesis that Conjecture~\ref{H coherent} holds for $B$ is exactly the statement that the Sign condition holds.
Consider the following conditions:
\begin{enumerate}
\item[(i)] $H\ck(v,e)$ is a positive scalar multiple of $H(v,e)$ for all $e\in I(v)$.
\item[(ii)] The exchange matrix $B^v$ has entries $\omega(H\ck(v,e),H(v,f))$ for $e,f\in I(v)$.
\item[(iii)] The exchange matrix $(B')^v$ has entries $\omega(H(v,e),H\ck(v,f))$ for $e,f\in I(v)$.
\end{enumerate}
These conditions hold at the base vertex.
Suppose now they hold at a vertex $v$ and suppose that $v'$ is an adjacent vertex.
Given that (ii) holds at $v$, the proof of Lemma~\ref{tilde Bs work} establishes that the (strengthened) Transition condition holds for $v$ and $v'$.
Similarly, given that (iii) holds at $v$, the proof of Lemma~\ref{tilde Bs work} establishes that the (strengthened) Co-transition condition holds for $v$ and $v'$.
Now the proof of Proposition~\ref{cotrans restate} is easily modified to show that condition (i) holds at $v'$ as well.
Thus (i) holds at every vertex, or in other words, the Co-label condition holds.
We have also established the Transition and Co-transition conditions, and the Base condition is immediate.
\end{proof}

We establish several corollaries.
First, combining Proposition~\ref{-B^T} with Theorems~\ref{framework principal} and~\ref{T model}, we obtain the following result, which is the first assertion of \cite[Theorem~1.2]{NZ}.
Let $v$ be any vertex in $T$.
Let $G^v$ be the matrix whose rows are the fundamental weight coordinates of the $\g$-vectors $\g_e^v$ for $e\in I(v)$.
As before, $H^v$ is the bottom half of $\tB^v$, the principal-coefficients extended exchange matrix at $v$.
The following corollary is about $\c$-vectors, although it is more convenient to phrase it in terms of matrices $H^v$.

\begin{cor}\label{NZ cor}
Suppose Conjecture~\ref{H coherent} holds for $B$ and let $v$ be any vertex of $T$.
Then the matrix $G^v$, calculated with initial exchange matrix $B$ at $v_b$, is the inverse to the matrix $H^v$, calculated with initial exchange matrix $-B^T$ at $v_b$.
\end{cor}
Naturally, the corollary assumes that we have chosen the same linear order on $I(v)$ to write both matrices.

\begin{cor}\label{fan coincide}
Suppose Conjecture~\ref{H coherent} holds for $B$ and suppose Conjecture~\ref{g fan conj} holds for $B$ and for $-B^T$.
Then the fans associated to $B$ and $-B^T$ coincide.
\end{cor}
\begin{proof} 
By Theorem~\ref{framework principal}(\ref{g vec}) and Theorem~\ref{T model},  the cones of the fan associated to $B$ are exactly the cones defined by the label sets $H(v)$ for vertices $v$ of $T$.
But Proposition~\ref{-B^T}, together with the same two theorems, implies that the cones of the fan associated to $-B^T$ are exactly the cones defined by the label sets $H\ck(v)$ for vertices $v$ of $T$.
The Co-label condition says that these are the same cones.
\end{proof}

\begin{remark}\label{coincide caveat}
Recall that we are defining $\g$-vectors as elements of the weight lattice, rather than as integer vectors.
(See Remark~\ref{gvec weight or int vec}.)
Corollary~\ref{fan coincide} becomes false if we interpret $\g$-vectors in terms of the wrong basis for $V^*$.
Importantly, the basis changes when we pass from $B$ to $-B^T$.
(See also Remark~\ref{gvec caveat}.)
\end{remark}

\begin{cor}\label{fan H}
Suppose Conjectures~\ref{g fan conj} and~\ref{H coherent} hold for $B$.
Then Conjecture~\ref{H equiv} holds for $B$.  
\end{cor}
\begin{proof}
Suppose $u$ and $v$ are vertices of $T$ with $H^u=H^v$.
Then Theorem~\ref{framework principal}(\ref{g vec}) and Theorem~\ref{T model} imply that $\Cone(u)=\Cone(v)$.
Now Conjecture~\ref{g fan conj} implies that $u=v$.
\end{proof}

Recall that $\Ex_\bullet(B)$ is the exchange graph of the principal-coefficients cluster algebra associated to $B$.
\begin{cor}\label{identity model}
Suppose Conjecture~\ref{H coherent} holds for $B$.
Suppose either that $B$ is skew-symmetric or that Conjecture~\ref{g fan conj} holds for $B$ and $-B^T$.
Then the triple $(\Ex_\bullet(B),H,H\ck)$ is a framework for $B$.
\end{cor}
\begin{proof}
We first show that $\Ex_\bullet(B)$ and $\Ex_\bullet(-B^T)$ are identical as quotients of the $n$-regular graph $T$.
This is a tautology if $B$ is skew-symmetric, so we need only consider the case where Conjecture~\ref{g fan conj} holds for $B$ and $-B^T$.
We need to show that, for any two vertices $u$ and $v$ of $T$, we have $\tB^u$ equivalent to $\tB^v$ if and only if $(\tB')^u$ is equivalent to $(\tB')^v$.
By symmetry, we need only check one direction.
Suppose $\tB^u$ is equivalent to $\tB^v$, so that the cones defined by $H(u)$ and $H(v)$ coincide.
Then the Co-label condition implies that the cones defined by $H\ck(u)$ and $H\ck(v)$ coincide.
Since Conjecture~\ref{g fan conj} holds for $-B^T$, we conclude that $(\tB')^u$ is equivalent to $(\tB')^v$.

Since $\Ex_\bullet(B)$ and $\Ex_\bullet(-B^T)$ are identical as quotients, the quotient inherits the labeling $H$ and the co-labeling $H\ck$.
The triple $(\Ex_\bullet(B),H,H\ck)$ inherits the Co-label, Sign, Base, Transition, and Co-transition conditions from $(T,H,H\ck)$.
\end{proof}

\section{Global conditions on frameworks}\label{global sec}
All of the conditions defining a framework are local.
In this section, we consider some global conditions on a framework and show how the existence of a framework for $B$ with various global properties establishes, for $B$, various conjectures from Section~\ref{conj sec}.
All the results in this section apply to general frameworks, whether or not they are reflection frameworks.

\subsection{Complete, exact, and well-connected frameworks}\label{nice frame sec}
We say that a framework $(G,C,C\ck)$ is \newword{complete} if $G$ has no half-edges. 
This is a local condition, but it is convenient to discuss it together with the other global conditions discussed in this section.
The universal cover of a complete framework is the $n$-regular tree, and thus it is easily seen that the map $\Seed$ must be surjective when the framework is complete.
Thus Theorem~\ref{framework principal} implies the following theorem.
\begin{theorem}\label{complete conj}
If a complete framework exists for $B$, then Conjectures~\ref{g lattice conj}, \ref{F 1 conj}, and~\ref{H coherent} all hold for $B$. 
If in addition a complete framework exists for $-B$, then Conjecture~\ref{F max conj} also holds for $B$. 
\end{theorem}
The assertion about Conjecture~\ref{F max conj} holds by the relationship between Conjectures~\ref{F 1 conj} and~\ref{F max conj} explained in Section~\ref{conj sec}.

A framework is \newword{injective} \label{injective} if, for every pair $u,v$ of vertices of $G$, the following three conditions are equivalent:
$C(u)=C(v)$; $C\ck(u)=C\ck(v)$; and $u=v$.

Consider the special case of Theorem~\ref{framework exchange} where $(B,Y,X)$ is a seed with principal coefficients, so that $\A(B,Y,X)$ is $\A_\bullet(B)$.
Given a framework $(G,C,C\ck)$, Theorem~\ref{framework exchange} asserts the existence of a map $v\mapsto\Seed(v)$ from the universal cover $\hG$ to $\Ex_\bullet(B)$.
The framework $(G,C,C\ck)$ is \newword{ample} \label{ample} if $\Seed$ factors through the covering map $\hG \to G$.
A framework is \newword{exact} \label{exact}if it is injective and ample.
Note that exactness and ampleness are not purely combinatorial, but depend on the framework's interaction with a cluster algebra.

\begin{theorem}\label{exact framework}
Suppose $(G,C,C\ck)$ is an exact framework for $B$ and let $\A_\bullet(B)$ be the cluster algebra with principal coefficients whose initial exchange matrix is $B$.
Then the map $v\mapsto\Seed(v)=(\tB^v,X^v)$ is a graph isomorphism from $G$ (ignoring half-edges) to its image, a subgraph of the exchange graph $\Ex_\bullet(B)$.
If $(G,C,C\ck)$ is also complete, then the map is a graph isomorphism from $G$ to $\Ex_\bullet(B)$.
\end{theorem}
\begin{proof}
Since $(G,C,C\ck)$ is ample, the map $v\mapsto\Seed(v)$ descends to a map $v\mapsto\Seed(v)$ from $G$ to $\Ex_\bullet(B)$.
We will show that the map is an isomorphism to its image.
If $u$ and $v$ are vertices of $G$ with $\Seed(u)=\Seed(v)$, then Theorem~\ref{framework principal}(\ref{extended mat}) implies that $C(u)=C(v)$.
Since $(G,C,C\ck)$ is an injective framework, we conclude that $v\mapsto\Seed(v)$ is one-to-one.
Now $G$ is $n$-regular and $\Seed$ maps the neighbors of each vertex $v$ to distinct neighbors of $\Seed(v)$ in $\Ex_\bullet(B)$.
Since $\Ex_\bullet(B)$ is also $n$-regular, an easy proof by induction on the distance from the initial seed shows that $v\mapsto\Seed(v)$ is an isomorphism to its image.
If $G$ has no half-edges, then the image of $\Seed$ is an $n$-regular subgraph of the $n$-regular graph $\Ex_\bullet(B)$.
Since $\Ex_\bullet(B)$ is connected, the image is all of $\Ex_\bullet(B)$.
\end{proof}

\begin{remark}\label{framework uniqueness}
Theorem~\ref{exact framework} implies that, up to isomorphism of $G$, there is at most one complete, exact framework for a given $B$.
Furthermore, non-complete, exact frameworks coincide where they overlap, in a sense that can be made precise.
However, in this circumstance, the phrase ``up to isomorphism'' allows some meaningful freedom.
Making a useful framework means choosing an appropriate combinatorial, algebraic, or geometric realization of the triple $(G,C,C\ck)$. 
\end{remark}

Theorem~\ref{exact framework} combines with Theorems~\ref{framework principal}(\ref{extended mat}) to give the following corollary.
\begin{cor}\label{exact conj}
If a complete, exact framework exists for $B$, then Conjecture~\ref{H equiv} holds for $B$. 
\end{cor}

The framework $(G,C,C\ck)$ is \newword{polyhedral} \label{polyhedral} if the cones $\Cone(v)$, where $v$ ranges over all vertices of $G$, are the maximal cones of a fan with the property that  distinct vertices $v$ of $G$ define distinct cones $\Cone(v)$.
The fan in question is always \newword{simplicial}, meaning that it is composed of simplicial cones.
Theorem~\ref{framework principal}(\ref{g vec}) also implies that a polyhedral framework is automatically injective.  
Thus a polyhedral framework is exact if and only if it is ample. 

A polyhedral framework is \newword{well-connected} \label{well-connected} if it has the following property:
If $F$ is a face of $\Cone(v)$ and of $\Cone(v')$, then there exists a path $v=v_0$, $v_1$, \dots, $v_k=v'$ in $G$ such that $F$ is a face of $\Cone(v_i)$  for all $i$ from $0$ to $k$.
The following theorem is the reason for considering well-connected polyhedral frameworks.
It is immediate from Theorem~\ref{framework principal}(\ref{g vec}).

\begin{theorem}\label{well-connected g}
If $(G,C,C\ck)$ is a polyhedral framework, then $\g$-vector cones for seeds in the image of $v\mapsto\Seed(v)$ form a fan, identical to the fan defined by $(G,C,C\ck)$.
If $(G,C,C\ck)$ is also well-connected, then Conjecture~\ref{strong g fan conj} holds for pairs of cluster monomials supported on clusters in the image of $v\mapsto\Seed(v)$.
\end{theorem}

\begin{cor}\label{polyhedral conj}
If a complete, exact, well-connected polyhedral framework exists for $B$, then Conjectures~\ref{vertex conj}, \ref{face conj}, \ref{mon indep} (for principal coefficients) and Conjectures~\ref{g fan conj} and~\ref{strong g fan conj} all hold for $B$. 
Furthermore, the fan defined by the framework is identical to the fan defined by $\g$-vectors of clusters in $\A_\bullet(B)$. 
\end{cor}
\begin{proof}
Conjecture~\ref{strong g fan conj} follows from Theorem~\ref{well-connected g} because the completeness of the framework implies that the image of $\Seed$ is the whole exchange graph.
As discussed in Section~\ref{conj sec}, Conjecture~\ref{strong g fan conj} implies the other conjectures except for Conjecture~\ref{mon indep}.
Conjecture~\ref{F 1 conj} holds by Theorem~\ref{complete conj}.  
Conjectures~\ref{g fan conj} and~\ref{F 1 conj} imply Conjecture~\ref{mon indep}, as explained in Section~\ref{conj sec}.
\end{proof}

We can expand on Corollary~\ref{identity model} to include global conditions.
\begin{theorem}\label{exact identity model}
Suppose Conjecture~\ref{H coherent} holds for $B$.
If Conjecture~\ref{g fan conj} holds for $B$ and $-B^T$, then the triple $(\Ex_\bullet(B),H,H\ck)$ is a complete, exact polyhedral framework for $B$.
It is well-connected if and only if Conjecture~\ref{strong g fan conj} holds for $B$.
\end{theorem}
\begin{proof}
By Corollary~\ref{identity model}, $(\Ex_\bullet(B),H,H\ck)$ is a framework, and it is complete because the exchange graph has no half-edges.
Conjecture~\ref{g fan conj} for $B$ and $-B^T$ and Theorem~\ref{framework principal} imply Conjecture~\ref{H equiv} for $B$ and $-B^T$.
Thus $(\Ex_\bullet(B),H,H\ck)$ is injective.
It is ample because the map $\Seed:\widehat{\Ex}_\bullet(B)\to\Ex_\bullet(B)$
descends to the identity map on $\Ex_\bullet(B)$.
The statement about well-connectedness is immediate.
\end{proof}

Combining Theorem~\ref{T model} and/or Corollary~\ref{identity model} and Theorem~\ref{exact identity model} with Theorem~\ref{complete conj} and Corollaries~\ref{exact conj} and~\ref{polyhedral conj}, we obtain some additional dependencies among conjectures. 

Ampleness is a difficult condition to establish. 
The easiest way is to know Conjecture~\ref{H equiv} or~\ref{g fan conj} in advance:

\begin{prop}\label{conj imply ample}
Suppose $(G,C,C\ck)$ is a framework for $B$.
If Conjecture~\ref{H equiv} or Conjecture~\ref{g fan conj} holds for $B$, then $(G,C,C\ck)$ is ample.
\end{prop}
\begin{proof}
Consider the map $\Seed$ from $\hG$ to the principal-coefficients exchange graph $\Ex_\bullet(B)$.
If $v$ is the vertex $(v_0,\ldots,v_k)$ of $\hG$, then Theorem~\ref{framework principal}(\ref{extended mat}) implies that the extended exchange matrix $\tB^v$ depends only on $v_k$.
Thus if Conjecture~\ref{H equiv} holds for $B$, then $\Seed(v)$ depends only on $v_k$.
Furthermore, Theorem~\ref{framework principal}(\ref{g vec}) implies that $\g$-vectors of the cluster $X^v$ depend only on $v_k$.
Thus if Conjecture~\ref{g fan conj} holds for $B$, then $\Seed(v)$ depends only on $v_k$.
\end{proof}

The following sections discuss other ways to prove ampleness.

\subsection{Simply connected frameworks}\label{simply sec}
We now define the notion of a simply connected framework, and show that simple connectivity implies ampleness.
The ideas here are similar in spirit to the ideas in \cite[Section~2]{ca2}.
The definition of simple connectivity requires much preparation, beginning with the definition of a \newword{rank-two cycle}.

Suppose $(G,C,C\ck)$ is a framework, let $v$ be a vertex of $G$ and let $e$ and $f$ be edges incident to $v$.
Construct a doubly infinite sequence of vertices and edges as follows:
Set $v_0=v$, $e_0=f$ and $e_1=e$.
Then, recursively, let $e_{k+1} = \mu_{e_k}(e_{k-1})$ and let $e_k$ join $v_{k-1}$ and $v_k$, as shown in Figure~\ref{vi ei fig}.
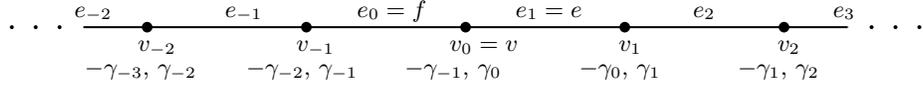
\begin{figure}
\includegraphics{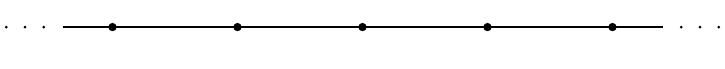}
\begin{picture}(0,0)(174,-20)
\put(-128,-9){\small$v_{-2}$}
\put(-148,-18){\small$-\gamma_{-3}$, $\gamma_{-2}$}
\put(-68,-9){\small$v_{-1}$}
\put(-87,-18){\small$-\gamma_{-2}$, $\gamma_{-1}$}
\put(-10,-9){\small$v_0=v$}
\put(-27,-18){\small$-\gamma_{-1}$, $\gamma_{0}$}
\put(54,-9){\small$v_1$}
\put(39,-18){\small$-\gamma_{0}$, $\gamma_{1}$}
\put(114,-9){\small$v_2$}
\put(99,-18){\small$-\gamma_{1}$, $\gamma_{2}$}
\put(-152,4){\small$e_{-2}$}
\put(-95,4){\small$e_{-1}$}
\put(-45,4){\small$e_0=f$}
\put(15,4){\small$e_1=e$}
\put(82,4){\small$e_2$}
\put(135,4){\small$e_3$}
\end{picture}
\caption{The path with vertices $v_i$, and the relevant elements of each $C(v_i)$}
\label{vi ei fig}
\end{figure}
Assume, for the rest of this section, that none of the $e_k$ are half-edges, so that this recursion is well-defined.
If the set of distinct edges in this sequence is finite, then the edges define a cycle in $G$ which we will call a \newword{rank-two cycle}.  \label{rank two cycle}

Let $\gamma_k = C(v_k, e_{k+1})$ and define $\gamma\ck_k$ similarly. 
The Transition and Co-transition conditions imply that $C(v_k, e_k) = - \gamma_{k-1}$, that $C\ck(v_k, e_k) = - \gamma\ck_{k-1}$, and that
\begin{equation}
\begin{aligned}
\gamma_{k+1} &= - \gamma_{k-1} - [\sgn(\gamma_k) \omega(\gamma_k^{\vee}, \gamma_{k-1})]_{-} \gamma_k \\
\gamma^{\vee}_{k+1} &= - \gamma^{\vee}_{k-1} + [\sgn(\gamma_k^{\vee}) \omega(\gamma_{k-1}^{\vee}, \gamma_{k})]_{+} \gamma^{\vee}_k
\end{aligned}\label{GammaRecur}
\end{equation}

\begin{lemma} \label{lem:negsomewhere}
Let $(G, C, C^{\vee})$ be a framework, let $v$ be a vertex and $e$ and $f$ two edges adjacent to $v$ as above. Define $v_k$, $e_k$ and $\gamma_k$ as above.

Suppose that $v_k$ is periodic with some finite period. Then there is some $k$ such that $\sgn(\gamma_k)=\sgn(-\gamma_{k-1})=-1$.  
\end{lemma}

\begin{proof}  
Since $v_k$ is periodic, the sequence of cones $\Cone(v_k)$ is periodic. 
By Proposition~\ref{dual adjacent}, $R(v_{k-1})\cap R(v_k)$ is a set of $n-1$ vectors, with $R(v_{k-1})$ and $R(v_k)$ differing in the vectors $R(v_{k-1},e_k)$ and $R(v_k,e_k)$.
Since $e_{k+1}=\mu_{e_k}(e_{k-1})$, the set $R(v_k)\setminus\set{R(v_k,e_k),R(v_k,e_{k+1})}$ is independent of $k$.
The nonnegative span of $R(v_k)\setminus\set{R(v_k,e_k),R(v_k,e_{k+1})}$ is a codimension-$2$ face $F$ of $\Cone(v_k)$.
All of the cones $\Cone(v_k)$ share the face $F$ and wind in cyclic order around $F$. 
Let $U_k$ be the cone consisting of points $x \in V^*$ with $\br{x,\gamma_k\ck}\ge 0$ and $\br{x,-\gamma_{k-1}\ck}\ge 0$.   
The cone $U_k$ is full-dimensional and has exactly two codimension-$1$ faces, defined by $\gamma_k^\perp$ and $\gamma_{k-1}^\perp$, both containing~$F$.

Since the $\Cone(v_k)$ repeat periodically, the $U_k$ do as well, winding cyclically around $\mathrm{Span}_{\RR}(F)$. So $\bigcup U_k = V^*$. 
Moreover, all of the vectors $\gamma_k\ck$ satisfy the Sign condition, so none of the hyperplanes $\gamma_k^\perp$ intersect the interior of the negative fundamental domain $-D$.
Putting these observations together, there is some $k$ such that $-D \subset U_k$. Since $\langle \gamma, \ \rangle$ is positive on $-D$ if and only if $\sgn(\gamma)=-1$, we see that $\sgn(\gamma_k)=\sgn(-\gamma_{k-1})=-1$ for this $k$.
\end{proof}

From Equations~\eqref{GammaRecur} we see that 
$$\omega(\gamma_{k-1}^{\vee}, \gamma_k) = - \omega(\gamma^{\vee}_{k+1}, \gamma_k) \ \mbox{and} \ \omega(\gamma_k^{\vee}, \gamma_{k-1}) = - \omega(\gamma^{\vee}_k, \gamma_{k+1}).$$
So, if $b = \omega(\gamma_{-1}^{\vee}, \gamma_0)$ and $b' = \omega(\gamma_0^{\vee}, \gamma_{-1})$, then the values of $\omega(\gamma_{k}\ck, \gamma_{k+1})$ alternate $b$, $-b'$, $b$, $-b'$ et cetera, and the values of $\omega(\gamma\ck_{k+1}, \gamma_k)$ alternate $b'$, $-b$, $b'$, $-b$.
Note that $b$ and $b'$ have opposite signs, and one is $0$ if and only if the other is.
Switching $e$ and $f$ if necessary, we may assume without loss of generality that $b \geq 0$ and $b' \leq 0$.

Motivated by Lemma~\ref{lem:negsomewhere}, we assume, for the moment, that $\sgn(C(v, e))=\sgn(C(v,f))=-1$. 
Recall that $C(v,e) = \gamma_0$ and $C(v,f) = - \gamma_{-1}$.
Let $\tilde{V}$ be the two dimensional vector space spanned by $\gamma_{-1}$ and $\gamma_0$.
Let $\tilde{A}$ be the Cartan matrix $\begin{psmallmatrix*}[r] 2 & b' \\ -b & 2 \end{psmallmatrix*}$.
Let $\tilde{\Phi}$ be the root system in $\tilde{V}$ defined by $\tilde{A}$ with negative simple roots $-\gamma_{-1}$ and $\gamma_0$, and negative simple co-roots $-\gamma_{-1}\ck$ and $\gamma_0\ck$.
This defines a symmetric bilinear form $\tilde{K}$ on $\tilde{V}$ with $\tilde{K}(\gamma_{-1}\ck, \gamma_0) = b$ and  $\tilde{K}(\gamma_0\ck, \gamma_{-1}) = -b'$.

\begin{lemma} \label{lem:rk2roots}
Suppose that $C(v,e)$ and $C(v,f)$ are negative roots and make the definitions of the above paragraph, including assuming that $b \geq 0$ and $b' \leq 0$ as discussed above. 
Then the vectors $\{\gamma_k \}_{k \in \ZZ}$ are the almost positive roots (see Section~\ref{ref frame subsec}) of $\tilde{\Phi}$, and $\{\gamma^{\vee}_k \}_{k \in \ZZ}$ are the almost positive co-roots, occurring in cyclic order.
More specifically, $\gamma_0$ and $\gamma_1=-\gamma_{-1}$ are the negative simple roots of $\tilde{\Phi}$.  
\end{lemma}

\begin{remark}\label{tildePhi clarification}  
The root system $\tilde\Phi$ is not necessarily the same as the sub root system $\Phi\cap\Span_\RR\set{\gamma_{-1},\gamma_0}$, because $-|b|$ is not necessarily equal to $K(\gamma_{-1}^{\vee}, \gamma_0)$ and $-|b'|$ is not necessarily equal to $K(\gamma_0^{\vee}, \gamma_{-1})$.
\end{remark}

\begin{remark}
If $b \leq 0$ and $b' \geq 0$, then $\{\gamma_k \}_{k \in \ZZ}$ is the negatives of the almost positive roots, with $\gamma_{-2}$ and $\gamma_{-1}$ being the positive simple roots.
\end{remark}

\begin{proof}[Proof Sketch for Lemma~\ref{lem:rk2roots}]  
This is a straightforward computation, using the recursions~\eqref{GammaRecur}.
We have $\gamma_1 = - \gamma_{-1}$ and then 
$$\gamma_{k+1} = -\gamma_{k-1} + (b \ \textrm{or} \ (-b')) \gamma_{k}$$
as long as $\gamma_k$ is a positive root, where the coefficient of $\gamma_k$ alternates between $b$ and $-b'$ based on the parity of $k$.
This recursion marches through the positive roots of $\tilde{\Phi}$ in order. If the matrix $\begin{psmallmatrix*}[r] 2 & b' \\ -b & 2 \end{psmallmatrix*}$ is of infinite type, then the $\gamma_k$ will always be positive roots.
If this matrix is of finite type, the recursion takes on all the positive roots, then becomes $\gamma_0$ and $\gamma_1$ again, and then repeats the positive roots, indefinitely.
So, for $k >0$, the recursion travels through the almost positive roots in order.

A similar analysis applies for $k<0$.
\end{proof}

\begin{cor} \label{cor:Periodic}
Let $(G, C, C^{\vee})$ be a framework. Let $v$ be a vertex of $G$ with adjacent vertices $e$ and $f$; define $v_k$, $e_k$, $\gamma_k$, $b$ and $b'$ as above. Suppose that the $v_k$ form a finite cycle, of length $\ell$.
Then  $\begin{psmallmatrix} 2 & -|b'| \\ -|b| & 2 \end{psmallmatrix}$ is a Cartan matrix of finite type. 
Moreover, $\ell$ is divisible by $h+2$, where $h$ is the Coxeter number of this Cartan matrix.
\end{cor}

\begin{proof}
By Lemma~\ref{lem:negsomewhere}, we can reindex the cycle of $v_k$'s so that  $\gamma_k$ and $-\gamma_{k-1}$ are negative roots.
Switching $e$ and $f$ if necessary, we may assume that $b \geq 0$ and $b' \leq 0$.
By Lemma~\ref{lem:rk2roots}, the $\gamma_k$ go through the almost positive roots of $\tilde{\Phi}$ in circular order.
So, since $v_k$ repeats, this shows that the root system $\tilde{\Phi}$ is finite, so $\begin{psmallmatrix*}[r] 2 & b' \\ -b & 2 \end{psmallmatrix*}$ is of finite type.
\end{proof}

\begin{remark}
We take the opportunity to remind the reader that $\begin{psmallmatrix} 2 & -|b'| \\ -|b| & 2 \end{psmallmatrix}$ is of finite type if and only if $|bb'| < 4$, and that $h$ is $2$, $3$, $4$ or $6$ according to whether $|bb'|$ is $0$, $1$, $2$ or $3$.
\end{remark}

We now prove the key lemmas of this section.

\begin{lemma} \label{lem:rank2}
Let $(G, C, C^{\vee})$ be a framework. 
Let $v$ be a vertex of $G$ with adjacent vertices $e$ and $f$ and define $v_k$, $e_k$ as above. 
Suppose that the $v_k$ form a finite cycle of length $m$. 
Let $\tilde{v}_0,\tilde{v}_1,\ldots$ be an infinite path in $\hG$ lying above the infinite path $v_0,v_1,\ldots$ in $G$.
Then $\Seed(\tilde{v}_0) = \Seed(\tilde{v}_m)$.
\end{lemma}

\begin{proof}
Using Lemma~\ref{Bs work} to translate \cite[Theorem~6.2 and~7.7]{ca1} into the language of this paper, we obtain the statement that the sequence $\Seed(\tilde{v}_k)$ is periodic with period $h+2$, where $h$ is the Coxeter number appearing in Corollary~\ref{cor:Periodic}.
This, combined with Corollary~\ref{cor:Periodic} implies the lemma.
\end{proof}

\begin{lemma} \label{lem:rank2variant}
Let $(G, C, C^{\vee})$ be a framework. Let $v$ be a vertex of $G$ with adjacent vertices $e$ and $f$ with $C(v,e)$ and $C(v,f)$ negative. Define $v_k$, $b$ and $b'$ as above.
Suppose that $\begin{psmallmatrix} 2 & -|b'| \\ -|b| & 2 \end{psmallmatrix}$ is of finite type, with Coxeter number $h$.
Then $C(v_k)$ and $C^{\vee}(v_k)$ repeat with period $h+2$.
\end{lemma}

\begin{proof}
We have already seen that the $\gamma_k$ repeat in this manner. 
Let $\delta$ be an element of $C(v_0)$ other that $\gamma_0$ and $-\gamma_{-1}$. Let $\delta_k$ be the element of $C(v_k)$ that is obtained by repeatedly applying the Transition condition to $\delta$. For shorthand, set $\alpha = C(v,e)$ and $\beta = C(v,f)$.

\textbf{Case 1:} $\tilde{\Phi}$ is of type $A_1 \times A_1$, so $\omega(\alpha^{\vee}, \beta)=0$. Then 
$$\delta_4 = \delta + [\omega(\alpha^{\vee}, \delta)]_{+} (\alpha) + [\omega(\beta\ck, \delta)] (\beta) +  [\omega(\alpha^{\vee}, \delta)]_{+} (-\alpha) + [\omega(\beta^{\vee}, \delta)] (-\beta) =\delta$$
as desired.

\textbf{Case 2:} $\tilde{\Phi}$ is not of type $A_1 \times A_1$, so the restriction of $\omega$ to $\tilde{V}$ is nondegenerate.
Then we can write $\delta = \nu + \kappa$ with $\kappa \in \tilde{V}$ and $\omega(\alpha^{\vee}, \nu) = \omega(\beta^{\vee}, \nu)=0$.
The piecewise linear transformations turning $\delta_{j}$ into $\delta_{j+1}$ all preserve the $\nu$ component and act solely on $\kappa$. 
Thus, it is enough to see that these piecewise linear transformations act on $\tilde{V}$ with period $h+2$. 

Divide $\tilde{V}$ into $h+2$ cones, one spanned by each pair of adjacent almost positive roots. A case-by-case verification shows that the action on each of these cones is linear, and repeats after $h+2$ steps.
\end{proof}

We pause to notice a geometric description of the rank-two cycles in a polyhedral framework.
The following proposition is immediate.

\begin{prop}\label{rk2cyc poly}
Suppose $(G,C,C\ck)$ is a well-connected polyhedral framework with corresponding fan $\F$.
Then a cycle $\rho$ in $G$ is a rank-two cycle if and only if there exists a codimension-$2$ face $F$ of $\F$  such that $\rho$ is the set of vertices $v$ of $G$ with $F\subset\Cone(v)$.  
\end{prop}

We are now prepared to define a simply connected framework.
As before, let $(G, C, C^{\vee})$ be a framework.
Motivated by Lemma~\ref{lem:rank2}, define a $2$-dimensional $CW$-complex $\Sigma$ whose $1$-skeleton is $G$ and whose $2$-faces have boundaries the rank-two cycles.
We define $(G, C, C^{\vee})$ to be \newword{simply connected} \label{simply connected} if $\Sigma$ is simply connected. 
Equivalently, $(G, C, C^{\vee})$ is simply connected if $\pi_1(G, v)$ is generated by paths of the form $\sigma \tau \sigma^{-1}$ where $\tau$ travels around a rank $2$ cycle and $\sigma$ is some path from the basepoint $v$ to that rank-two cycle. (This condition is plainly independent of the choice of basepoint $v$.)

\begin{prop}\label{simply ample}
If $(G,C,C\ck)$ is simply connected, then it is ample.
\end{prop}

\begin{proof}
Let $\tilde{v}_{\mathrm{start}}$ and $\tilde{v}_{\mathrm{end}}$ be two vertices of $\hG$, lying above the same vertex $v$ of $G$. Let $\tilde{\rho}$ be the path from $\tilde{v}_{\mathrm{start}}$ to $\tilde{v}_{\mathrm{end}}$ in $\hG$, so $\tilde{\rho}$ projects down to a cycle $\rho$ in $G$. So $\rho$ can be written as the concatenation of paths of the form $\sigma \tau \sigma^{-1}$ as above. Let $\sigma$ run from $v$ to $u$. The cycle $\tau$ lifts to a path $\tilde{\tau}$ from some $\tilde{u}_1$ to some $\tilde{u_2}$. Let $\sigma$ lift to the path in $\hG$ from $\tilde{v}_1$ to $\tilde{u}_1$, and let $\sigma^{-1}$ lift to the path from $\tilde{u}_2$ to $\tilde{v}_2$.

By Lemma~\ref{lem:rank2}, $\Seed$ takes the same value at $\tilde{u}_1$ and $\tilde{u}_2$. 
Since mutation is involutive, traveling from $\tilde{u}_2$ to $\tilde{v}_2$ precisely undoes the effect on $\Seed$ of traveling from  from $\tilde{v}_1$ to $\tilde{u}_1$. So $\Seed(\tilde{v}_1) = \Seed(\tilde{v}_2)$.   
Continuing in this manner, we deduce that $\Seed(\tilde{v}_{\mathrm{start}}) = \Seed(\tilde{v}_{\mathrm{end}})$, as desired.
\end{proof}

It seems natural that simple connectivity should have some characterization in terms of the topology of the fan arising from the framework.
We now explain what prevents us from giving such a characterization.
Let $(G, C, C\ck)$ be a well-connected polyhedral framework.  
Define $\Delta$ to be the simplicial complex whose vertices are the rays of the cones $\Cone(v)$, as $v$ ranges over the vertices of $G$, with a collection of rays forming a face if they are contained in $\Cone(v)$ for some vertex $v$ of~$G$. 
In light of Proposition~\ref{rk2cyc poly}, the framework $(G,C,C\ck)$ is simply connected if and only if the space obtained from $\Delta$ by deleting all codimension-$3$ faces is simply connected as a topological space.

However, we would really like a characterization of simple connectivity in terms of a concrete topological space in $V^*$. 
Let $\F$ be the fan associated to $(G, C, C\ck)$.
We write $|\F|$ for the union of the cones of $\F$. 
The space $|\F|$ is contractible to the origin; the important topological information is contained in $|\F|\cap S$, where $S$ be the unit sphere in $V^*$. 
While there is an obvious continuous surjection $\Delta \longrightarrow |\F| \cap S$, it isn't clear that $|\F| \cap S$ is homeomorphic to $\Delta$. 
One can imagine situations where $\Delta$ is simply connected while $|\F|\cap S$ is not, or vice versa.  
Thus, while we certainly encourage the reader to think about simple connectivity of a framework in terms of the topology of $|\F| \cap S$, we do not adopt this as a definition.

\subsection{Descending frameworks} \label{sec descend}

We now describe a condition that implies simple-connectivity, and many other good conditions, but requires no topological notions.  
We say a framework is \newword{descending} \label{descending} if it satisfies the following three conditions.\\

\noindent
\textbf{Unique minimum condition:}   \label{Unique minimum condition}
If a vertex $v$ of $G$ has $\set{\sgn(\beta):\beta\in C(v)}=\set{1}$, then~$v$ is the base vertex $v_b$.\\

\noindent
\textbf{Full edge condition:}  \label{Full edge condition}
If $(v,e)$ is an incident pair and $\sgn(C(v,e)) = -1$ then $e$ is a full edge.\\

The Sign and Transition conditions let us give an orientation to each edge of $G$.
If $e$ is an edge incident to a vertex $v$, then we direct $e$ towards $v$ if $\sgn(C(v,e)) = 1$ and away from $v$ if $\sgn(C(v,e))=-1$.\\

\noindent
\textbf{Descending chain condition:}  \label{Descending chain condition}
There exists no infinite sequence $v_0\to v_1\to\cdots$.\\

\begin{remark} \label{green framework}
This remark concerns the relation between frameworks and maximal green sequences in the sense of~\cite{MGS}. (See also~\cite{MGS2}.)
Let $v$ be a vertex of $G$ and let $e$ be an edge incident to $v$. Let $Q$ be the quiver at $v$ and let $[e]$ be the corresponding vertex of $Q$.
From Theorem~\ref{framework principal}(\ref{extended mat}), we see that $[e]$ is a ``green" vertex of $Q$ if $C(v,e)$ is a positive root and $[e]$ is a ``red" vertex if $C(v,e)$ is a negative root.
By our convention of directing $e$ towards $v$ if $C(v,e)$ is positive, a green sequence is a path $v_b=v_0 \leftarrow v_1 \leftarrow \cdots \leftarrow v_p$.
We read the green sequence in the direction opposed to the arrows.
In a descending framework, following red edges (i.e.\ following arrows)
from any vertex is guaranteed to eventually lead to the initial vertex.
\end{remark}

The Descending chain condition is unsatisfying, because it will be easy to see, by \cite[Theorem~1.8]{ca2} (and in particular the implication (iii)$\implies$(i) in that theorem), that a complete framework cannot satisfy the Descending chain condition unless $B$ is of finite type.
However, the notion of a descending framework will be critical to the construction of complete exact frameworks for exchange matrices $B$ whose associated Cartan matrix is of finite or affine type, and the construction of (non-complete) exact frameworks in general.
The key point is the following theorem.

\begin{theorem}\label{descending good}
A descending framework is exact, polyhedral, well-connected, and simply connected.
\end{theorem}
Recall that the polyhedral property implies injectivity, and that simple connectivity implies ampleness.
Thus we need only prove that a descending framework is polyhedral, well-connected, and simply connected.
We prove this as three separate propositions.

\begin{prop}\label{descending polyhedral}
A descending framework is polyhedral.
\end{prop}

We first need a lemma:

\begin{lemma}\label{finite}
Suppose $(G,C,C\ck)$ is a descending framework.
For each vertex $v$ of $G$, there are a finite, nonzero number of directed paths from $v$ to $v_b$.
\end{lemma}
\begin{proof}
The Full edge condition and the Unique minimum condition imply that every vertex $v\neq v_b$ has an edge directed from $v$ to another vertex $v'$.
This observation and the Descending chain condition imply that, for every vertex $v$, there is a finite directed path from $v$ to $v_b$.
On the other hand, suppose there are infinitely many directed paths from $v$ to $v_b$.
The vertex $v$ has out-degree at most $n$.
Thus some vertex $v'$ with $v\to v'$ has infinitely many directed paths from $v'$ to $v_b$.
For the same reason, there is some vertex $v''$ with $v'\to v''$ such that there are infinitely many directed paths from $v''$ to $v_b$.
Continuing in this way, we obtain a contradiction to the Descending chain condition.
\end{proof}

In light of Lemma~\ref{finite}, we can define the \newword{length} $\ell(v)$ of $v$ to be the length of a longest directed path from $v$ to $v_b$.

\begin{proof}[Proof of Proposition~\ref{descending polyhedral}]
We first check that, if $\Cone(u) = \Cone(v)$, then $u=v$. Our proof is by induction on $\ell(u) + \ell(v)$; the base case is trivial.
If $\ell(u) +\ell(v)>0$ then one of $u$ and $v$, say without loss of generality $u$, must not be the base vertex $v_b$. Take a path $u \to u_1 \to u_2 \to \cdots \to v_b$. 
Let $e$ be the edge $u\to u_1$. Since $\Cone(u) = \Cone(v)$, there is an edge $f$ incident to $v$ with $C(v,f) = C(u,e)$. By the Full edge condition, there is a vertex $v_1$ at the other end of $f$.
Then, $\Cone(u_1) = \Cone(v_1)$ as they are computed from $\Cone(u)$ and $\Cone(v)$ by the same recursion. 
By induction, $u_1=v_1$. Then $e$ and $f$ are two edges incident to $u_1$ with $C(u_1, e) = C(u_1, f)$, so $e=f$ and $u=v$, as desired.

So the vertices of $G$ are in bijection with the set of cones $\{ \Cone(v): v \in G \}$. We now must check that these cones are the maximal faces of a fan.
In light of Lemma~\ref{MaxCheckFan}, we need simply check that, for any two distinct vertices $u$ and $v$ of $G$, the cones $\Cone(u)$ and $\Cone(v)$ meet nicely.

We will say that two cones \newword{meet badly} if they do not meet nicely.
Suppose for the sake of contradiction that there exist (necessarily distinct) vertices $u$ and $v$ such that $\Cone(u)$ and $\Cone(v)$ meet badly and choose $u$ and $v$ so as to minimize $\ell(u)+\ell(v)$.

We consider two cases.
Throughout the argument, $x$ will be a point in the interior of $\Cone(v_b)$. 

\noindent
\textbf{Case 1:} 
The cones $\Cone(u)$ and $\Cone(v)$ intersect in dimension~$n$.
Since $u$ and $v$ are distinct, at least one of them has positive length, and therefore by the Unique minimum condition, there is at least one element of $C(u)\cup C(v)$ whose sign is $-1$.
Let $p$ be a point in the intersection of the interior of $\Cone(u)$ with the interior of $\Cone(v)$.
Consider points of the form $p+\ep x$ for $\ep>0$.
If $\beta\in C(u)\cup C(v)$, then the set $\set{p+\ep x:\ep>0}$ intersects the hyperplane $\beta^\perp\subset V^*$ if and only if $\sgn(\beta)=-1$.
By choosing $p$ generically, we can assume that $\set{p+\ep x:\ep>0}$ intersects each of the hyperplanes $\beta^\perp$ with $\beta\in C(u)\cup C(v)$ and $\sgn(\beta)=-1$ at a different point (except if two of the hyperplanes coincide).
Let $\ep_0$ be the smallest positive $\ep$ such that $p+\ep x$ is contained in a hyperplane $\beta^\perp$ with $\beta\in C(u)\cup C(v)$.
Then $p+\ep_0x$ is contained in the relative interior of a facet of $\Cone(u)$ or a facet of $\Cone(v)$ or both.

We first consider the case where $p+\ep_0x$ is contained both in the relative interior of a facet $F$ of $\Cone(u)$ and in the relative interior of a facet $G$ of $\Cone(v)$.
By the Full edge condition, the edge in $I(u)$ labeled $\beta$ is directed from $u$ to a vertex $u'$, and the edge in $I(v)$ labeled $\beta$ is directed from $v$ to a vertex $v'$.
By Corollary~\ref{cor:BdyFacet}, $\Cone(u)$ and $\Cone(u')$ share the facet $F$ and $\Cone(v)$ and $\Cone(v')$ share the facet~$G$.
The facets $F$ and $G$ are defined by the same hyperplane.
If $u'=v'$, then $u'$ is incident to two edges with the label $-\beta$, contradicting Proposition~\ref{basis}.
Otherwise, for small enough $\ep>\ep_0$, the point $p+\ep x$ is in the intersection of the interior of $\Cone(u')$ with the interior of $\Cone(v')$.
This contradicts our choice of $u$ and $v$ to minimize $\ell(u)+\ell(v)$.

Now we can assume, without loss of generality, that $p+\ep_0x$ is in the interior of $\Cone(v)$ and in the relative interior of a facet $F$ of $\Cone(u)$.
The Full edge condition says that the edge in $I(u)$ labeled $\beta$ is directed from $u$ to a vertex $u'$, and Corollary~\ref{cor:BdyFacet} says that $\Cone(u)$ and $\Cone(u')$ share the facet $F$.
Thus for small enough $\ep>\ep_0$, 
the point $p+\ep x$ is in the intersection of the interior of $\Cone(u')$ with the interior of $\Cone(v)$.
Again, this contradicts our choice of $u$ and $v$.

\noindent
\textbf{Case 2:} 
The cones $\Cone(u)$ and $\Cone(v)$ intersect in dimension less than~$n$.
Let $F_1,\ldots,F_k$ be the set of facets of $\Cone(u)$ containing $\Cone(u)\cap\Cone(v)$, and let $G_1,\ldots,G_l$ be the set of facets of $\Cone(v)$ containing $\Cone(u)\cap\Cone(v)$.
Each of the facets in $\set{F_1,\ldots,F_k,G_1,\ldots,G_l}$ is defined by a vector $\beta\in C(v)\cap C(v)$.
We claim that at least one of these vectors $\beta$ has $\sgn(\beta)=-1$.

Suppose for the sake of contradiction that the claim fails.
Let $p$ be a point in the relative interior of $\Cone(u)\cap\Cone(v)$.
Since each $F_i$ is defined by a vector $\beta$ with $\sgn(\beta)=1$, the vector $p+\ep x$ is in the interior of $\Cone(u)$ for small enough positive $\ep$.
For the same reason, $p+\ep x$ is in the interior of $\Cone(v)$ for small enough positive $\ep$.
This shows that the interior of $\Cone(u)$ intersects the interior of $\Cone(v)$, contradicting the hypothesis of Case 2, and thus proving the claim.

Without loss of generality, let $F$ be a facet of $\Cone(u)$ that contains $\Cone(u)\cap\Cone(v)$ and such that $F$ is defined by $\beta\in C(u)$ with $\sgn(\beta)=-1$.
Then $F\cap\Cone(v)=\Cone(u)\cap\Cone(v)$, and this intersection is either not a face of $F$ or not a face of $\Cone(v)$ or both.
In other words, $F$ and $\Cone(v)$ meet badly.
The Full edge condition says that the edge in $I(u)$ labeled $\beta$ is directed from $u$ to a vertex $u'$.
By Corollary~\ref{cor:BdyFacet}, $\Cone(u)$ and $\Cone(u')$ share the facet $F$.
Since $F$ and $\Cone(v)$ meet badly, $\Cone(u')$ and $\Cone(v)$ meet badly.
This contradicts our choice of $u$ and $v$ to minimize $\ell(u)+\ell(v)$, thus completing the proof.
\end{proof}

\begin{prop}\label{desc w-c}  
A descending framework is well-connected.
\end{prop}
\begin{proof}  
Suppose $u$ and $v$ are vertices of a descending framework and suppose that $F$ is a face of $\Cone(u)$ and of $\Cone(v)$.
We need to show that there exists a path $u=v_0$, $v_1$, \dots, $v_k=v$ in $G$ such that $F$ is a face of $\Cone(v_i)$  for all $i$ from $0$ to $k$.

We first construct a path $u=u_0$, $u_1$, \dots, $u_l$ with two useful properties:
First, $F$ is a face of $\Cone(u_i)$ for all $i$ from $0$ to $l$.
Second, if $\beta_1,\ldots,\beta_p$ are the labels in $C(u_l)$ such that $F=\Cone(u_l)\cap\beta_1^\perp\cap\cdots\cap\beta_p^\perp$, then each of the labels $\beta_1,\ldots,\beta_p$ has sign $+1$.
If the singleton path $u$ does not have the second property, then there is a label $\beta\in C(u)$ with $\sgn(\beta)=-1$ such that $F\subset\beta^\perp$.
The Full edge condition implies that $\beta$ is $C(u,e)$ where $e$ is an edge connecting $u$ to a vertex $u_1$.
By Corollary~\ref{cor:BdyFacet}, $F$ is a face of $\Cone(u_1)$.
The path $u=u_0,u_1$ has the first property. 
If it does not have the second property, then we construct $u_2$, etc.
The Descending Chain condition implies that eventually we will construct a path with the desired properties.
In particular, if $y$ is any point in the relative interior of $F$ and $x$ is a point in the interior of $\Cone(v_b)$, then for small enough $\ep$ the point $y+\ep x$ is in the interior of $\Cone(u_l)$.

We can now perform the same construction to obtain a path $v=v_0,v_1,\ldots,v_m$ such that $F$ is a face of $\Cone(v_i)$ for all $i$ from $0$ to $m$, and $F=\Cone(v_m)\cap\gamma_1^\perp\cap\cdots\cap\gamma_q^\perp$ such that $\gamma_1,\ldots,\gamma_q$ are labels in $C(v_m)$, each with sign $+1$.
But then, for small enough $\ep$ the point $y+\ep x$ is in the interior of $\Cone(v_m)$.
By Proposition~\ref{descending polyhedral}, we conclude that $u_l=v_m$.
Now, concatenating the two paths, we obtain the desired path between $u$ and $v$.
\end{proof}

The following proposition combines with Proposition~\ref{simply ample} to show that a descending framework is ample.
\begin{prop}\label{descending simply}
A descending framework is simply connected.
\end{prop}

The outline of the proof of Proposition~\ref{descending simply} is a simple inductive argument once we have the following lemma.

\begin{lemma}\label{descending down}
Suppose $(G,C,C\ck)$ is a descending framework and suppose $u$, $v$ and $w$ are vertices of $G$ with $v\to u$ and $v\to w$.
Then there exists a rank-two cycle in $G$ containing $u$, $v$ and $w$ that has $v$ as its unique source.
\end{lemma}

\begin{proof}
In the notation of Section~\ref{simply sec}, let $e$ and $f$ be the edges $v \to u$ and $v \to w$.
Form the sequences $v_k$, $e_k$ and $\gamma_k$ as in that section; although we don't know yet that those sequences are bi-infinite, because \textit{a priori} some $e_k$ might be a half-edge.
Let $b$ and $b'$ be as before.

First, suppose for the sake of contradiction that $\begin{psmallmatrix} 2 & -|b'| \\ -|b| & 2 \end{psmallmatrix}$ is of infinite type. 
Then all of the $\gamma_k$ for $k>0$ are negative roots and, by the Full edge condition, our sequence cannot terminate in the direction of positive $k$.
But then $v_0 \to v_1 \to v_2 \to \cdots$ is an infinite descending chain, a contradiction.

So $\begin{psmallmatrix} 2 & -|b'| \\ -|b| & 2 \end{psmallmatrix}$ is of finite type.
Let $h$ be its Coxeter number. Switching $u$ and $w$ if necessary, we may assume that $b \geq 0$ and $b' \leq 0$.
Then the computations of Section~\ref{simply sec} show that we have $v_{-2} \from v_{-1} \from v_0 \to v_1 \to \cdots \to v_h$, and all of these edges exist by the Full edge condition.
Moreover, from Lemma~\ref{lem:rank2variant}, we have $C(v_{-2}) = C(v_h)$. So $\Cone(v_{-2}) = \Cone(v_h)$ and, as $(G, C, C^{\vee})$ is polyhedral, we have $v_{-2} = v_h$.
So $v_{-2} \from v_{-1} \from v_0 \to v_1 \to \cdots \to v_h=v_2$ is the desired cycle.
\end{proof}

\begin{proof}[Proof of Proposition~\ref{descending simply}]
We will show that any path $v_0,\ldots,v_k$ from $v_b$ to itself can be contracted in $\Sigma$.
Let $m$ be the maximum length of a vertex on the path.
We argue by induction on $m$ and, for fixed $m$, by induction on the number $r$ of times the length $m$ is achieved on the path.
If $m=0$ then the path is the singleton at $v_b$.
Otherwise, let $v_i$ be a vertex in the path with $\ell(v)=m>0$.
Since $m$ is the maximum length on the path, the neighbors of $v_i$ in the path must have lower length, so the edges to these neighbors are directed $v_i\to v_{i-1}$ and $v_i\to v_{i+1}$.
If $v_{i-1}=v_{i+1}$ then the path backtracks at $v_i$ and we may delete this backtracking and proceed by induction.

If $v_{i-1} \neq v_{i+1}$, Lemma~\ref{descending down} says that there exists a rank-two cycle in $G$, having $v_i$ as its unique source.
In particular, every vertex on that cycle has length strictly less than $\ell(v)$.
Now replace the segment $(v_{i-1}, v_i, v_{i+1})$ by the other $h$ edges in that cycle.
By construction, the new path is related to the old path by moving across a face of $\Sigma$.

If $r>1$, then the new path has maximum length $m$, realized $r-1$ times.
If $r=1$, then the new path has maximum length $m-1$.
In any case, by induction, the new path can be contracted in $\Sigma$.
\end{proof}

We conclude our discussion of descending frameworks with one more property.
Recall from Remark~\ref{gvec weight or int vec} that the $\g$-vectors, as defined in this paper, are certain vectors in the weight lattice.

\begin{prop} \label{descending g sign-coherent}  
Let $(G,C,C\ck)$ be a descending framework.
For each vertex $v$ of $G$, the fundamental-weight coordinate vectors of the $\g$-vectors of the cluster variables in $\Seed(v)$ form a sign-coherent set.
\end{prop}
\begin{proof}
In light of Theorem~\ref{framework principal}(\ref{g vec}), the assertion of the proposition for each $v$ is the following:
For each $i\in I$, the interior of $\Cone(v)$ is disjoint from the hyperplane~$\alpha_i^\perp$.

Suppose the assertion fails for some $i$, and choose $v$ to minimize $\ell(v)$ among counterexamples.
We know that $v\neq v_b$ because $\Cone(v_b)=D$.
Let $y$ be a point in the intersection of $\alpha_i^\perp$ and the interior of $\Cone(v)$. 
Since $y\in\alpha_i^\perp$, expanding $y$ in the basis of fundamental weights, the coefficient of $\rho_i$ is $0$.
Let $x$ be a point of the form $\sum_{j\neq i} c_j \rho_j$, for some positive scalars $c_j$.
Then for a large enough scalar $a$, the vector $y+ax$ is in the boundary of $D$. 
Proposition~\ref{descending polyhedral} implies that the interior of $\Cone(v)$ is disjoint from $D=\Cone(v_b)$.
Thus the line segment connecting $y$ to $y+ax$ exits the interior of $\Cone(v)$ through the relative interior of some face $F$ of $\Cone(v)$. 

Let $\beta_1$, \dots, $\beta_k$ be the distinct vectors in $C(v)$ such that $F=\Cone(v) \cap {(\beta_1^{\perp} \cap \cdots \cap \beta_k^{\perp})}$.
Since $y$ is in the interior of $\Cone(v)$, we have $\br{y,\beta_j}>0$ and since ${y+ax}$ is on the other side of $\beta_i^\perp$, we have $\br{y+ax,\beta_j}<0$.
Thus $\br{ax,\beta_j}<0$ for each~$j$.
Since $ax$ is a positive combination of fundamental weights, we conclude that $\sgn(\beta_j)=-1$ for each~$j$.
If we choose the $c_j$ generically, then we can arrange that $\dim F \geq n-2$, so $k=1$ or $2$.

\textbf{Case 1:} $\dim F = n-1$, so $k=1$. Let $e_1$ be the edge incident on $v$ with label $\beta_1$. Since $\sgn(\beta_1) = -1$, by the Full edge condition there is a vertex $v'$ at the other end of $e_1$. The line segment from $y$ to $y+ax$ enters the interior of $\Cone(v')$, so $\alpha_i^{\perp}$ meets the interior of $\Cone(v')$. But $\ell(v') \leq \ell(v)-1$,  contradicting our choice of $v$ to minimize $\ell(v)$.

\textbf{Case 2:}  
$\dim F = n-2$, so $k=2$. 
Let $e$ be the edge with $C(v,e)=\beta_1$ and let $f$ be the edge with $C(v,f)=\beta_2$.
Without loss of generality, assume that $\omega(C\ck(v,e),C(v,f)) \geq 0$.  By the Full edge condition, there is a vertex $v'$ adjacent to $v$ through $e$, and the Transition condition says that $C(v',e)=-\beta_1$ and $C(v',\mu_e(f))=C(v,f)=\beta_2$.
Again, by the Full edge condition, there is a vertex $v''$ adjacent to $v'$ through $\mu_e(f)$, and the Transition condition says that $C(v'')$ contains the roots $-\beta_2$ and $\gamma=(-\beta_1)+[\sgn(\beta_2)\omega(C\ck(v',\mu_e(f)),-\beta_1)]_+$.
By the Co-label condition, $C\ck(v',\mu_e(f))$ is a positive scaling of $C(v',\mu_e(f))=\beta_2$, and we see that $\gamma=-\beta_1$.
Since $-\beta_1$ and $-\beta_2$ are in $C(v'')$ and since $\beta_1$ and $\beta_2$ define the face $F$ of $\Cone(v)$, we see that the line segment from $y$ to $y+ax$ passes from the interior of $\Cone(v)$, through $F$, and into the interior of $\Cone(v')$. 
Thus $\alpha_i^{\perp}$ meets the interior of $\Cone(v')$, but $\ell(v') \leq \ell(v)-2$, so we again have a contradiction.
\end{proof}

\section{Cambrian frameworks}\label{camb sec}
In this section, we apply Cambrian combinatorics to construct a descending reflection framework for a given acyclic exchange matrix $B$.
We have seen that $B$ determines a Cartan matrix $\Cart(B)$ and thus a Coxeter group $W$.
The matrix $B$ also determines a Coxeter element of $W$, in a way that we review below.
The graph $G$ in the framework is the Hasse diagram of the \newword{Cambrian semilattice} associated to the orientation.
Thus the vertices of $G$ are the \newword{sortable elements}.
The labels are certain roots that can be read off combinatorially from sorting words for the sortable elements.
The framework is not complete unless $\Cart(B)$ is of finite type, but in a future paper we will extend the construction to produce complete frameworks in the case where $\Cart(B)$ is of affine type.

More details on the combinatorial and polyhedral constructions in this section can be found in \cite{typefree}.

\subsection{Sortable elements and Cambrian lattices}\label{sort camb sec}
Let $W$ be the Coxeter group determined by $B$ as explained in Section~\ref{ref frame subsec}.
Recall that $S=\set{s_i:i\in I}$ is the set of simple reflections in $W$.
It will be convenient, in what follows, to sometimes suppress the indexing set $I$ and let $S$ serve as an indexing set.
Thus, for example, we may write $\alpha_s$ for $\alpha_i$ when $s=s_i$, and so forth.

When $B$ is acyclic, it encodes a \newword{Coxeter element} $c$ of $W$.
A Coxeter element is an element that arises by multiplying the generators $S$, in any order, with each generator appearing exactly once.
Since $B$ is acyclic, we can take $I$ to be $\set{1,\ldots,n}$ such that $b_{ij}>0$ implies $i<j$ and let $c$ be the Coxeter element $s_1\cdots s_n$.

\begin{example}\label{02-10 camb ex 1} 
To illustrate the constructions in this section, we consider the case where $B=\begin{psmallmatrix*}[r]0&2\\-1&0\end{psmallmatrix*}$.
The Cartan companion of $B$ is $\Cart(B)=\begin{psmallmatrix*}[r]2&-2\\-1&2\end{psmallmatrix*}$, and the associated Coxeter group has generators $s_1$ and $s_2$ with $m(1,2)=4$.  
The reflections in $W$ are $s_1$, $s_2$, $s_1s_2s_1$, and $s_2s_1s_2$.
The positive roots are $\alpha_1$, $\alpha_2$, ${\alpha_1+\alpha_2}$, and ${2\alpha_1+\alpha_2}$ and the positive co-roots are $\alpha\ck_1$, $\alpha\ck_2$, ${\alpha\ck_1+\alpha\ck_2}$, and ${\alpha\ck_1+2\alpha\ck_2}$.
The exchange matrix $B$ encodes the Coxeter element $c=s_1s_2$.
\end{example}

An expression for $w\in W$ as a word in the generators $S$ is called \newword{reduced} if it is of minimal length among words for $w$.
This minimal length is called the \newword{length} of $w$ and written $\ell(w)$. 
The \newword{(right) weak order} on $W$ is the transitive closure of the relations $w<ws$ for all $w\in W$ and $s\in S$ such that $\ell(w)<\ell(ws)$.  
In this paper, inequalities between elements of $W$ always mean this relation.

The weak order is a meet-semilattice, and furthermore, given any subset $U\subseteq W$, if $U$ has an upper bound in $W$, then it has a join $\Join U$.
The weak order is also characterized in terms of inversion sets.
An \newword{inversion} of $w\in W$ is a reflection $t$ such that $\ell(tw)<\ell(w)$.
Write $\inv(w)$ for the set of inversions of $w$.
We have $u\le w$ if and only if $\inv(u)\subseteq\inv(w)$.

Let $c^\infty$ be the half-infinite word $s_1\cdots s_ns_1\cdots s_ns_1\cdots s_n\ldots$ formed by infinitely repeating the word $s_1\cdots s_n$.
Given $w\in W$, the \newword{$c$-sorting word} for $w$ is the word obtained by choosing, among all subsequences of $c^\infty$ that form reduced words for $w$, the subsequence that is lexicographically leftmost in $c^\infty$.
Every element of $w$ has a unique $c$-sorting word.
This word is equivalent to a sequence of subsets of~$S$:  
Reading $c^\infty$ from left to right, in each repetition of $s_1\cdots s_n$, we take the set of letters that appear in the $c$-sorting word for $w$.
If this sequence of subsets is weakly decreasing in containment order, then we call $w$ a \newword{$c$-sortable element}.

A generator $s\in S$ is \newword{initial} in $c$ if there is a reduced word for $c$ having $s$ as its first letter.
Similarly, $s$ is \newword{final} in $c$ if it is the last letter of some reduced word for $c$.
In either case, the element $scs$ is another Coxeter element.

Given a subset $J\subseteq S$, the \newword{standard parabolic subgroup} $W_J$ is the subgroup of $W$ generated by $J$.
This subgroup forms an order ideal in the weak order on $W$.
The \newword{restriction} of $c$ to $W_J$ is the Coxeter element of $W_J$ obtained by deleting the letters in $S\setminus J$ from any reduced word for $c$.
If $w\in W$ then there exists a unique element, denoted $w_J$, such that $\inv(w_J)=\inv(w)\cap W_J$.
We will be most interested in the case where $J=S\setminus\set{s}$, and we define the special notation $\br{s}$ to stand for $S\setminus\set{s}$.

The next two lemmas are \cite[Lemmas~2.4, 2.5]{sortable}.
Since the identity element is $c$-sortable for any $c$, the lemmas are a recursive characterization of $c$-sortability, by induction on the length $\ell(w)$ and on the rank of $W$ (the cardinality of $S$).
\begin{lemma}
\label{sc}
Let~$s$ be initial in~$c$ and suppose $w\not\ge s$.
Then~$w$ is $c$-sortable if and only if it is an $sc$-sortable element of~$W_{\br{s}}$.
\end{lemma}

\begin{lemma}
\label{scs}
Let~$s$ be initial in~$c$ and suppose $w\ge s$.
Then~$w$ is $c$-sortable if and only if $sw$ is $scs$-sortable.
\end{lemma}

The following is \cite[Proposition~3.13]{typefree}.
\begin{prop}\label{sort para}
If $v$ is $c$-sortable and $J\subseteq S$, then $v_J$ is $c'$-sortable, where $c'$ is the restriction of $c$ to $W_J$.
\end{prop}

Let $v$ be a $c$-sortable element of $W$ and let $a_1\cdots a_k$ be its $c$-sorting word.
For $r\in S$, there is a leftmost instance of $r$ in $c^\infty$ that is not in the subword of $c^\infty$ corresponding to $a_1\cdots a_k$.
Let $a_1\cdots a_j$ be the initial segment of $a_1\cdots a_k$ consisting of those letters that occur in $c^\infty$ before the omission of $r$.
Say~$a_1\cdots a_k$ \newword{skips} $r$ after $a_1\cdots a_j$.
If $a_1\cdots a_jr$ is a reduced word, then this is an \newword{unforced skip}.
Otherwise it is a \newword{forced skip}.
Define $C_c^r(v)$ to be the root $a_1\cdots a_j \cdot \alpha_r$. 
This is a positive root if and only if the skip is unforced.
Although this definition refers to the position of $a_1\cdots a_k$ in $c^\infty$, the root $C_c^r(v)$ is read off easily from $a_1\cdots a_k$ itself.
Write $C_c(v)$ for $\set{C_c^r(v):r\in S}$.

This definition of $C_c(v)$ is shown in \cite[Proposition~5.1]{typefree} to be equivalent to the following recursive definition:  
For~$s$ initial in $c$,
\[C_c(v)=\left\lbrace\begin{array}{ll}
C_{sc}(v)\cup\set{\alpha_s}&\mbox{if } v\not\ge s\\
s C_{scs}(sv)&\mbox{if } v\ge s
\end{array}\right.\]
The sets $C_{sc}(v)$ and $C_{scs}(sv)$ are defined by induction on the rank of $W$ or on the length of $v$.

A \newword{cover reflection} of $w\in W$ is an inversion $t$ of~$w$ such that $tw=ws$ for some $s\in S$.
The name ``cover reflection'' refers to the fact that~$w$ covers $tw$ in the weak order.
Indeed, the cover reflections of $w$ are the elements $wsw^{-1}$ such that $s\in S$ and $ws\covered w$.
(Here and for the remainder of the paper, we use the symbol $\covered$ to denote cover relations.)  
The set of cover reflections of~$w$ is written $\cov(w)$.
If $t$ is a cover reflection of~$w$ then $\inv(tw)=\inv(w)\setminus\set{t}$.

The following is \cite[Proposition~5.2]{typefree}.
\begin{prop}\label{lower walls}
Let $v$ be a $c$-sortable element.
The set of negative roots in $C_c(v)$ is $\set{-\beta_t:t\in\cov(v)}$.
\end{prop}

The \newword{$c$-Cambrian semilattice} $\Camb_c$ is the subposet of the weak order on $W$ induced by the $c$-sortable elements.
It is a sub-meet-semilattice of the weak order on $W$ by \cite[Theorem~7.1]{typefree}.
We will also use the symbol $\Camb_c$ to denote the undirected Hasse diagram of $\Camb_c$.

\begin{example}\label{02-10 camb ex 2}
This example continues Example~\ref{02-10 camb ex 1}.
The following table lists each $c$-sortable elements $v$, represented by its $c$-sorting word (except that the identity is given by the symbol $e$).
For each $v$, we also give $\cov(v)$, $C_c^1(v)$, and $C_c^2(v)$.
\[\begin{array}{|l|c|c|c|}
\hline
v&\cov(v)&C_c^1(v)&C_c^2(v)\\\hline&&&\\[-11pt]\hline
e&\emptyset&\alpha_1&\alpha_2\\\hline
s_1&\set{s_1}&-\alpha_1&2\alpha_1+\alpha_2\\\hline
s_1s_2&\set{s_1s_2s_1}&\alpha_1+\alpha_2&-2\alpha_1-\alpha_2\\\hline
s_1s_2s_1&\set{s_2s_1s_2}&-\alpha_1-\alpha_2&\alpha_2\\\hline
s_1s_2s_1s_2&\set{s_1,s_2}&-\alpha_1&-\alpha_2\\\hline
s_2&\set{s_2}&\alpha_1&-\alpha_2\\\hline
\end{array}\]
The Cambrian lattice, in this case, has cover relations 
\[e\covered s_1\covered s_1s_2\covered s_1s_2s_1\covered s_1s_2s_1s_2\quad\text{and}\quad e\covered s_2\covered s_1s_2s_1s_2.\] 
\end{example}

\subsection{The Cambrian framework}\label{camb frame sec}
In this section, we show that $(\Camb_c,C_c)$ is, in essence, a descending reflection framework.
But there is a little more work to do before we can make a precise statement.
Specifically, we need to add some half-edges to the Hasse diagram of $\Camb_c$ to get an $n$-regular quasi-graph.
Also, as it stands, the labels $C_c$ are not assigned to edges incident to a vertex $v$, but rather are indexed by $S$.
To fill in these pieces of the Cambrian framework, we will need more background on sortable elements and Cambrian lattices.

Suppose $w$ is any element of $W$.
Among the $c$-sortable elements that are $\leq w$, there is a maximal one (\cite[Corollary~6.2]{typefree}); we denote this maximal element by $\pidown^c(w)$.
One can also define $\pidown^c$ recursively; see \cite[Section~6]{typefree}.
The following is \cite[Theorem~7.3]{typefree}.
\begin{theorem} \label{semilattice hom}
For~$U$ any subset of~$W$, 
if~$U$ is nonempty then $\Meet \pidown^c(U)=\pidown^c(\Meet U)$ and if~$U$ has an upper bound then $\Join \pidown^c(U)=\pidown^c( \Join U)$.
\end{theorem}

The following is \cite[Theorem~6.1]{typefree}, or it can be obtained as an immediate corollary of Theorem~\ref{semilattice hom}. 
\begin{theorem} \label{order preserving}
$\pidown^c$ is order preserving.
\end{theorem}

We now describe the cover relations in $\Camb_c$.
\begin{lemma}\label{Cambrian covers}
Let $v$ be a $c$-sortable element.
Then $v'\covered v$ in the $c$-Cambrian semilattice if and only if $v'=\pidown^c(tv)$ for some $t\in\cov(v)$.
Furthermore if $t_1$ and $t_2$ are distinct cover reflections of $v$, then $\pidown^c(t_1v)\neq\pidown^c(t_2v)$.
\end{lemma}

Lemma~\ref{Cambrian covers} follows from a general result (see e.g.\ \cite[Proposition~2.2]{con_app}) on lattice congruences on a finite lattice, because Theorem~\ref{semilattice hom} implies that the fibers of $\pidown^c$ are a lattice congruence on the interval $[1,v]$. 
We provide a direct proof as well.

\begin{proof}
We first verify the second assertion by proving the stronger statement that $\pidown^c(t_1v)$ and $\pidown^c(t_2v)$ are incomparable.
Let $t_1$ and $t_2$ be distinct cover reflections of $v$.
Then the join of $t_1v$ and $t_2v$ exists and equals $v$, so Theorem~\ref{semilattice hom} says that 
\[v=\pidown^c(v)=\pidown^c(t_1v\join t_2v)=\pidown^c(t_1v)\join\pidown^c(t_2v).\]
Since $t_1v<v$ and $t_2v<v$, Theorem~\ref{order preserving} says that $\pidown^c(t_1v)<v$ and $\pidown^c(t_2v)<v$.
The elements $\pidown^c(t_1v)$ and $\pidown^c(t_2v)$ must in particular be incomparable to join to $v$.

Suppose $v'\covered v$ in $\Camb_c$.
Then $v'<v$ in the weak order, so there exists an element $w\in W$ with $v'\le w\covered v$ in the weak order, and necessarily, $w=tv$ for some $t\in\cov(v)$.
The element $\pidown^c(tv)$ has $v'\le\pidown^c(tv)<v$ in $\Camb_c$.
Since $v'\covered v$ in $\Camb_c$, we conclude that $v'=\pidown^c(tv)$.

Finally, consider any $t \in \cov(v)$. By Theorem~\ref{order preserving}, $\pidown^c(tv)<v$.
Thus there exists a $c$-sortable element $v'$ with $\pidown^c(tv)\le v'\covered v$ in $\Camb_c$.
By the previous paragraph, $v'=\pidown^c(t'v)$ for some $t'\in\cov(v)$.
But we showed above that $\pidown^c(tv)$ and $\pidown^c(t'v)$ are incomparable if $t'\neq t$. 
We conclude that $t'=t$, so that $\pidown^c(tv)=v'$.
\end{proof}

For each $c$-sortable element $v$, define a cone  
\[\Cone_c(v)=\bigcap_{\beta \in C_c(v)}\set{x\in V^*:\br{x,\beta} \geq 0}.\]
Once we have completed the construction of a Cambrian framework, this $\Cone_c(v)$ will coincide with the cone $\Cone(v)$ defined in Section~\ref{frame sec}.
Recall the cone $D = \bigcap_{s \in S} \set{x\in V^*: \br{x,\alpha_s}\ge 0}$.
The cones $wD$, for $w\in W$, are distinct, and form a fan whose support is called the \newword{Tits cone}.
The following is \cite[Theorem~6.3]{typefree}.

\begin{theorem} \label{pidown fibers}
Let $v$ be $c$-sortable. Then $\pidown^c(w)=v$ if and only if $w D \subseteq \Cone_c(v)$.
\end{theorem}

The following lemma is the last step needed to let us make a precise statement about Cambrian frameworks.
\begin{lemma}\label{cov beta}
If $v'\covered v$ in the $c$-Cambrian semilattice, then there exists a unique root $\beta$ such that $\beta\in C_c(v')$ and $-\beta\in C_c(v)$.
The root $\beta$ is the positive root $\beta_t$ associated to the cover reflection $t$ of $v$ such that $v'=\pidown^c(tv)$.
\end{lemma}

\begin{proof}
By Lemma~\ref{Cambrian covers}, $v=\pidown^c(tv)$ for some $t\in\cov(v)$.
By Theorem~\ref{pidown fibers}, $vD\subseteq\Cone_c(v)$ and $tvD\subseteq\Cone_c(v')$, and furthermore, $vD\not\subseteq\Cone_c(v')$ and $tvD\not\subseteq\Cone_c(v)$.
We know that $vD$ and $tvD$ share a codimension-1 facet contained in $(\beta_t)^\perp$, so the cones $\Cone_c(v)$ and $\Cone_c(v')$ each have a facet contained in $(\beta_t)^\perp$. 
So one of $\{ \beta_t, - \beta_t \}$ is in $C_c(v')$ and the other is in $C_c(v)$.
Furthermore, $\Cone_c(v)$ is separated from $D$ by $(\beta_t)^\perp$,  so $\beta_t\in C_c(v')$ and $-\beta_t\in C_c(v)$.

Since $vD$ and $tvD$ are not separated by any hyperplane besides $(\beta_t)^\perp$, there is no other hyperplane separating $\Cone_c(v)$ from $\Cone_c(v')$ and thus $\beta_t$ is the unique root $\beta$, as desired.
\end{proof}

We describe how to use Lemma~\ref{cov beta} to associate a root to each incident pair in $\Camb_c$.
Suppose $v'\covered v$ in $\Camb_c$ with $v'=\pidown^c(tv)$ and write $e$ for the edge $(v, v')$.
We label the incident pair $(v',e)$ by the root $\beta_t$ and label $(v,e)$ by the root $-\beta_t$.
Since the root $\beta$ in Lemma~\ref{cov beta} is unique, no edge is assigned more than one root. We have a function from full edges incident on $v$ to $C_c(v)$.

We now argue that this function is injective.
First, Lemmas~\ref{Cambrian covers} and~\ref{cov beta} combine to imply that no two edges get the same negative root.
Next, suppose $v\covered v'$ and $v\covered v''$, so that each of these edges gets a positive root.
Suppose that both get the same positive root $\beta$.
As shown in the proof of Lemma~\ref{cov beta}, the cones $\Cone_c(v)$ and $\Cone_c(v')$ intersect in a common facet.
This facet is contained in the hyperplane $\beta^\perp$ and contains the intersection of $tv'D$ and $v'D$.
The same holds with $v'$ replaced by $v''$.
We conclude that $\Cone_c(v)\cap\Cone_c(v')$ equals $\Cone_c(v)\cap\Cone_c(v'')$ and thus that $v''D\subseteq \Cone_c(v')$.
Now Theorem~\ref{pidown fibers} says that $\pidown^c(v'')=v'$, but $v''$ is $c$-sortable, so $v''=v'$.

Since no two distinct edges get the same root, the degree of each vertex $v$ in $\Camb_c$ is at most $n$.
If the degree is less than $n$, then we affix half-edges to $v$ to adjust the degree of $v$ to be $n$.
These half-edges are labeled with the remaining roots from $C_c(v)$.
We again re-use the symbol $\Camb_c$ for the quasi-graph thus obtained.
We also re-use the symbol $C_c$ to denote the labeling of incident pairs of $\Camb_c$ by roots.
We can now make a precise statement about Cambrian frameworks.  \label{end Cambrian framework}

\begin{theorem}\label{camb frame}
The pair $(\Camb_c,C_c)$ is a descending reflection framework for the exchange matrix $B$.
\end{theorem}

\begin{example}\label{02-10 camb ex 3}
This example continues Examples~\ref{02-10 camb ex 1} and~\ref{02-10 camb ex 2}.
Lemma~\ref{cov beta} lets us associate a label from  $C_c(v)$ to each incident pair $(v_e)$.
The resulting reflection framework $(\Camb_c,C_c)$ is shown in the left picture of Figure~\ref{B2 frame fig}.  
(In light of Proposition~\ref{ref implies frame}, the triple $(\Camb_c,C_c,C\ck_c)$, shown in Figure~\ref{B2 frame fig}, is a framework.)
\end{example}

In light of Theorem~\ref{descending good}, we have the following corollary to Theorem~\ref{camb frame}.
\begin{cor}\label{camb good}
The pair $(\Camb_c,C_c)$ is an exact, polyhedral, well-connected, simply connected reflection framework for $B$.
\end{cor}

Combining Corollary~\ref{camb good} with Theorem~\ref{well-connected g}, we obtain the following result, which was conjectured in \cite[Section~9.1]{typefree}.
Let $\F_c$ be the collection  consisting of all of the cones $\Cone_c(v)$ together with their faces, where $v$ ranges over all $c$-sortable elements. 
We call $\F_c$ the \newword{$c$-Cambrian fan}.

\begin{cor}\label{camb fan}
The Cambrian fan $\F_c$ is a simplicial fan, and distinct vertices of $\Camb_c$ label distinct maximal cones of $\F_c$.
Each maximal cone in $\F_c$ is the $\g$-vector cone for a unique seed for the cluster algebra $\A_\bullet(B)$.
\end{cor}

When $\Cart(B)$ is of finite type, $W$ is finite.
In this case $\Camb_c$ is the \newword{Cambrian lattice}, rather than \emph{semilattice}, and $\F_c$ coincides, via Theorem~\ref{pidown fibers}, to the fan defined by the \newword{Cambrian congruence}.  See \cite{cambrian,sort_camb} for details on the Cambrian congruence, and see \cite{con_app} for details on the construction of a fan from a lattice congruence on the weak order.
Also, when $\Cart(B)$ is of finite type, the Hasse diagram of the Cambrian lattice is an $n$-regular graph \cite[Corollary~8.1]{camb_fan}, so the framework $(\Camb_c,C_c)$ is complete.
Thus Corollary~\ref{camb fan} becomes the following statement in finite type.

\begin{cor}\label{camb fan finite}
If $\Cart(B)$ is of finite type, then the Cambrian fan $\F_c$ is a complete fan and coincides with the fan of $\g$-vector cones for the cluster algebra $\A_\bullet(B)$.
\end{cor}

\begin{remark}\label{gvec caveat}
When interpreting Corollaries~\ref{camb fan} and~\ref{camb fan finite}, it is important to keep in mind the definition of $\g$-vectors as vectors in the weight lattice.
(Recall from Remark~\ref{gvec weight or int vec} that the $\g$-vectors are more typically defined as integer vectors.)
It is also essential to recall the difference between fundamental weights (dual to the simple \emph{co-root} basis) and fundamental co-weights (dual to the simple \emph{root} basis).
If we identify $V$ with $\reals^n$ by identifying each simple co-root $\alpha_i\ck$ with the standard unit basis vector $\e_i$ and identify $\reals^n$ with $(\reals^n)^*$ in the usual way, then Corollary~\ref{camb fan finite} says that the Cambrian fan $\F_c$ coincides with the $\g$-vector fan defined in terms of integer vectors.
On the other hand, if we identify each simple \emph{root} $\alpha_i$ with $\e_i$, then the Cambrian fan coincides with the $\g$-vector fan (in terms of integer vectors) for $\A_\bullet(-B^T)$.
(Compare Proposition~\ref{-B^T} and Remark~\ref{coincide caveat}.)
\end{remark}

Since the exchange graph is dual to the $\g$-vector fan, we also have the following corollary.
\begin{cor}\label{exch finite}
If $\Cart(B)$ is of finite type, then the exchange graph of the cluster algebra $\A_\bullet(B)$ is isomorphic, as a graph, to $\Camb_c$ (the Hasse diagram of the $c$-Cambrian lattice).
\end{cor}

Since the Cambrian framework is complete when $B$ is of finite type, we can also combine Theorems~\ref{camb frame} with Theorems~\ref{complete conj} and~\ref{descending good}, Corollaries~\ref{exact conj} and~\ref{polyhedral conj} and Proposition~\ref{descending g sign-coherent} to give the following result.

\begin{cor}\label{finite type conj}
If $\Cart(B)$ is of finite type, then Conjectures~\ref{vertex conj}, \ref{face conj}, \ref{tildeB equiv}, \ref{H equiv}, \ref{g lattice conj}, \ref{g sign-coherent}, \ref{mon indep}, \ref{g fan conj}, \ref{strong g fan conj}, \ref{F 1 conj}, \ref{F max conj}, and~\ref{H coherent} all hold for $B$.
(Those conjectures that admit a choice of coefficients hold in the case of principal coefficients.)
\end{cor}
In Section~\ref{denom sec}, we prove Conjectures~\ref{nu conj}, \ref{denom lattice conj}, \ref{denom fan conj} and~\ref{strong denom fan conj} as well for $\Cart(B)$ of finite type.
For details on the status of these conjectures before this paper, see the end of Section~\ref{conj sec}.

\begin{remark}\label{somewhat similar manner}  
This proof of Conjecture~\ref{strong g fan conj} for $\Cart(B)$ of finite type is not entirely new.
The combination of \cite[Theorem~1.1]{con_app} with \cite[Theorem~1.1]{sort_camb} amounts to a proof that the $c$-Cambrian fan is indeed a fan for $\Cart(B)$ of finite type.
It was conjectured in \cite[Section~10]{camb_fan} that the $c$-Cambrian fan coincides with the collection of $\g$-vector cones, and proven for a special choice of $c$.
The conjecture was proved for all $c$ by Yang and Zelevinsky in \cite{YZ}.
\end{remark}

When $\Cart(B)$ is of infinite type, the framework $(\Camb_c,C_c)$ is not complete.
Indeed, as mentioned in Section~\ref{sec descend}, a descending framework cannot be complete unless $B$ is of finite type.
But there is a deeper reason for the incompleteness.
Theorem~\ref{pidown fibers} implies in particular that each cone $\Cone_c(v)$ intersects the Tits cone.
Typically, there are $\g$-vector cones that don't intersect the Tits cone.

Since the Cambrian framework is not always complete, we can in general only use the Cambrian fan to make partial statements about the cluster algebra.
For example, combining Theorem~\ref{camb frame} with Proposition~\ref{descending g sign-coherent}, we obtain the following result.
\begin{cor}\label{camb g sign-coherent}
Given any vertex $v$ in the Cambrian framework $(\Camb_c,C_c)$, the fundamental-weight coordinate vectors of the $\g$-vectors of the cluster variables in $\Seed(v)$ form a sign-coherent set.
\end{cor}

We now proceed to prove Theorem~\ref{camb frame}.
To begin with, the pair $(\Camb_c,C_c)$ satisfies the Base condition because, when $v$ is the identity element, $C_c(v)=\set{\alpha_i:i\in I}$.
The pair satisfies the Root condition by construction.

The form $\omega$ agrees with the form $\omega_c$ defined in \cite[Section~3]{typefree}.
Thus, the following proposition establishes the Reflection condition for $(\Camb_c,C_c)$.

\begin{prop}\label{Cc adjacent}
Suppose $v'\covered v$ in the $c$-Cambrian semilattice and let $e$ be the edge connecting them.
Let $t$ be the reflection in the statement of Lemma~\ref{cov beta}, so that $C_c(v,e)=-\beta_t$ and $C_c(v',e)=\beta_t$.
Let $\gamma$ be any other root in $C_c(v)$.
Then 
\begin{enumerate}
\item If $\omega_c(\beta_t\ck,\gamma)\ge 0$ then $t\gamma\in C_c(v')$. 
\item If $\omega_c(\beta_t\ck,\gamma)< 0$ then $\gamma\in C_c(v')$. 
\end{enumerate}
\end{prop}

The three lemmas below are \cite[Lemmas~3.7--3.9]{typefree}.
Proposition~\ref{Cc initial} is a combination of parts of \cite[Proposition~5.1]{typefree} and \cite[Proposition~5.4]{typefree}.
The subspace $V_J$ of $V$ is the real linear span of the simple roots $\set{\alpha_s:s\in J}$.

\begin{lemma} \label{OmegaRestriction}
Let $J\subseteq S$ and let $c'$ be the restriction of $c$ to $W_J$.
Then $\omega_c$ restricted to $V_J$ is $\omega_{c'}$. 
\end{lemma}

\begin{lemma} \label{OmegaInvariance}
If~$s$ is initial or final in $c$, then $\omega_c(\beta,\gamma)=\omega_{scs}(s\beta,s\gamma)$ for all roots $\beta$ and $\gamma$.
\end{lemma}

\begin{lemma} \label{OmegaNegativity}
If~$s$ is initial in~$c$ and~$t$ is a reflection in~$W$, then \mbox{$\omega_c(\alpha_s, \beta_t) \geq 0$}, with equality only if~$s$ and~$t$ commute.
\end{lemma}

\begin{prop}\label{Cc initial}
Let~$s$ be initial in~$c$ and let $v$ be a $c$-sortable element of~$W$ such that~$s$ is a cover reflection of $v$.
Then $\cov(v)=\set{s}\cup\cov(v_{\br{s}})$, and the set of positive roots in $C_c(v)$ is obtained by applying the reflection $s$ to each positive root in $C_{sc}(v_{\br{s}})$.
\end{prop}

Using the above lemmas and proposition, we now prove Proposition~\ref{Cc adjacent}.

\begin{proof}[Proof of Proposition~\ref{Cc adjacent}]
Our proof is by induction on the rank of $W$ and the length of $v$. Let $s$ be initial in $c$.  

\noindent
\textbf{Case 1:} $v' \geq s$. Then $v \geq s$. 
Since $s \in \inv(v')$, and $t \not \in \inv(v')$, we have $s \neq t$.
The recursive definition of $C_c$ says that $s(-\beta_t)$ and $s\gamma$ are in $C_{scs}(sv)$ and that $s\beta_t\in C_{scs}(sv')$.
As $s \neq t$, we have $s(-\beta_t) = -\beta_{sts}$ and $s\beta_t = \beta_{sts}$.
By induction on $\ell(v)$, we conclude that 
\begin{enumerate}
\item If $\omega_{scs}(s\beta\ck,s\gamma)\ge 0$ then $(sts)s\gamma\in C_{scs}(sv')$. 
\item If $\omega_{scs}(s\beta\ck,s\gamma)< 0$ then $s\gamma\in C_{scs}(sv')$. 
\end{enumerate}
Lemma~\ref{OmegaInvariance} says that $\omega_{scs}(s\beta\ck,s\gamma)=\omega_c(\beta_t\ck,\gamma)$, and we apply the recursive definition of $C_c$ to obtain the desired conclusion in this case.

\noindent
\textbf{Case 2:} $v \not \geq s$. Then $v' \not \geq s$.
Since $s \not \in \inv(v)$, and $t \in \inv(v)$, we have $s\neq t$.  
Both $v$ and $v'$ are $sc$-sortable elements of $W_\br{s}$ by Lemma~\ref{sc}.
Also, $C_c(v)=C_{sc}(v)\cup\set{\alpha_s}$ and $C_c(v')=C_{sc}(v')\cup\set{\alpha_s}$.
Since $t \neq s$, we have $t \in W_J$.

If $\gamma=\alpha_s$, then Lemma~\ref{OmegaNegativity} and the antisymmetry of $\omega_c$ say that $\omega_c(\beta_t\ck,\gamma) \le 0$, with equality if and only if $s$ and $t$ commute.
If the inequality is strict, then the conclusion of the proposition holds because $\gamma=\alpha_s\in C_c(v')$.
If equality holds, then $s$ and $t$ commute, so $\gamma=t\gamma=t\alpha_s=\alpha_s\in C_c(v')$.

If $\gamma\neq\alpha_s$, then $\gamma$ is associated to a reflection in $W_{\br{s}}$, so by induction on the rank of $W$, we see that 
\begin{enumerate}
\item If $\omega_{sc}(\beta_t\ck,\gamma)\ge 0$ then $t\gamma\in C_{sc}(v')$.
\item If $\omega_{sc}(\beta_t\ck,\gamma)< 0$ then $\gamma\in C_{sc}(v')$. 
\end{enumerate}
By Lemma~\ref{OmegaRestriction}, we obtain the conclusion of the proposition.

\noindent
\textbf{Case 3:} $v\ge s$ and $v'\not\ge s$.
Then the hyperplane separating $vD$ from $tvD$ is $(\alpha_s)^\perp$, so $t=s$.
We claim that $v'=v_{\br{s}}$.
Indeed, $\inv(sv)=\inv(v)\setminus\set{s}$, so $\inv((sv)_{\br{s}})=\inv(sv)\cap W_{\br{s}}=\inv(v)\cap W_{\br{s}}=\inv(v_{\br{s}})$.
Thus $(sv)_{\br{s}}=v_{\br{s}}$.
This is the maximal element of $W_{\br{s}}$ below $sv$, and it is also $sc$-sortable by Proposition~\ref{sort para}.
By Lemma~\ref{sc}, every $c$-sortable element below $sv$ is in $W_{\br{s}}$, so $v_{\br{s}}$ is the maximal such, i.e.\ $v'=\pidown^c(sv)=v_{\br{s}}$.

Now $\beta_t=\alpha_s$, so Lemma~\ref{OmegaNegativity} implies that $\omega_c(\beta_t\ck,\gamma)$ agrees weakly in sign with $\sgn(\gamma)$.
In this case, we obtain the desired conclusion by combining Proposition~\ref{lower walls} with Proposition~\ref{Cc initial}.
\end{proof}

We now complete the proof that $(\Camb_c,C_c)$ is a reflection framework by verifying the Euler conditions (E1), (E2), and (E3).
The form $E$ agrees with the form $E_c$ defined in \cite[Section~3]{typefree}.
Condition (E3) is immediate from \cite[Lemma 5.9]{typefree}, which defines a total order $\beta_1,\ldots,\beta_n$ on the roots in $C_c(v)$ such that $E_c(\beta_i,\beta_j)=0$ for $i<j$.
We will quote three lemmas from~\cite{typefree} to establish conditions (E1) and (E2).
They are parts of \cite[Lemmas~3.1--3.3]{typefree}.

\begin{lemma}\label{Ec initial}
If $s$ is initial in $c$ and $\beta_t$ is any positive root, then $E_c(\alpha_s,\beta_t)\ge0$, with equality if and only if $t\in W_{\br{s}}$.  
\end{lemma}

\begin{lemma}\label{Ec final}
If~$s$ is final in~$c$ and $\beta_t$ is any positive root, then $E_c(\beta_t, \alpha_s) \ge 0$, with equality if and only if $t \in W_{\br{s}}$.
\end{lemma}

\begin{lemma}\label{Ec invariant}
If~$s$ is initial or final in $c$, then $E_c(\beta,\beta')=E_{scs}(s \beta, s \beta')$ for all $\beta$ and~$\beta'$ in $V$.
\end{lemma}

\begin{prop}\label{Cc E1}
Condition (E1) holds for $(\Camb_c,C_c)$.
\end{prop}

\begin{proof}
Our proof is by induction on the rank of $W$ and the length of $v$. 
The statement is vacuously true when $\mathrm{Rank}(W) = \ell(v)=0$, as $C_c(v)$ is empty.

Let $\beta$ and $\gamma$ be in $C_c(v)$ with $\sgn(\beta)=1$ and $\sgn(\gamma)=-1$.
Let $s$ be initial in $c$.

\noindent
\textbf{Case 1:} $v\not\ge s$. 
In this case $v\in W_{\br{s}}$ and $C_c(v)=C_{sc}(v)\cup\set{\alpha_s}$,
so $\gamma$ is in $C_{sc}(v)$ and thus equals $-\beta_t$ for some $t\in W_{\br{s}}$.
If $\beta$ is $\alpha_s$, then Lemma~\ref{Ec initial} says that $E_c(\beta,\gamma)=0$.
Otherwise, $\beta$ is in $C_{sc}(v)$, and $E_{sc}(\beta,\gamma)=0$ by induction on rank.
It is immediate from the definition that $E_{sc}$ is the restriction of $E_c$, so $E_c(\beta,\gamma)=0$.

\noindent
\textbf{Case 2:}  $v\ge s$.
The roots $\pm\alpha_s$ switch signs when acted on by $s$, but no other roots change sign when acted on by $s$.
It is impossible to have $\beta=\alpha_s$, because if so, $\Cone_c(v)$ and $vD$ are on opposite sides of the hyperplane $\alpha_s^\perp$ in $V^*$.
This contradicts Theorem~\ref{pidown fibers}.
Thus $s\beta$ is a positive root.

If $\gamma\neq-\alpha_s$, then $s\gamma$ is a negative root.
The roots $s\beta$ and $s\gamma$ are in $C_{scs}(sv)$, so by induction on $\ell(v)$, we see that $E_{scs}(s\beta,s\gamma)=0$.
Lemma~\ref{Ec invariant} now says that $E_c(\beta,\gamma)=0$.
If $\gamma=-\alpha_s$ then $s$ is a cover reflection of $v$ by Proposition~\ref{lower walls}.
Proposition~\ref{Cc initial} says that $\beta=s\beta'$ for some positive root $\beta'\in C_{sc}(v_{\br{s}})$.
Lemma~\ref{Ec invariant} says that $E_c(\beta,\gamma)=E_{scs}(s\beta,s\gamma)$, which can be rewritten as $E_{scs}(\beta',\alpha_s)$, which is zero by Lemma~\ref{Ec final}.
\end{proof}

\begin{prop}\label{Cc E2}
Condition (E2) holds for $(\Camb_c,C_c)$.
\end{prop}

\begin{proof}
Our statement is by the usual induction on length and rank; it is vacuously true for rank $<2$.

Let $\beta$ and $\gamma$ be distinct roots in $C_c(v)$ with $\sgn(\beta)=\sgn(\gamma)$.
Since Condition~(E3) holds, we know that either $E_c(\beta,\gamma)$ or $E_c(\gamma,\beta)=0$.
Thus, in light of Proposition~\ref{sym antisym}, when it is convenient, we can verify that $K(\beta,\gamma)\le 0$ to show that $E_c(\beta,\gamma)\le0$.

\noindent
\textbf{Case 1:} $\beta$ and $\gamma$ are both negative roots.
This case can be handled without induction.
Proposition~\ref{lower walls} says that $\beta$ and $\gamma$ are both associated to cover reflections of $v$.
Thus $\beta=v\alpha_p$ and $\gamma=v\alpha_q$, for $p$ and $q \in S$.
But $K$ is invariant under the action of $W$, so $K(\beta,\gamma)=K(\alpha_p,\alpha_q)\le 0$.

\noindent
\textbf{Case 2:} $\beta$ and $\gamma$ are both positive roots. 

Let $s$ be initial in $c$.

\noindent
\textbf{Case 2a:} $v\not\ge s$. 
In this case $v\in W_{\br{s}}$ and $C_c(v)=C_{sc}(v)\cup\set{\alpha_s}$.
If neither $\beta$ nor $\gamma$ equals $\alpha_s$, then both are in $C_{sc}(v)$, and $E_c(\beta,\gamma)=E_{sc}(\beta,\gamma)$, which is nonpositive by induction on rank.
If $\beta=\alpha_s$ then $\gamma$ is in $C_{sc}(v)$ and so Lemma~\ref{Ec initial} says that $E_c(\beta,\gamma)=0$.
If $\gamma=\alpha_s$, then $\beta$ is a positive root in $C_{sc}(v)$.
Thus $\beta$ is a positive combination of simple roots $\alpha_r$ with $r\neq s$, and $K(\alpha_r,\alpha_s)\le 0$ for each such $r$, so $K(\beta,\alpha_s)\le 0$.

\noindent
\textbf{Case 2b:}  $v\ge s$.
It is impossible to have $\beta=\alpha_s$ or $\gamma=\alpha_s$, because if so, we reach a contradiction to Theorem~\ref{pidown fibers} as in Case 2 of the proof of Proposition~\ref{Cc E1}.
Thus $s\beta$ and $s\gamma$ are also positive.
By induction on $\ell(v)$, we see that $E_{scs}(sv)(s\beta,s\gamma)\le0$, so Lemma~\ref{Ec invariant} says that $E_c(\beta,\gamma)\le 0$.
\end{proof}

\begin{remark} 
If $W$ is finite, then we can approach Case~2 in a manner analogous to Case~1. 
Let $u = \piup_c(v)$, as defined in~\cite{sort_camb}. Then $\beta$ and $\gamma$ are of the form $u \alpha_p$ and $u \alpha_q$ for $p$, $q \in S$. But $\piup_c$ cannot be defined for infinite Coxeter groups. 
\end{remark}

We have shown that $(\Camb_c,C_c)$ is a reflection framework.
The fact that this framework is descending is a consequence of Proposition~\ref{lower walls}, as we now explain.
The Unique minimum condition follows because the identity element is the unique minimal element of the weak order, and thus every non-identity element has at least one cover reflection.
The Full edge condition follows because, if $\sgn(C(v,e))=-1$ then according to Proposition~\ref{lower walls}, the root $C(v,e)$ is associated to a cover $v'\covered v$, and by construction, $(v', v)$ is the edge $e$.
The Descending chain condition follows for the same reason:  every arrow $v\to v'$ corresponds to a cover $v'\covered v$.
In particular, $\ell(v')<\ell(v)$.
We have now completed the proof of Theorem~\ref{camb frame}.

\subsection{Denominator vectors}\label{denom sec}
We now comment on the problem of determining denominator vectors within a Cambrian framework.
As of now, we only have a direct way of determining denominators in the case where $\Cart(B)$ is of finite type, where we rely on results of \cite{sortable} and \cite{camb_fan}.
However, we conjecture that the same method works in arbitrary Cambrian frameworks.

In \cite[Section~8]{sortable}, a map $\cl_c$ was defined, taking a $c$-sortable element to an $n$-tuple of roots.
Here we give the same definition, modifying the notation slightly to allow us to reference individual roots in the $n$-tuple.
Suppose $v\in W$ is $c$-sortable and let $a_1\cdots a_k$ be its $c$-sorting word.
Let $r\in S$.
If $r$ does not occur as a letter in $a_1\cdots a_k$, then define $\cl_c^r(v)=-\alpha_r$.
If $r$ occurs in $a_1 \cdots a_k$, then let $i$ be the largest index such that $a_i=r$. The \newword{last reflection} for $r$ in $v$ is $a_1\cdots a_i\cdots a_1$.
In this case, we define $\cl_c^r(v)$ to be the positive root associated to the last reflection for $r$, that is to say, $\cl_c^r(v)=a_1\cdots a_{i-1}\alpha_r$.
Write $\cl_c(v)$ for $\set{\cl_c^r(v):r\in S}$.

As a consequence of \cite[Theorem~1.9]{ca2}, when $\Cart(B)$ is of finite type, the denominator vectors of cluster variables are all distinct.  
Thus in particular, the seeds in the exchange graph can be specified by the $n$-tuple of denominator vectors of the cluster variables in the seed.
This $n$-tuple of denominator vectors, realized as roots as in Section~\ref{ca background sec}, form a \newword{combinatorial cluster}.
The roots in the combinatorial cluster are all almost positive (see Section~\ref{ref frame subsec}). The map from cluster variables to almost positive roots is a bijection.
When $W$ is finite, the map $\cl_c$ is a bijection from $c$-sortable elements to combinatorial clusters \cite[Theorem~8.1]{sortable}.  
The Cambrian fan $\F_c$, in the finite case, is a complete fan.
On the other hand, the nonnegative linear span of each combinatorial cluster is a distinct $n$-dimensional simplicial cone, and these cones are the maximal cones of a complete simplicial fan \cite[Theorem~1.10]{ga}.
The map taking $\Cone_c(v)$ to the nonnegative span of $\cl_c(v)$ is a combinatorial isomorphism of fans \cite[Theorem~1.1]{camb_fan}.
We now prove a more precise statement about denominator vectors.

\begin{theorem}\label{camb denom}
Suppose $\Cart(B)$ is of finite type.
Let $(\Camb_c,C_c)$ be the Cambrian framework.
If $(v,e)$ is an incident pair in the graph $\Camb_c$ and $x^v_e$ is the cluster variable assigned to $(v,e)$, then $\d(x^v_e)$ is the root $\cl_c^r(v)$, where $r$ is the element of $S$ such that $C(v,e)=C_c^r(v)$.
\end{theorem}

We now prepare to prove Theorem~\ref{camb denom}.
First, we will need a lemma, which is immediate from the definitions, and which is a slightly more detailed version of \cite[Lemma~8.5]{sortable}.
The lemma refers to a map $\sigma_s$, for $s\in S$.  
This is an involution on almost positive roots defined by
\[\sigma_s(\alpha)=\left\lbrace\begin{array}{ll}
\alpha&\mbox{if }\alpha\in(-\Pi)\mbox{ and }\alpha\neq-\alpha_s,\mbox{ or}\\
s(\alpha)&\mbox{otherwise.}
\end{array}\right.\]

\begin{lemma}\label{cl s}
Let~$s$ be initial in~$c$, let~$v$ be $c$-sortable and let $r\in S$.
If $v\not\ge s$ then 
\[\cl_c^r(v)=\left\lbrace\begin{array}{ll}
-\alpha_s&\mbox{if }r=s,\mbox{ or}\\
\cl_{sc}^r(v)&\mbox{if }r\neq s
\end{array}\right.\]
If $v\ge s$ then $\cl_c^r(v)=\sigma_s(\cl^r_{scs}(sv))$.
\end{lemma}

Write $\nu_c$ for the linear map $\nu$ defined in Section~\ref{conj sec}.
That is, for a simple root $\alpha_r$, let 
\[\nu_c(\alpha_r)=-\sum_{s\in S}E_c(\alpha_s\ck,\alpha_r)\rho_s.\]

Let $R_c^r(v)$ denote the element dual to $(C_c^r)\ck(v)$ in the dual basis to $C_c\ck(v)$.
\begin{prop}\label{cl linear}
Let $v$ be a $c$-sortable element and let $r\in S$.
If $v\not\in W_\br{r}$ then $R_c^r(v)=\nu_c(\cl_c^r(v))$.
\end{prop}
We prove this result for arbitrary acyclic $B$, not only in the special case where $\Cart(B)$ is of finite type.
\begin{proof}
Let $s$ be initial in $c$.
If $v\not\ge s$ then $v\in W_\br{s}$.
Since $v\not\in W_\br{r}$, we conclude that $r\neq s$.
By induction on the rank of $W$, the vector $R^r_{sc}(v)$ in $V_\br{s}$ equals $\nu_{sc}(\cl_{sc}^r(v))$, which equals $\nu_{sc}(\cl_c^r(v))$ by Lemma~\ref{cl s}. 
The coefficient of $\rho_s$ in $\nu_c(\cl_c^r(v))$ is zero by Lemma~\ref{Ec initial}, since $\cl_c^r(v)$ is a root associated to a reflection in $W_\br{s}$.

Recall that $V_{\br{s}}$ is the subspace of $V$ spanned by the simple roots indexed by $\br{s}=S\setminus\set{s}$.
We identify the dual space $V_{\br{s}}^*$ naturally with the subspace of $V^*$ spanned by the fundamental weights indexed by $\br{s}$.
This subspace is the hyperplane $\alpha_s^\perp$.
Under this identification, the natural pairing $V_{\br{s}}^*\times  V_{\br{s}}\to\reals$ is the restriction of the natural pairing on $V^*$ and $V$.
Thus, $R^r_c(v)$ equals $\nu_{sc}(\cl_c^r(v))=\nu_{c}(\cl_{c}^r(v))$, and we are done in this case.

If $v\ge s$ then, by the recursive definition of $C_c\ck$, the vector $R_c^r(v)$ is $s\cdot R_{scs}^r(sv)$, where $s$ acts on $V^*$ by the action dual to its action on $V$.
This action preserves $\rho_r$ for $r\neq s$ and sends $\rho_s$ to $-\rho_s-\sum_{q\in\br{s}}K(\alpha\ck_q,\alpha_s)\rho_q$.

We must treat separately the case where $sv\in W_\br{r}$.
Since $v\not\in W_\br{r}$, we have $r=s$ in this case.
Furthermore, $\cl_c^s(v)=\alpha_s$, so $\nu_c(\cl_c^s(v))$ is $-\sum_{q\in S}E_c(\alpha_q\ck,\alpha_s)\rho_q$.
But applying Lemmas~\ref{sym antisym} and~\ref{Ec initial}, $\nu_c(\cl_c^s(v))$ is $-\rho_s-\sum_{q\in\br{s}}K(\alpha_q\ck,\alpha_s)\rho_q$.
On the other hand, $R_{scs}^s(sv)=\rho_s$, so 
\[R_c^s(v)=s\rho_s=-\rho_s-\sum_{q\in\br{s}}K(\alpha_q\ck,\alpha_s)\rho_q.\]

Finally, if $sv\not\in W_\br{r}$, then in particular $s$ is not the last reflection for $r$ in $v$, so $\cl_{scs}^r(sv)$ is a positive root and $\cl_c^r(v)=s\cdot\cl_{scs}^r(sv)$ by Lemma~\ref{cl s}.
By induction on $\ell(v)$, the vector $R_{scs}^r(sv)$ is $\nu_{scs}(\cl_{scs}^r(sv))$, which, by Lemmas~\ref{Ec invariant} and~\ref{cl s} equals 
\[\nu_{scs}(s\cdot\cl_c^r(v))=-\sum_{q\in S}E_{scs}(\alpha_q\ck,s\cdot\cl_c^r(v))\rho_q=-\sum_{q\in S}E_c(s\alpha_q\ck,\cl_c^r(v))\rho_q.\]
This can be rewritten as $-\sum_{q\in S}E_c(\alpha\ck_q-K(\alpha_q\ck,\alpha_s)\alpha_s\ck,\cl_c^r(v))\rho_q$ and then as
\begin{eqnarray*}
&&=E_c(\alpha_s\ck,\cl_c^r(v))\rho_s-\sum_{q\in \br{s}}E_c(\alpha\ck_q-K(\alpha_q\ck,\alpha_s)\alpha_s\ck,\cl_c^r(v))\rho_q\\
&&=-E_c(\alpha_s\ck,\cl_c^r(v))\left(-\rho_s-\sum_{q\in \br{s}}K(\alpha_q\ck,\alpha_s)\rho_q\right)-\sum_{q\in \br{s}}E_c(\alpha\ck_q,\cl_c^r(v))\rho_q.
\end{eqnarray*}
This is $-E_c(\alpha_s\ck,\cl_c^r(v))(s\cdot\rho_s)-\sum_{q\in \br{s}}E_c(\alpha\ck_q,\cl_c^r(v))(s\cdot\rho_q)$, which simplifies to $s\cdot\left(-\sum_{q\in S}E_c(\alpha\ck_q,\alpha_s)\rho_q\right)=s\cdot\nu_c(\cl_c^r(v))$.
Thus $R_c^r(v)=\nu_c(\cl_c^r(v))$ in this case as well.
\end{proof}

We can now prove the theorem.
\begin{proof}[Proof of Theorem~\ref{camb denom}]
The fact that the map taking $\Cone_c(v)$ to the nonnegative span of $\cl_c(v)$ is a combinatorial isomorphism of fans can be rephrased as follows:  There exists a bijection from rays in the $c$-Cambrian fan to almost positive roots such that a set of rays spans a cone in the $c$-Cambrian fan if and only if the corresponding set of almost positive roots spans a cone in the fan of combinatorial clusters.
(Indeed, this rephrasing is closer to the way that \cite[Theorem~1.1]{camb_fan} was proved.)
The results of this paper give another such bijection: each ray in the Cambrian fan contains a unique vector $R(v,e)$ that is the $\g$-vector of $x^v_e$.
This ray is mapped to the denominator vector of $x_v^e$.

We first verify that the two bijections are the same.
The identity is a $c$-sortable element, and the rays of its $c$-Cambrian cone are spanned by the vectors $\rho_i:i\in I$, which are the $\g$-vectors of $x_i:i\in I$.
The map $\cl_c$, applied to the identity element, returns the set of negative simple roots, which are the denominator vectors of $x_i:i\in I$.
For each $i\in I$, the element $s_i$ is a $c$-sortable element adjacent to the identity in $\Camb_c$.
The cluster variable $x_i$ is the one removed in the mutation from the initial seed to $\Seed(s_i)$, and the negative simple root $-\alpha_i$ is the root removed from $\cl_c$ when moving from the identity element to $s_i$.
Thus both bijections map the ray spanned by $\rho_i$ to the negative simple root $-\alpha_i$.
It is now easy to see that the two bijections coincide.
(For example, one can show by induction on $\ell(v)$ that the rays of $\Cone_c(v)$ are treated the same by both bijections, using the fact that each new $v$ adds at most one new ray.)

It now remains to check the statement about the element $r$.
Suppose $v$ and $v'$ are connected by an edge in $\Camb_c$ and suppose that $\cl_c^r(v)$ is the denominator vector of the cluster variable that is removed in passing from the combinatorial cluster $\cl_c(v)$ to the combinatorial cluster $\cl_c(v')$.
Proposition~\ref{cl linear} implies that every other element $\beta$ of $\cl_c(v)$ is mapped by $\nu_c$ into $C_c^r(v)$ and thus that $\beta\in\cl_c(v')$.
We conclude that $\cl_c^r(v)$ is the denominator vector of the cluster variable that is removed in passing from $\Seed(v)$ to $\Seed(v')$.
This cluster variable is $x^v_e$.
\end{proof}

The proof of Theorem~\ref{camb denom} via Proposition~\ref{cl linear} is the motivation for Conjecture~\ref{nu conj}.
The proof also establishes the following result:
\begin{theorem}\label{nu conj camb}
Conjecture~\ref{nu conj} holds when $\Cart(B)$ is of finite type.
\end{theorem}

Several conjectures follow from Conjecture~\ref{nu conj} and Conjectures~\ref{g lattice conj}, \ref{g fan conj}, and~\ref{strong g fan conj}, as explained in Section~\ref{conj sec}.
\begin{cor}\label{camb denom cor}
Conjectures~\ref{denom lattice conj}, \ref{denom fan conj} and~\ref{strong denom fan conj} hold for $\Cart(B)$ of finite type and principal coefficients.
\end{cor}

As a consequence of Theorem~\ref{nu conj camb}, the combinatorial clusters model contains complete information about $\g$-vectors, and therefore, in light of Theorem~\ref{framework principal} and Corollary~\ref{identity model}, about exchange matrices and principal coefficients as well. 
(Earlier results of Yang and Zelevinsky~\cite{YZ} already show that the combinatorial clusters model contains complete information about $\g$-vectors.  
These results combined with the result of~\cite{NZ} stated here as Corollary~\ref{NZ cor} also show that the combinatorial clusters model contains complete information about principal coefficients.)

Conjecture~\ref{nu conj} implies the following conjecture about Cambrian frameworks.
\begin{conj}\label{camb denom conj}
Theorem~\ref{camb denom} holds even when $\Cart(B)$ is not of finite type.
\end{conj}
Indeed, Conjecture~\ref{camb denom conj} follows from Conjecture~\ref{nu conj} because Proposition~\ref{cl linear} identifies $\cl_c^r(v)$ as $\eta_c(R_c^r(v))$ and Theorem~\ref{framework principal}(\ref{g vec}) identifies $R_c^r(v)$ as the $\g$-vector of the appropriate cluster variable.

With Conjecture~\ref{camb denom conj} unproven, we have no direct way, when $\Cart(B)$ is not of finite type, of reading off denominator vectors from $c$-sortable elements.
However, Proposition~\ref{cl linear} is still useful in general, in that it provides a way to read off $\g$-vectors from $c$-sortable elements without computing a dual basis.
Combining Proposition~\ref{cl linear} with Theorem~\ref{framework principal}(\ref{g vec}), we obtain the following theorem.
\begin{theorem}\label{camb nu g}
Let $(\Camb_c,C_c)$ be a Cambrian reflection framework. 
Suppose $(v,e)$ is an incident pair in the graph $\Camb_c$ and suppose that the cluster variable $x^v_e$ assigned to $(v,e)$ is not contained in the initial cluster.
Then $\g(x^v_e)$ is the weight $\nu_c\cl_c^r(v)$, where $r$ is the element of $S$ such that $C(v,e)=C_c^r(v)$.
\end{theorem}

\newpage

\section{Summary of framework terminology}  \label{sec:defns}
In the following table, we describe the key terms related to frameworks,  and indicate the page where their definitions can be found in the paper.  

\begin{longtable}{|p{105pt}|p{210pt}|c|}\hline
Term&Definition&Page\\\hline&&\\[-11pt]\hline
\raggedright\newword{complete}&a framework is complete if its underlying quasi-graph has no half-edges&\pageref{complete}\\\hline
\raggedright\newword{Cartan companion}, denoted $\boldsymbol{\Cart(B)}$ & the matrix whose diagonal elements are $2$ and whose off diagonal elements are $-|B_{ij}|$&\pageref{defn:Cart}\\\hline
\raggedright\newword{quasi-graph}&a graph $G$ with edges and half-edges&\pageref{quasi-graph}\\\hline
\raggedright\newword{incident pair}&$(v,e)$ for $v$ a vertex and $e$ an edge incident to~$v$&\pageref{incident pair}\\\hline
\raggedright\newword{label}, \raggedright\newword{co-label}&vectors $C(v,e)$ and $C\ck(v,e)$ assigned to $(v,e)$&\pageref{label}\\\hline
\raggedright\newword{Co-label condition}&$C\ck(v,e)$ is a positive multiple of $C(v,e)$&\pageref{Co-label condition}\\\hline
\raggedright\newword{Sign condition}&labels have a well-defined sign&\pageref{Sign condition}\\\hline
\raggedright\newword{Base condition}&$G$ has a vertex labeled by simple roots and co-labeled by simple co-roots&\pageref{Base condition}\\\hline
\raggedright\newword{(Co)-transition condition}&relates (co)-labels on two adjacent vertices&\pageref{Transition condition}, \pageref{Transition condition, strengthened}\\\hline
\raggedright\newword{Framework}&a labeled connected quasi-graph satisfying the Co-label, Sign, Base, Transition, and (Co)-transition conditions&\pageref{framework}\\\hline
\raggedright\newword{Euler form}&a bilinear form associated to an exchange matrix&\pageref{Euler form}\\\hline
\raggedright\newword{Root condition}& the requirement that labels be roots&\pageref{Root condition}\\\hline
\raggedright\newword{Euler conditions}& conditions on the Euler form applied to labels at a given vertex&\pageref{Euler conditions}\\\hline
\raggedright\newword{Reflection condition}& a version of the Transition condition where some transitions are given by reflections&\pageref{Reflection condition}\\\hline
\raggedright\newword{injective}&a framework is injective if each vertex is determined uniquely by its label set &\pageref{injective}\\\hline
\raggedright\newword{ample}&a framework is ample if there is a natural map from $G$ to the exchange graph.&\pageref{ample}\\\hline
\raggedright\newword{exact}&injective and ample&\pageref{exact}\\\hline
\raggedright\newword{polyhedral}&a framework is polyhedral if its cones are the maximal faces of a fan&\pageref{polyhedral}\\\hline
\raggedright\newword{well-connected}&a polyhedral framework is well-connected if it satisfies a certain local connectivity condition&\pageref{well-connected}\\\hline
\raggedright\newword{rank-two cycle}& a certain type of cycle in $G$ & \pageref{rank two cycle}\\\hline
\raggedright\newword{simply connected}&simple connectivity is a topological condition that implies ampleness&\pageref{simply connected}\\\hline
\raggedright\newword{Unique minimum}  &the only vertex with all labels positive is $v_b$&\pageref{Unique minimum condition}\\\hline
\raggedright\newword{Full edge condition}&negative labels may only occur on full edges
&\pageref{Full edge condition}\\\hline
\raggedright\newword{Descending chain condition}&there is no infinite path in the graph that ``descends'' in the sense of signs on labels&\pageref{Descending chain condition}\\\hline
\raggedright\newword{descending}&satisfying the Unique minimum, Full edge, and Descending chain conditions &\pageref{descending}\\\hline
\raggedright\newword{Cambrian framework}&a framework defined in terms of sortable elements and Cambrian semilattices&\pageref{camb frame sec}--\pageref{end Cambrian framework}\\\hline
\end{longtable}

\end{document}